\pdfoutput=1
\documentclass[final,3p,times]{elsarticle}
\usepackage[utf8]{inputenc}
\usepackage[english]{babel}
\usepackage{blindtext}
\usepackage{amssymb}
\usepackage[numbers]{natbib}
\usepackage{fullpage,amsfonts,nicefrac,amsmath,mathtools,romannum}
\usepackage[shortlabels]{enumitem}
\usepackage{tikz}
\usepackage[cjk]{kotex}
\usepackage{color}
\usepackage{stackengine}
\usepackage{amsthm}
\usepackage{algorithmic, algorithm}
\usepackage{mathtools}
\usepackage{caption}
\usepackage{subcaption}
\usepackage{afterpage}

\usepackage{hyperref}
\usepackage[nameinlink]{cleveref}
\usepackage{appendix}
\usepackage{multirow}
\usepackage{booktabs}

\usepackage[normalem]{ulem}
\usepackage{titlesec}

\usepackage{relsize} 
\newcommand{\righthalfcup}{\mathbin{\mathlarger{\lrcorner}}}

\newtheorem{theorem}{Theorem}[section]
\newtheorem{corollary}{Corollary}[theorem]
\newtheorem{lemma}[theorem]{Lemma}
\newtheorem{proposition}[theorem]{Proposition}

\newtheorem{definition}[theorem]{Definition}

\newtheorem{example}[theorem]{Example}
\newtheorem{remark}[theorem]{Remark}
\newtheorem{construction}[theorem]{Construction}

\theoremstyle{definition}
\crefname{thm}{Thm}{Thms}
\newcounter{examplecounter}[section]
\renewcommand{\theexamplecounter}{\arabic{section}.\arabic{examplecounter}}

\newcommand{\Rip}{\mathcal{R}}
\newcommand{\bcd}{\mathsf{bcd}}

\journal{Advances in Applied Mathematics}

\begin{document}

\begin{frontmatter}
\author[label1]{Keunsu Kim}
\ead{keunsu@postech.ac.kr}

\author[label2,label3]{Jae-Hun Jung\corref{cor1}}
\ead{jung153@postech.ac.kr}
\cortext[cor1]{Corresponding author.}

\affiliation[label1]{organization={Institute of Mathematics for Industry, Kyushu University},
            city={Fukuoka},
            postcode={819-0395},
            country={Japan}}

\affiliation[label2]{organization={Department of Mathematics, Pohang University of Science \& Technology},
            city={Pohang},
            postcode={37673},
            country={Korea}}

\affiliation[label3]{organization={POSTECH Mathematical Institute for Data Science (MINDS)},
            city={Pohang},
            postcode={37673},
            country={Korea}}

\title{Exact multi-parameter persistent homology of time-series data: Fast and variable topological inferences}


\pagenumbering{arabic}

\begin{abstract}
We propose the Exact Multi-Parameter Persistent Homology (EMPH) method for the topological analysis of time-series data based on the Liouville torus. Assuming, as in Takens' embedding, that a time-series represents observations of an underlying dynamical system, we model the system as a Hamiltonian system of uncoupled one-dimensional harmonic oscillators. Under this setting, the Liouville torus arises naturally as a dynamical object, and the persistent homology of the Vietoris–Rips complex built on this torus can be interpreted through Fourier analysis. EMPH constructs a multi-parameter filtration framework using Fourier decomposition and provides a closed-form expression for the fibered barcode, an invariant obtained by restricting multi-parameter persistent homology along a specific ray. This formulation establishes a direct correspondence between the choice of a ray and the weighting of Fourier modes, enabling variable topological inferences by exploring different rays in the filtration space. Compared with conventional sliding window based analysis of time-series data, which is computationally expensive, EMPH yields exact barcode formulas with the symmetry of the Liouville torus, achieving much lower computational cost while maintaining comparable or superior accuracy. Thus, EMPH offers both computational efficiency and interpretive flexibility, bridging Fourier analysis and multi-parameter persistent homology in time-series data analysis.
\end{abstract}


\end{frontmatter}

\begin{section}{Introduction}

Topological Data Analysis (TDA) is a recent development in modern data science that utilizes the topological features of the given data. Contrary to traditional approaches such as the statistical methods, TDA rather tries to understand the given data by revealing the  topological and geometrical structures of the data.  To extract topological features from the given data, we consider so-called the filtered simplicial complex and record the change of its homology with  scale. That is, instead of fixing the scale for the construction of the complex out of the given data points, TDA measures the homological invariants to each scale. This way, we see how the topological properties of the given data evolve with scale. As the changes with respect to scale are  summarized through TDA, the local and global structures of the given data can be concisely visualized and used for characterizing the given data.

Persistent homology, more generally referred to as a persistence module, concisely summarizes the evolution of topological features across scales. Its output is typically visualized as a barcode or, equivalently, a persistence diagram. The map that transforms data into a barcode is piecewise-defined and generally non-differentiable \cite{leygonie2022framework}. Even a slight perturbation of the input data can reach boundary conditions (e.g., equal edge lengths in a Vietoris–Rips complex), where insertion orders and pairings change, thereby altering the very expression of this transformation. Consequently, no closed-form formula expressed as a combination of elementary functions exists. Only in very special settings, such as a circle \cite{adamaszek2017vietoris} or a regular polygon whose number of sides is a multiple of six (Theorem \ref{6k barcode}), can restricted exact formulas be derived.

Data are, in general, defined with several parameters, so one-parameter filtration may be insufficient to analyze the structure of data \cite{carlsson2009theory}. For this reason multi-parameter persistent homology theory seems necessary and researchers have tried to develop its full theory. For one-parameter filtration, pointwise finite dimensional persistent module can be uniquely decomposed by half-open intervals, and the barcode is a complete invariant in persistence module category \cite{botnan2020decomposition}. In contrast, for  multi-parameter persistent homology, it is highly complicated to define complete invariant \cite{carlsson2009theory, lesnick2015theory, botnan2022introduction}. To resolve this problem, we may relax the condition of the completeness. Although rank invariant is not a complete invariant for  multi-parameter persistent homology, it can capture a persistence of homological class  as a practical invariant and is equivalent to barcode in one-parameter filtration \cite{carlsson2009theory}. Rank invariant is known to be equivalent to fibered barcode that is a collection of one-dimensional reduction of multi-parameter persistent homology  \cite{cerri2013betti, botnan2022introduction}. With fibered barcode, one could use the considered vectorization of (incomplete) multi-parameter persistent homology, e.g. multi-parameter persistence kernel \cite{corbet2019kernel}, multi-parameter persistence landscape \cite{vipond2020multiparameter} and multi-parameter persistence image \cite{carriere2020multiparameter}. In \cite{carriere2020multiparameter}, 
using two image data sets, i.e. the intensity images of immune cells and cancer cells, two-parameter sublevel filtration was constructed to predict the survival rate of the breast cancer patients. It was shown, for this example,  that multi-parameter persistence theory helps us to capture the interaction patterns of multiple phenotypes at once.

A common TDA method for analyzing time-series data involves translating the data into a point cloud using sliding window embedding, as proposed in \cite{skraba2012topological} and theoretically studied in \cite{perea2015sliding}. In \cite{perea2015sliding}, the authors provide several properties of sliding window embedding: (i) Sliding window embedding translates a trigonometric polynomial into a closed curve on an $N$-torus, where $N$ represents the degree of the trigonometric polynomial used. Specifically, it converts a sinusoidal function into an elliptic curve (planar curve).  (ii) With respect to the bottleneck distance (the standard metric on barcodes), the barcode of the Vietoris–Rips complex of the sliding window embedding of a periodic time-series is approximated by that of its truncated Fourier series (see Theorem \ref{truncated approximation}). And (iii) the minimum embedding dimension required to preserve geometric information is demonstrated. Based on these results, the authors propose the periodicity score, a metric for measuring the periodicity of the given time-series data. This approach includes the following processes and utilizes the information from the barcode:

\[
\begin{array}{ccccc}
f & \longrightarrow & \Rip_{\epsilon}(SW_{M,\tau}f) & \longrightarrow & \mathsf{bcd}_{n}^{\Rip}(SW_{M,\tau}f) \\
\text{(Time-series data)} &&  \text{(Vietoris-Rips complex of the sliding window embedding)} && \text{(Barcode)}
\end{array}
\]

The Liouville torus is an object in a complete integrable Hamiltonian system. Loosely speaking, a complete integrable Hamiltonian system is a Hamiltonian system that has as many independent invariants as possible. In such a system, a particle’s trajectory should be confined to an $n$-torus, where $n$ is the number of maximal independent invariants and such a torus is known as the Liouville torus. To provide a rationale for analyzing the Liouville torus in TDA, we review Takens' embedding theorem in Section \ref{Continuosization section}.

Our main idea is to transform time-series data into a barcode, through the Liouville torus without utilizing sliding window embedding. Given time-series data $f$, denote its Liouville torus as $\Psi_f$ (Definition \ref{def:Liouville torus}). We view $\Psi_f$ as a product of $N$ circles (Corollary~\ref{cor:Liouville isomorphic}) and write $\pi_i$ for the $i$-th coordinate projection onto the $i$-th circle. We then proceed as follows, in analogy with the sliding window embedding:
\[
\begin{array}{ccccc}
f & \longrightarrow & \Rip_{\epsilon}(\Psi_f) =  \Rip_{\epsilon}(\pi_1 \Psi_f) \times \cdots \times \Rip_{\epsilon}(\pi_N \Psi_f) & \longrightarrow & \mathsf{bcd}_{n}^{\Rip}(\Psi_{f}) \\
\text{(Time-series data)} &&  \text{(Vietoris-Rips complex of the Liouville torus)} && \text{(Barcode)}
\end{array}
\]
\noindent

Furthermore, in Section~\ref{Sec : Application of multi-parameter theory and its interpretation}, 
we extend the one-parameter filtration 
$\Rip_{\epsilon}(\pi_1 \Psi_f)\times \cdots \times \Rip_{\epsilon}(\pi_N \Psi_f)$ 
to an $N$-parameter filtration 
$\Rip_{\epsilon_1}(\pi_1 \Psi_f)\times \cdots \times \Rip_{\epsilon_N}(\pi_N \Psi_f)$. 
Instead of working on the entire parameter space $\mathbb{R}^N$, 
we fix a ray $\ell \subset \mathbb{R}^N$ and consider the one-parameter filtration induced along $\ell$, 
for which the barcode $\mathsf{bcd}_{n}^{\Rip, \ell}(\Psi_{f})$ is well-defined. 
This procedure can be summarized as follows:
\[
\begin{array}{ccccc}
f & \longrightarrow & 
\Rip_{\epsilon_1}(\pi_1 \Psi_f) \times \cdots \times \Rip_{\epsilon_N}(\pi_N \Psi_f) 
& \longrightarrow & 
\mathsf{bcd}_{n}^{\Rip, \ell}(\Psi_{f}) \\
\text{(Time-series data)} && \text{(Multi-parameter filtration of the Liouville torus)} && \text{(Barcode)}
\end{array}
\]

From this approach, we can naturally raise the following questions:
\begin{enumerate}
\item What information from time-series data is encoded in the barcode?
\item What are the benefits of this approach?
\end{enumerate}

The main results of our study presented in this paper are the followings:

\begin{enumerate}
    \item The exact formula of the barcode of the Liouville torus $\Psi_f$, which contains the sliding window embedding of the given time-series data $f$ is obtainable and interpretable. The results are provided in Section \ref{sec : Exact formula and interpretation of barcode}. 
    
    \begin{enumerate}
    \item Sliding window embedding of periodic time-series data can be formulated by the trajectory of uncoupled one-dimensional harmonic oscillators. (Theorem \ref{thm:sliding comparison})
   \item The barcode of $\Psi_f$ is given by the following formula $$\mathsf{bcd}_{n}^{\Rip}(\Psi_{f}) = 	\left\{ J_1^{n_1} \bigcap \cdots \bigcap J_N^{n_N} : J_L^{n_{L}} \in \mathsf{bcd}_{n_L}^{\Rip}	\left(\pi_{L} \Psi_{f}\right) \ \text{and} \ \sum\limits_{L=1}^{N} n_L = n \right\},$$ \\
    i.e. $J_{L}^{n} = \begin{cases} (0,\infty), & \mbox{if }n=0 \\ \left(2r_L^f \sin\left(\pi {k \over 2k+1}\right), 2r_L^f \sin\left(\pi {k+1 \over 2k+3}\right)\right], & \mbox{if }n=2k+1 \\ \ \emptyset, & \mbox{otherwise}
\end{cases}$.  (Theorem \ref{cor:barcode})
    
 \item Each bar in $\mathsf{bcd}_{n}^{\Rip}(\Psi_{f})$ corresponds to a bar of the projected point cloud onto 
$P_{i_1} \oplus \cdots \oplus P_{i_k}$ for $k = 1, \dots, n$, where each $P_L$ is the two-dimensional subspace onto which the sinusoidal functions of frequency $L$ are mapped under the sliding window embedding, and $\sum\limits_{\substack{L=1 \\ n_{i_L} \in \mathbb{N}}}^{k} n_{i_L} = n$. 
That is,
\(
\mathsf{bcd}_{n}^{\Rip}(\Psi_{f}) 
= 
\bigcup\limits_{1 \le i_1 < \cdots < i_k \le N} 
\bigcup\limits_{1 \le k \le n} 
\mathsf{bcd}_{n}^{\Rip}\left( \pi_{i_1 \cdots i_k} \Psi_{f} \right).
\)
(Theorem~\ref{barcode meaning})

    \item 
In machine learning frameworks, combinatorial properties are often needed. For example, Deep Sets \cite{zaheer2017deep} and RipsNet \cite{de2022ripsnet} are such examples. Since a barcode is a combinatoric object, we can provide a combinatorial perspective on time-series data from the barcode (Proposition \ref{prop:permutation symmetric}).
    
    \end{enumerate} 
    
    \item We propose a multi-parameter persistent homology method based on the filtration with Fourier bases with the exact barcode (Section \ref{Sec : Application of multi-parameter theory and its interpretation}).
    The Fourier bases constitute the multi-parameter filtration space. The exact barcode to each Fourier mode is precomputed and the actual barcode is then calculated with the Fourier coefficient of the corresponding Fourier mode. 

    \begin{enumerate}
    \item If a ray $\ell$ in the filtration space has the direction vector $\mathbf{a} = (a_{1},\cdots , a_{N})$ with each component $a_i > 0$, and the endpoint $\mathbf{b} = (b_1, \cdots , b_N)$, then the barcode is given by the following: $$\mathsf{bcd}_{n}^{\Rip,\ell}(\Psi_{f}) = \left\{ J_{1}^{n_1, \ell}\bigcap \cdots \bigcap J_{N}^{n_N, \ell} :  J_{L}^{n_L, \ell} \in  \mathsf{bcd}_{n_L}^{\Rip,\ell} 	\left( \pi_L \Psi_{f}\right)  \ \text{and} \ \sum\limits_{L=1}^{N} n_L = n \right\},$$ i.e. \[
J_{L}^{n, \ell} =
\begin{cases}
\left(
    \dfrac{-b_L}{\sqrt{N}\,a_L / \lVert \mathbf{a} \rVert}, \infty
\right), & \text{if } n = 0, \\[2ex]
\left(
    \dfrac{2r_L^f \sin\!\left(\pi \tfrac{k}{2k+1}\right) - b_L}
          {\sqrt{N}\,a_L / \lVert \mathbf{a} \rVert},\ 
    \dfrac{2r_L^f \sin\!\left(\pi \tfrac{k+1}{2k+3}\right) - b_L}
          {\sqrt{N}\,a_L / \lVert \mathbf{a} \rVert}
\right], & \text{if } n = 2k+1, \\[3ex]
\emptyset, & \text{otherwise}. 
\end{cases} \qquad
\text{(Theorem \ref{direction barcode})} \]

The diagonal ray is the ray with $\mathbf{a}$ of $a_1 = a_2 = \cdots = a_N$ and $\mathbf{b}$ of $(0, 0, \cdots, 0)$.

    \item Usual persistent homology is equivalent to choosing  the  diagonal ray in the multi-parameter filtration space. That is,  $\mathsf{bcd}_{n}^{\Rip,\ell}(\Psi_{f}) = \mathsf{bcd}_{n}^{\Rip}(\Psi_{f})$. (Corollary \ref{general one-parameter})

    \item {\color{black}One of the advantages of considering $\mathsf{bcd}_{n}^{\Rip,\ell}(\Psi_{f})$ is that it can provide  variable topological viewpoints. (Example \ref{ex:benefit} and Example \ref{example 2})}

    \item We construct our method on a collection of rays in multi-parameter space. With this consideration, we derive Theorem \ref{thm:collection of ray}, that is a counterpart of Example \ref{onedimension persistent}.
    
    \end{enumerate}
    \item The proposed method has several practical advantages. First, the computational complexity of our proposed method is very low. Further, variable topological inferences are possible with low computational cost in a machine learning workflow such as for the classification and clustering problems. (Section \ref{Experiment})
\end{enumerate}

The Liouville torus has more symmetry compared to sliding window embedding. It enables us to obtain the exact formula of the barcode for the given time-series data and interpret the results more intuitively. As a result, we can understand the meaning of the barcode obtained by the Liouville torus and interpret its relation with the given time-series data. Exploiting the advantages of the Liouville torus, we utilize the rank invariant of multi-parameter persistent homology, which is equivalent to restricting the multi-parameter persistent homology to a ray in the multi-parameter space.

Our method is comparable with the usual sliding window embedding method and the computational complexity is very low. For the time-series data with the length of $T$\label{not:length}, the computational complexity of calculating the barcodes of Vietoris-Rips complex through sliding window embedding is known as $O(T^{3n+3})$ \cite{chung2019brain} where $n$ is the dimension of the barcode. For the exact barcode,  however, the computational complexity of $\mathsf{bcd}_{n}^{\Rip,\ell}(\Psi_{f})$ is $O(T\log T) + O\left(N \times \binom{N+n-1}{n}\right)$, where $N(\le T)$ is the degree of the truncated trigonometric polynomial (Remark \ref{remark:computational complexity}). Due to the very low computational complexity of the proposed method, various rays can be tested almost simultaneously and variable inferences are obtainable, highly efficient when implemented in a machine learning workflow. 

This paper is composed of the following sections. In Section \ref{sec:Defintions} we provide all the definitions necessary for the analysis presented in this paper. Also, previous results that the current paper relies on are presented. \textcolor{black}{In Section \ref{sec : Exact formula and interpretation of barcode}, \textcolor{black}{we review Takens' embedding theorem and the motivation of introducing the Liouville torus. We, then, provide the exact barcode formula and its interpretation.} In Section \ref{Sec : Application of multi-parameter theory and its interpretation}, we construct multi-parameter persistent homology based on the Fourier decomposition and derive the exact barcode formula for the one-dimensional reduction of multi-parameter persistent homology, referred to as the Exact Multi-parameter Persistent Homology (EMPH).} In Section \ref{Experiment}, we present numerical examples for the classification and clustering problems. Particularly we compare several methods in terms of computational complexity and show that the proposed method is highly efficient. 
In Section \ref{conclusion}, we provide a concluding remark and future research subjects.
\end{section}

\begin{section}{Definitions and theorems}
\label{sec:Defintions}
We construct a filtration of simplicial complexes over a metric space. A popular method of such a filtration is the Vietoris-Rips complex and we are mainly interested in the Vietoris-Rips complex in this work. The following provides the definition of the Vietoris-Rips complex. 

\begin{definition}[Vietoris-Rips complex]
\label{Vietoris-Rips complex}
Let $(X,d_X)$ be a metric space. Vietoris-Rips complex $\Rip(X) = \left\{\Rip_{\epsilon}(X)\right\}_{\epsilon \in \mathbb{R}}$ is a one-parameter collection of simplicial complexes, where 
$\Rip_{\epsilon}(X) := \bigl\{ \left\{ x_{0}, \cdots, x_{n} \right\} \subset X : \bigr.$ $\bigl. \max\limits_{0 \le i,j \le n} d_{X}(x_{i},x_{j}) < \epsilon \bigr\}.$
\end{definition}
Persistent homology is a tool from TDA that tracks the birth and death of homological features across different scales in a filtration.
This information is summarized in a barcode, which records each homology class by its birth and death times.

\begin{definition}[Persistent homology and barcode]
\label{not:barcode}
Let $\mathcal{K} = \{\mathcal{K}_{\epsilon}\}_{\epsilon \in \mathbb{R}}$ be a one-parameter filtration of simplicial complexes, 
that is, $\mathcal{K}_{\epsilon} \subseteq \mathcal{K}_{\epsilon'}$ whenever $\epsilon \le \epsilon'$. 
Then the pair 
\[
\Bigl(\{H_{n}(\mathcal{K}_{\epsilon})\}_{\epsilon \in \mathbb{R}}, \;
\{\iota_{*}^{\epsilon, \epsilon'} : H_{n}(\mathcal{K}_{\epsilon}) 
\to H_{n}(\mathcal{K}_{\epsilon'})\}_{\epsilon \le \epsilon'}\Bigr)
\]
is called the $n$-dimensional persistent homology, 
where $\iota_{*}^{\epsilon, \epsilon'}$ is the map on homology induced by the inclusion 
$\iota^{\epsilon, \epsilon'}: \mathcal{K}_{\epsilon} \hookrightarrow \mathcal{K}_{\epsilon'}$. 
The $n$-dimensional barcode $\mathsf{bcd}_{n}(\mathcal{K})$ 
is a multiset of intervals $(b,d]$, called bars, where each bar represents a homology class born at $\epsilon=b$ 
and dying at $\epsilon=d$. In particular, when $\mathcal{K}=\Rip(X)$ is the Vietoris–Rips complex of a metric space $X$, 
we write the barcode as $\mathsf{bcd}_{n}^{\Rip}(X)$.
Note that we adopt the convention $(b,d]$ for bars, since in the Vietoris–Rips complex defined by $d(x,y)<\epsilon$, the resulting barcode naturally takes this form. Using $d(x,y)\le\epsilon$ yields bars of the form $[b,d)$.
\end{definition}

Note that persistent homology is uniquely represented by its barcode. 
Since the barcode characterizes the given data, the distance between two barcodes 
is used to measure their similarity. 
One standard choice is the bottleneck distance, defined as follows.

\begin{definition}[Bottleneck distance]
\label{bottleneck distance}
Let $\mathsf{bcd}_{n}(\mathcal{K}_1)$ and $\mathsf{bcd}_{n}(\mathcal{K}_2)$ be two barcodes.  
Since these barcodes may have different cardinalities, we extend each by adding the diagonal set 
$\Delta := \{(a,a] \mid a \in \mathbb{R}\}$ with infinite multiplicity.  
For $I=(b,d]$, $J=(b',d']$, and $\Delta$, set  
$d_\infty(I,J) := \max\{|b-b'|,\,|d-d'|\}$,  
$d_\infty(I,\Delta) := \tfrac{d-b}{2}$,  
$d_\infty(\Delta,J) := \tfrac{d'-b'}{2}$,  
$d_\infty(\Delta,\Delta) := 0$. The Bottleneck distance between the two barcodes is
\[
d_B\!\left(\mathsf{bcd}_{n}(\mathcal{K}_1),\, \mathsf{bcd}_{n}(\mathcal{K}_2)\right)
:= \inf_{\phi} \; \sup_{I \in \mathsf{bcd}_{n}(\mathcal{K}_1) \cup \Delta} 
             d_\infty\!\bigl(I,\phi(I)\bigr),
\]
where $\phi$ ranges over all bijections 
$\mathsf{bcd}_{n}(\mathcal{K}_1) \cup \Delta \to \mathsf{bcd}_{n}(\mathcal{K}_2) \cup \Delta$.
\end{definition}

The Vietoris-Rips complex can be used to infer a population manifold from a sampling point cloud. For example, if $M$ is a Riemannian manifold and $X$ is sufficiently close to $M$ in terms of Gromov-Hausdorff distance, then for a sufficiently small $\epsilon>0$, the Vietoris-Rips complex $\Rip_{\epsilon}(X)$ is homotopic to $M$ \cite{latschev2001vietoris}. However, in practice, since we often lack prior information about $M$, the Vietoris-Rips complex with a larger scale is frequently considered, leading to studies on the homotopy type of the Vietoris-Rips complex with a large scale of circle, ellipse, $n$-sphere, etc. \cite{adamaszek2017vietoris, adamaszek2019vietoris, lim2024vietoris}. In the case of the circle, the homotopy type of the Vietoris-Rips complex has been fully studied, but other cases have only been partially studied.

Theorem \ref{barcode of circle} is the result for the Vietoris-Rips complex of a unit circle equipped with the Euclidean metric (denoted by $\mathbb{S}^{1}$). In \cite{adamaszek2017vietoris}, cyclic graph $\vec{G}$ and its invariant winding fraction $wf(\vec{G})$ are introduced. It was proven that $wf\left(\Rip_{\epsilon}\left(S^{1}\right)\right) = \epsilon$ where $S^1$ is a circle equipped with arc-length metric whose circumference is $1$. Using the fact that the Vietoris-Rips complex is a clique complex and $wf\left(\Rip_{\epsilon}\left(S^{1}\right)\right) = \epsilon$, the authors of \cite{adamaszek2017vietoris}  applied the previous results \cite{adamaszek2013clique} of homotopy classification of clique complex and obtained the exact formula of the Vietoris-Rips complex of $S^1$. The barcode formula was also given  for the Vietoris-Rips complex of  $\mathbb{S}^{1}$ in Proposition 10.1 in \cite{adamaszek2017vietoris} via arc-length results. In this paper, we will deal with $\mathbb{S}^{1}$ rather than $S^1$. But we note that it is also meaningful to deal with $S^1$, even if $S^1$ is not isometric embedded into $\mathbb{R}^{M+1}$. As mentioned earlier in the Introduction,
we are more interested in the topological properties of data rather than the metric properties.

\begin{theorem}[Proposition 10.1, \cite{adamaszek2017vietoris}, Sec 6.2, \cite{gakhar2019k}]
\label{barcode of circle}
Let $\mathbb{S}^{1}$ be a unit circle equipped with the Euclidean metric. Then 

$\mathsf{bcd}_{n}^{\Rip}	\left( \mathbb{S}^{1} \right)=
\begin{cases}
\left\{ \left( 0, \infty \right) \right\}, & \mbox{if }n=0 \\
 \left\{ \left( 2 \sin\left(\pi{k \over 2k +1}\right), 2 \sin\left(\pi {k+1 \over 2k+3}\right) \right] \right\}, & \mbox{if } n=2k+1, k\in \mathbb{Z}_{\ge 0} \\
 \ \emptyset \ , & \mbox{otherwise}.
\end{cases}$
\end{theorem}

Sliding window embedding is a popular method for time-series data analysis using TDA. Through sliding window embedding a point cloud is formed and simplicial complex is constructed toward TDA. The following provides the definition of sliding window embedding. We refer the reader to \cite{perea2015sliding} for detailed explanation of the application of persistent homology to time-series data with sliding window embedding. 
\begin{definition}[\cite{perea2015sliding}, Sliding window embedding]
\label{slidingwindow}
Let $\mathbb{T} = \mathbb{R} \big/ 2\pi\mathbb{Z}$ and $f : \mathbb{T} \rightarrow \mathbb{R}$. Choose $M \in \mathbb{N}$ and $\tau \in \mathbb{R}$. Then sliding window embedding of $f$ is defined by 
$$SW_{M,\tau}f(t) = \begin{bmatrix} f(t) \\ f(t+\tau) \\ \vdots \\ f(t+M\tau) \end{bmatrix} \in \mathbb{R}^{M+1}.$$
\end{definition}
Sliding window embedding translates a sinusoidal function into an ellipse (planar curve). If we set $\tau = {2\pi \over M+1}$, then a sinusoidal function is translated into a circle. Here $M$ is a hyperparameter that determines the dimension of the embedding space. The value of $\tau$ is the sampling resolution of the given time-series data. The given data is represented as a point cloud in the embedding space of dimension $M+1$.

{\color{black} In Theorem \ref{barcode of circle}, we have the barcode formula for the Vietoris-Rips complex of a circle. By setting $\tau = {2\pi \over M+1}$, we can deduce the exact barcode formula for the Liouville torus of time-series data as described in Theorem \ref{cor:barcode} and Theorem \ref{direction barcode}. The consideration of different $\tau$ values motivates us to determine the homotopy type of the Vietoris-Rips complex of an ellipse, as partially demonstrated in \cite{adamaszek2019vietoris}. From now on, unless otherwise specified, we set $\tau = {2\pi \over M+1}$. This condition is useful to calculate the barcode and clarify our theory.}

The following theorem provides the justification of using the truncated Fourier approximation of the given time-series data for TDA. 
\begin{theorem}[Proposition 4.2, \cite{perea2015sliding}]
\label{truncated approximation}
Let $f \in C^{l}(\mathbb{T},\mathbb{R})$ and $S_N f$ be the $N$th truncated Fourier series of $f$.\label{meaning of N} 
If $\mathsf{bcd}_{n}^{\Rip}(f)$ and $\mathsf{bcd}_{n}^{\Rip}(S_N f)$ are the $n$-dimensional barcodes of $SW_{M,\tau}f$ and $SW_{M,\tau}S_{N}f$, then 
\[
d_{B}(\mathsf{bcd}_{n}^{\Rip}(f),\mathsf{bcd}_{n}^{\Rip}(S_N f)) 
\le 2 \sqrt{{2 \over 2l-1}} \lVert f^{(l)} - S_{N} f^{(l)} \rVert_{2} {\sqrt{M+1} \over (N+1)^{l-{1 \over 2}}}.
\]
\end{theorem}

Note that for fixed $n$ and $l$, the right-hand side  vanishes as $N$ goes to infinity. This theorem tells us that the barcode of the truncated Fourier series is an approximation of the barcode of the given time-series data with respect to the bottleneck distance. 

\begin{proposition}[Proposition 5.1, \cite{perea2015sliding}] \label{embedding dimension}
Let $u_{L} = \bigl( 1, \cos(L\tau), \cdots , \cos(LM\tau) \bigr)$ and $v_{L} =  \bigl( 0, \sin(L\tau),\cdots , \sin(LM$ $\tau) \bigr), L = 0, 1, \cdots, N$.  If $M\tau < 2\pi$, then $u_{0},u_{1},v_{1},\cdots,u_{N},v_{N}$ 
are linearly independent if and only if $M \ge 2N$. 
\end{proposition}

\begin{proposition}[Sec 5, \cite{perea2015sliding}] \label{roundness}
$SW_{M,\tau}\cos(Lt) = \cos(Lt)u_{L} - \sin(Lt)v_{L}$ and $SW_{M,\tau}\sin(Lt)=\sin(Lt)u_{L}$ $+ \cos(Lt)v_{L}$ are the images of the sliding window embedding of a sinusoidal function on $P_{L} := span\left\{ u_{L} , v_{L} \right\}$. We call $P_L$ $L$-plane.
\end{proposition}
Proposition \ref{embedding dimension} tells us that a sufficiently large embedding dimension is important to preserve geometric information. For example, suppose that $M =1$ and $\tau = {2 \pi \over M+1}$. Then $u_{1} = (1,-1)$,  $v_{1} = (0,0)$ and $SW_{M,\tau}\cos(t) = \cos(t)(1,-1)$, so $SW_{M,\tau}\cos(Lt)$ loses the circle information. {\color{black}On the other hand, if $M \ge 2$, for any $\tau$ that satisfies $M\tau < 2\pi$, the previous situation does not occur.} From now on, unless otherwise specified, we set $M = 2N$.

\begin{theorem}
[Theorem 5.6, \cite{perea2015sliding}] 
\label{N-torus} Let $C : \mathbb{R}^{M+1} \rightarrow \mathbb{R}^{M+1}$ be the centering map
\begin{equation*}
C(\mathbf{x}) = \mathbf{x} - {\left\langle \mathbf{x},\mathbf{1} \right\rangle \over \lVert \mathbf{1} \rVert^{2}} \mathbf{1} \mbox{ where } \mathbf{1} = \begin{bmatrix}
1 \\ \vdots \\ 1
\end{bmatrix} \in \mathbb{R}^{M+1}.  
\end{equation*}
If $S_{N}f(t) = \sum\limits_{n=0}^N a_{n} \cos(nt) + b_{n} \sin(nt)$, then 
$$C(SW_{M,\tau}S_N f (t)) = \sum\limits_{L=1}^{N} \sqrt{{M+1 \over 2}}r^{f}_{L}(\cos(Lt) \tilde{x}_{L} + \sin(Lt)\tilde{y}_{L}),$$ 
where $\ r_{L}^{f} = 2 \left| \hat{f}(L) \right|$ and orthonormal vectors $\tilde{x}_{L} = {\sqrt{{2 \over M+1}}} {a_{L}u_{L} + b_{L}v_{L} \over r_L^f}$ and $\tilde{y}_{L} = {\sqrt{{2 \over M+1}}} {b_{L}u_{L} - a_{L}v_{L} \over r_L^f}$.
Here $\hat{f}(L)$ is the $L$th Fourier coefficient. 
\end{theorem}
Let us define $\psi_{f,N}(t)$ as $\sqrt{{2 \over M+1}} C(SW_{M,\tau}S_N f (t))= {\sqrt{{2 \over M+1}}}\left( SW_{M,\tau}S_N f (t)  -\hat{f}(0) \cdot \mathbf{1} \right)$. That is, $\psi_{f,N}(t)$  is simply given by the following 
\begin{equation}
\label{eq:N-torus}
\psi_{f,N}(t) = \sum\limits_{L=1}^{N} r^{f}_{L}(\cos(Lt) \tilde{x}_{L} + \sin(Lt)\tilde{y}_{L}). 
\end{equation}
The above procedure helps our argument to become more concise. Here note that this procedure  does not change the topology of the given point cloud since we only apply an expansion and a translation to $S_Nf(t)$. From now on, unless otherwise specified, we abbreviate $\psi_{f,N}$ as $\psi_f$.

Computing persistent homology over the Vietoris–Rips complex of the sliding window embedding of time-series data becomes computationally expensive when the time-series length is large. This motivates the need for exact barcode formulas, which provide closed-form expressions that bypass costly computations. To this end, we next recall a standard fact about product metric spaces that will be useful in our analysis.

For metric spaces $(X_i,d_{X_i})$ $(i=1,\cdots,k)$, we equip the product space 
$X_1 \times \cdots \times X_k$ with the maximum metric, defined for 
$\mathbf{x}=(x_1,\cdots,x_k)$ and $\mathbf{y}=(y_1,\cdots,y_k)$ by
\[
d_{\max}(\mathbf{x},\mathbf{y})
= \max_{1 \le i \le k} d_{X_i}(x_i,y_i).
\]

The maximum metric is adopted because it makes the diameter of a set in the product equal to the largest diameter of its coordinate projections, which guarantees that the Vietoris–Rips complex of the product decomposes as the product of the Vietoris–Rips complexes of the factors.

\begin{proposition}[{Proposition 10.2, \cite{adamaszek2017vietoris}}]
\label{product rips}
Let $(X_1,d_{X_1}), \cdots ,(X_k,d_{X_k})$ be metric spaces and let $(X_1 \times \cdots \times X_k, d_{\max})$ be the product space  equipped with maximum metric. For $\epsilon \in \mathbb{R}$, 
$$\Rip_{\epsilon}(X_1 \times \cdots \times X_k) =  \Rip_{\epsilon}(X_1) \times \cdots \times \Rip_{\epsilon}(X_k).$$ 
\end{proposition}

To analyze the given topological space, we usually consider its subspace. For example, Seifert-Van Kampen theorem implies that to calculate the fundamental group of figure eight, it is enough to know the fundamental group of circle \cite{Hatcher:478079}. K\"unneth formula helps us to calculate homology group of the product space. Similarly persistent K\"unneth formula is a useful tool to calculate persistent homology of the product space from persistent homology of each space. 
In \cite{gakhar2019k}, it was shown that by using geometric realization and the equivalence between simplicial and singular homology, the classical Künneth formula also applies to simplicial complexes. Based on this result, one can derive a persistent version of the Künneth formula for Vietoris–Rips complexes. The following theorem presents this persistent Künneth formula, which serves as a key tool in our work.

\begin{theorem}[Persistent K\"unneth formula, Corollary 4.5, \cite{gakhar2019k}]
\label{thm:persistent-kunneth}
Let $\mathcal{K}^1 , \cdots, 
    \mathcal{K}^k $ 
be one-parameter filtrations of finite simplicial complexes.  
Then for all $n \ge 0$,
\begin{equation}
\label{eq:persistent-kunneth}
\mathsf{bcd}_{n}(\mathcal{K}^1 \times \cdots \times \mathcal{K}^k)
= \left\{ I_1^{n_1} \cap \cdots \cap I_k^{n_k} :
I_j^{n_j} \in \mathsf{bcd}_{n_j}(\mathcal{K}^j), \;
\sum_{j=1}^{k} n_j = n \right\},
\end{equation}
where $\mathcal{K}^1 \times \cdots \times \mathcal{K}^k$ denotes the product filtration, i.e., $(\mathcal{K}^1 \times \cdots \times \mathcal{K}^k)_\epsilon = \mathcal{K}^1_\epsilon \times \cdots \times \mathcal{K}^k_\epsilon$ for each $\epsilon \in \mathbb{R}$. In particular, when $\mathcal{K}^j = \Rip(X_j)$ for finite metric spaces $(X_j,d_{X_j})$, 
the product filtration is the Vietoris–Rips complex of the product space $(X_1 \times \cdots \times X_k, d_{\max})$.
\end{theorem}

Now finally the following definitions and proposition state about  multi-parameter persistent homology and rank invariant, one of the invariants of multi-parameter persistent homology.

\begin{definition}[Multi-parameter persistent homology]
\label{def:multiPH}
Let $\mathcal{K} = \{ \mathcal{K}_{\boldsymbol{\epsilon}} \}_{\boldsymbol{\epsilon} \in \mathbb{R}^d}$ be a $d$-parameter filtration of simplicial complexes. 
For $\boldsymbol{\epsilon}=(\epsilon_{1},\cdots,\epsilon_{d}), \boldsymbol{\epsilon}'=(\epsilon'_{1},\cdots,\epsilon'_{d}) \in \mathbb{R}^d$, 
we write $\boldsymbol{\epsilon} \le \boldsymbol{\epsilon}'$ if $\epsilon_{i} \le \epsilon'_{i}$ for all $i=1,\cdots,d$.
Then the pair
\[
\Bigl(\{ H_{n}(\mathcal{K}_{\boldsymbol{\epsilon}}) \}_{\boldsymbol{\epsilon}\in\mathbb{R}^d}, \;
\{ \iota_{*}^{\boldsymbol{\epsilon}, \boldsymbol{\epsilon}'} : H_{n}(\mathcal{K}_{\boldsymbol{\epsilon}}) 
\to H_{n}(\mathcal{K}_{\boldsymbol{\epsilon}'}) \}_{\boldsymbol{\epsilon} \le \boldsymbol{\epsilon}'} \Bigr)
\]
is called the $n$-dimensional $d$-parameter persistent homology, 
where $\iota_{*}^{\boldsymbol{\epsilon}, \boldsymbol{\epsilon}'}$ is the homomorphism induced by the inclusion 
$\iota^{\boldsymbol{\epsilon}, \boldsymbol{\epsilon}'} : \mathcal{K}_{\boldsymbol{\epsilon}} \hookrightarrow \mathcal{K}_{\boldsymbol{\epsilon}'}$.
\end{definition}

\begin{definition}[Rank invariant, \cite{carlsson2007theory}, \cite{lesnick2015interactive}]
\label{rank invariant}
Let $
\mathcal{H} := \{ (\boldsymbol{\epsilon},\boldsymbol{\epsilon}') \in \mathbb{R}^{d} \times \mathbb{R}^{d} : \boldsymbol{\epsilon} \le \boldsymbol{\epsilon}' \}$.
For a fixed dimension $n$, the rank invariant of a $d$-parameter persistent homology $\{H_n(\mathcal{K}_{\boldsymbol{\epsilon}})\}_{\boldsymbol{\epsilon} \in\mathbb{R}^d}$ is the function
\[
\begin{aligned}
\operatorname{rank}_n(\mathcal{K}) &: \mathcal{H} \to \mathbb{N}, \\
(\boldsymbol{\epsilon},\boldsymbol{\epsilon}') &\mapsto \operatorname{rank}\!\big(H_n(\mathcal{K}_{\boldsymbol{\epsilon}})
\xrightarrow{\;\iota_{*}^{\boldsymbol{\epsilon},\boldsymbol{\epsilon}'}\;} H_n(\mathcal{K}_{\boldsymbol{\epsilon}'})\big).
\end{aligned}
\]

\end{definition}

\begin{definition}[Fibered barcode, Sec 1.5, \cite{lesnick2015interactive}]
Let $\mathcal{L}$ be a collection of affine lines in $\mathbb{R}^{d}$ with nonnegative slope. 
For each $L \in \mathcal{L}$, restricting the $d$-parameter filtration 
$\mathcal{K} = \{ \mathcal{K}_{\boldsymbol{\epsilon}} \}_{\boldsymbol{\epsilon} \in \mathbb{R}^d}$ to $L$ 
yields a one-parameter persistent homology 
$\{H_n(\mathcal{K}_\epsilon)\}_{\epsilon\in L}$. 
Its barcode, denoted $\mathsf{bcd}_{n}(\mathcal{K}|_{L})$, is called the  fibered barcode of $\mathcal{K}$ along $L$. 
The collection $\{ \mathsf{bcd}_{n}(\mathcal{K}|_{L}) \}_{L \in \mathcal{L}}$ 
is called the the fibered barcode of $\mathcal{K}$.
\end{definition}

\begin{proposition}
[Sec 4.2, \cite{botnan2022introduction}]
\label{rank,fibered}
The rank invariant and fibered barcode are equivalent.
\end{proposition}

\end{section}

\begin{section}{Exact formula and interpretation of barcode}
\label{sec : Exact formula and interpretation of barcode}
In this section, we explain the Liouville torus in Hamiltonian dynamical systems. To justify its significance in TDA, we review Takens' embedding theorem and
explain sliding window embedding with the theorem. Analyzing time-series data with TDA through sliding window embedding involves inferring the trajectory of particles in the phase space (or state space). Similarly, the analysis of the Liouville torus of time-series data with TDA involves inferring the Liouville torus of particles in the phase space. Then we provide the exact barcode formula of the Liouville torus of time-series data based on the Fourier transform. We examine the properties of the barcode obtained from the Liouville torus.

\subsection{Takens' embedding theorem and Liouville torus}
\label{Continuosization section}
TDA of time-series data often involves converting the given time-series data into a point cloud using sliding window embedding according to Takens' embedding theorem. 
In general, however, it is hard to interpret the barcode constructed through sliding window embedding. 
For example, consider a time-series data given by $f(t) = \cos t + \cos 3t$. 
For this case, the Fourier coefficients in the cosine series are simply $(0, 1, 0, 1, 0, \cdots, 0)$. 
Further consider the case where such data is sampled with a length of $15$ and embedded using the sliding window method with $M=6$ and $\tau = \tfrac{2\pi}{M+1}$. 
We can easily show that its $1$-dimensional barcode has 11 intervals, as shown in Figure~\ref{cosinelength}.

\begin{figure}[h]
    \centering
    \includegraphics[scale=0.4]{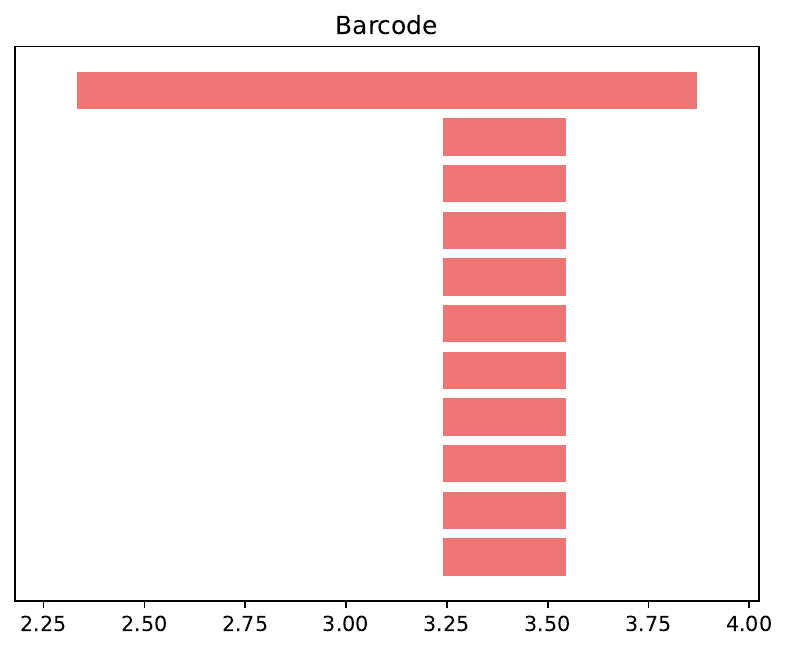}
    \caption {Barcode of $\cos t + \cos 3t$ with length $15$ ($M=6$, $\tau = \tfrac{2\pi}{M+1}$)}
    \label{cosinelength}
\end{figure}

Although the barcode consists of only a small number of intervals, it is not straightforward to understand the meaning of each bar. 
That is, the geometric interpretation of each bar is not easy to make. 
It is also difficult to guess the exact formula of the barcode of the data $f(t)$ from the viewpoint of each mode, i.e., $\cos t$ or $\cos 3t$.

Regarding sliding window embedding, notice that 
Takens' theorem is not limited to sliding window embedding. Thus, 
it is not necessary to exclusively rely on sliding window embedding for the TDA of time-series data. 
The main motivation for utilizing the Liouville torus is that it provides exact formulas with interpretability. This approach enables us to comprehend the information within the barcode unlike the sliding window embedding approach.

\begin{theorem}[Takens' embedding theorem \cite{takens2006detecting}]
Let $M$ be a compact manifold of dimension $m$. For pairs $(\phi, y)$ with $\phi \in Diff^2(M)$, $y \in C^2(M,\mathbb{R})$, it is a generic property that the map $\Phi_{(\phi, y)} : M \rightarrow \mathbb{R}^{2m+1}$, defined by
$$\Phi_{(\phi , y)}(x) = (y(x), y(\phi(x)), \cdots,y(\phi^{2m}(x)))$$
is an embedding. Here `generic' means that such $(\phi,y)$ consists of both an open subset and is dense in $Diff^2(M) \times C^2(M,\mathbb{R})$, and each space is equipped with the $C^2$-topology. We refer to functions $y \in C^2(M,\mathbb{R})$ as measurement functions.
\end{theorem}

Sliding window embedding is a method used to extract information about a dynamical system $\phi$ and its phase space (or state space) $M$ from measurements. If $M$ is non-compact and we restrict our measurement functions to be proper maps, then we can extend Takens’ embedding theorem to non-compact manifolds \cite{takens2006detecting}. In a nutshell, according to differential topology theory, any smooth function $f: M \rightarrow \mathbb{R}^{2m+1}$ can be approximated by an injective immersion. If $y$ is a proper map, then we can perturb $\Phi_{(\phi, y)}$ to be a proper injective immersion. Finally, we can apply the proposition that a proper injective immersion is an embedding \cite{mukherjee2015differential}. Instead of considering a proper measurement function, we can also focus on the compact subset of $M$. Since our experimental data is finite, it contains dynamical information within a certain compact subset of the phase space. Therefore, we can analyze this compact subset of the phase space \cite{huke2006embedding}. Figure \ref{Takens} shows the schematic illustration of Takens' embedding theorem. 

\begin{figure}[hbt!]
    \centering
    \includegraphics[scale=0.8]{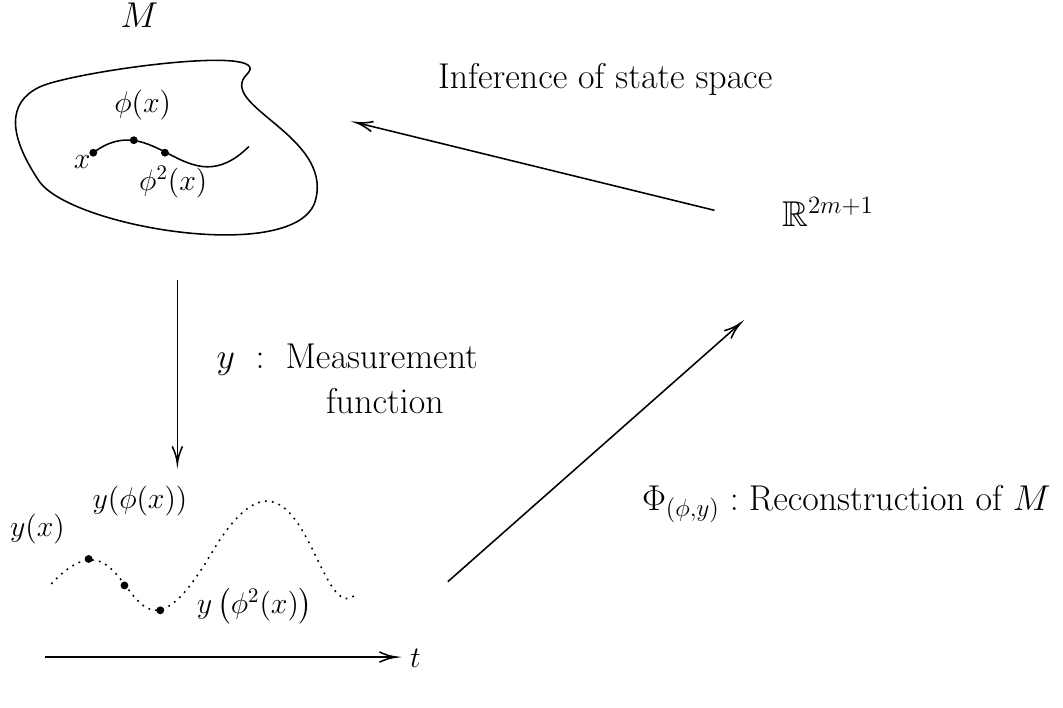}
    \caption{The schematic illustration of Takens' theorem. The figure shows how to infer information of $M$ using experimental data from the unknown dynamical system $(M,\phi)$.}
    \label{Takens}
\end{figure}

\begin{example}[Sliding window embedding of a discrete time-series data] 
Consider a discrete time-series $\left\{ z_t \right\}_{t \in \mathbb{Z}_{\ge 0}}$ generated by a dynamical system $(M, \phi)$, that is, $z_t = y(\phi^t(x))$.
The embedding map is given by $\Phi_{(\phi, y)}(x) = (y(x), y(\phi(x)), \cdots, y(\phi^{2m}(x))) = (z_0, z_1, \cdots, z_{2m})$, and after one iteration of $\phi$, we have $\Phi_{(\phi, y)}(\phi(x)) = (y(\phi(x)), \allowbreak y(\phi^2(x)),  \cdots, y(\phi^{2m+1}(x))) = (z_1, z_2, \cdots, z_{2m+1})$.
In general, $\Phi_{(\phi, y)}(\phi^k(x)) = (z_k, z_{k+1}, \cdots, z_{2m+k})$.
This construction corresponds to the standard sliding window (or time-delay) embedding of discrete time-series, which allows us to recover the topology of the trajectory on $M$ from the observed sequence $\left\{ z_t \right\}_{t \in \mathbb{Z}_{\ge 0}}$.
\end{example}

To understand the Liouville torus, we need to cover the basics of symplectic manifold theory. In brief, a Hamiltonian system is a description of a particle's trajectory using the Hamiltonian function $H : M \rightarrow \mathbb{R}$ and the Hamiltonian differential equation, where $M$ is the phase space of the particle. We have summarized the essential concepts in Appendix \ref{appendix:symplectic} to facilitate our theoretical development.

\begin{definition}[\cite{da2008lectures}]
Smooth functions $f_1, \cdots, f_n \in \mathcal{C}^{\infty}(M, \mathbb{R})$ are said to be independent if $(df_1)_p, \cdots, (df_n)_p$ are linearly independent at all $p$ in some open dense subset of $M$.    
\end{definition}

\begin{definition}[Integrable Hamiltonian system \cite{da2008lectures}]
A Hamiltonian system $(M,\omega,H)$ is called (completely) integrable if for $n = {1 \over 2}\dim M$, there are independent smooth functions $f_1 =H, f_2, \cdots, f_n \in C^{\infty}(M,\mathbb{R})$ such that $\left\{ f_i, f_j \right\} = 0$ for all $i,j$, where $(M,\omega)$ is a symplectic manifold, $H : M \rightarrow \mathbb{R}$ is a smooth map and $\left\{ \cdot, \cdot \right\}$ is the Poisson bracket.    
\end{definition}

Integrable Hamiltonian systems are known to have a maximal invariant set along the integral curves of the Hamiltonian vector field $X_H$. This deduction can be made using Theorem \ref{thm:conservity} and basic symplectic linear algebra (cf. p.8, \cite{da2008lectures}).

\begin{theorem}[Liouville–Arnold theorem, \cite{da2008lectures}]
\label{thm:Liouville-Arnold}
Let $(M,\omega, H)$ be an integrable Hamiltonian system and $n = {1 \over 2}\dim M$. Suppose $\mathbf{c} = (c_1,\cdots,c_n) \in \mathbb{R}^n$ is a regular value of $F=(f_1=H,\cdots,f_n)$ and denote the level set by $L_{\mathbf{c}} = F^{-1}(\mathbf{c})$. Then 
\begin{enumerate}
    \item $L_{\mathbf{c}}$ is a (Lagrangian) submanifold.
    \item If $L_{\mathbf{c}}$ is furthermore compact and connected, it is diffeomorphic to the $n$-torus $T^{n}$. 
    \item There exist (local) coordinates $(\theta _{1},\cdots ,\theta _{n}, I_{1},\cdots ,I_{n})$  on $M$ such that $\dot{\theta }_{i} = \omega_i$, where $\omega_i$ is a constant and $\dot{I}_i = 0$ on $L_{\mathbf{c}}$, i.e. $L_{\mathbf{c}} = L_{\mathbf{c}}(\theta _{1},\cdots ,\theta _{n})$. These coordinates are called angle-action coordinates.
\end{enumerate}    
\end{theorem}

{\color{black}
\begin{definition}[Liouville torus]
In (2) of Theorem \ref{thm:Liouville-Arnold}, we call $L_{\mathbf{c}}$ the Liouville torus in the integrable Hamiltonian system.
\end{definition}
}

If we assume time-series data $f : \mathbb{T} \rightarrow \mathbb{R}$ is obtained from a measurement $y : M \rightarrow \mathbb{R}$ of an integrable Hamiltonian system $(M, \omega, H)$, then $f$ can be expressed by $f(t) = y\left(\phi_H^t(x_0)\right)$, where $\phi_H^t$ is the Hamiltonian flow and $x_0 \in M$ is the initial point.

\begin{example}[Example 2.1.2 \cite{benatti2009dynamics}]
\label{ex:two harmonic}
Consider two uncoupled one-dimensional harmonic oscillators described by $(\mathbf{q},\mathbf{p}) = (q_1,q_2,p_1,p_2) \in M = \mathbb{R}^4$ and the Hamiltonian
$$H(\mathbf{q},\mathbf{p}) = \underbrace{{p_1^2 \over 2m_1} + {m_1\omega_1^2 \over 2} q_1^2}_{H_1(\mathbf{q},\mathbf{p})} + \underbrace{{p_2^2 \over 2m_2} + {m_2\omega_2^2 \over 2} q_2^2}_{H_2(\mathbf{q},\mathbf{p})}.$$
The trajectories that conserve the energy $H$ and $H_1$ (or equivalently, $H_1$ and $H_2$) for each harmonic oscillator are confined to the 2-torus $\mathbb{T}^2 := \left\{ \boldsymbol\theta = (\theta_1, \theta_2) : \theta_L \in [0, 2\pi) \right\}$, where $q_L = \sqrt{{2I_L \over m_L \omega_L}} \cos \theta_L, p_L = \sqrt{2m_L\omega_LI_L} \sin \theta_L$ and $I_L(\mathbf{q},\mathbf{p}) := H_L(\mathbf{q},\mathbf{p})/ \omega_L = {p^2_L \over 2m_L\omega_L} + {m_L\omega_L \over 2} q_L^2$ for $L=1,2$. This 2-torus is the Liouville torus in this Hamiltonian system. Note that $H(\mathbf{q},\mathbf{p}) = \omega_1 I_1 + \omega_2 I_2$. Let $f_1 = H, f_2 = H_1$ and $\mathbf{I} = (I_1,I_2), \boldsymbol\omega = (\omega_1, \omega_2)$, then we get the angle-action coordinate $(\boldsymbol\theta, \mathbf{I})$. Using the Hamiltonian equation, we can check $\dot{\boldsymbol\theta}  = {\partial H \over \partial \mathbf{I}} = \boldsymbol\omega$ and $\dot{\mathbf{I}} = -{\partial H \over \partial \boldsymbol\theta} = 0$. This means that the trajectory of a particle starting at $\boldsymbol\theta_0$ is governed by the Hamiltonian flow $\boldsymbol\theta(t) = \boldsymbol\theta_0 + \boldsymbol\omega t$ and preserves the action $\mathbf{I}$. In other words, in the motion, the energy of each harmonic oscillator is conserved. It is worth noting that if the ratio of $\omega_1$ and $\omega_2$ is a rational number, then the trajectory is a closed curve on the Liouville torus; otherwise, the trajectory fills the Liouville torus. This is related to the sliding window embedding of a quasi-periodic time-series data \cite{gakhar2023sliding}. However, in this paper, we only focus on periodic time-series data (trajectory is closed curve).
\end{example}

Now, we define the Liouville torus of time-series data, which is the main tool of this research.
\begin{definition}[Liouville torus of time-series data]
\label{def:Liouville torus}
Given a time-series data $f : \mathbb{T} \rightarrow \mathbb{R}$, we define the ($N$-truncated) Liouville torus $\Psi_{f,N}$ of $f$ as the Liouville torus of $H(\mathbf{q},\mathbf{p}) = \sum\limits_{L=1}^N {L \over 2}(p_L^2 + q_L^2)$ with each $L$th harmonic oscillator preserving $I_L = {(r_L^f)^2 \over 2}$, where $r_L^f = 2|\hat f(L)|$. From now on, unless otherwise specified, we abbreviate $\Psi_{f,N}$ as $\Psi_f$.
\end{definition}

\begin{theorem}
\label{thm:sliding comparison}
Sliding window embedding of time-series data can be formulated by the trajectory of uncoupled one-dimensional harmonic oscillators.
\end{theorem}

\begin{proof}
Note that the sliding window embedding of time-series data is given by $\psi_{f}(t) = \sum\limits_{L=1}^{N} r^{f}_{L}(\cos(Lt) \tilde{x}_{L} + \sin(Lt)\tilde{y}_{L})$. Consider a Hamiltonian system composed of uncoupled one-dimensional harmonic oscillators $H(\mathbf{q},\mathbf{p}) = \sum\limits_{L=1}^N{p_L^2 \over 2m_L} + {m_L\omega_L^2 \over 2} q_L^2$. Let the initial condition be $\boldsymbol\theta_0 = (0,\cdots,0)$, the frequency vector $\boldsymbol\omega = (1,\cdots,N)$, and $m_L = {1 \over \omega_L}$ for $L=1, \cdots, N$. Then, similar to Example \ref{ex:two harmonic}, its trajectory is $q_L(t) = \sqrt{2I_L}\cos Lt$ and $p_L(t) = \sqrt{2I_L}\sin Lt$ in the phase space $(\mathbf{q}, \mathbf{p})$. Therefore, $\psi_f$ can be formulated by the trajectory of such Hamiltonian system that preserves the condition $I_L = {(r_L^f)^2 \over 2}$ for each harmonic oscillator. In the sliding window embedding space, we choose the orthonormal basis $\left\{ \tilde{x}_{L}, \tilde{y}_{L} \right\}_{L=1}^{N}$. The linear map from this orthonormal set to the standard basis on $(\mathbf{q}, \mathbf{p})$ is an isometry, meaning that $\psi_f$ and the trajectory are isometric.
\end{proof}

{\color{black}From this theorem, we know $\psi_{f,N} \subseteq \Psi_{f,N}$ and when $N=1$, $\psi_{f,N} = \Psi_{f,N}$ holds. Since the Vietoris-Rips complex of two metric spaces that are isometric to each other is the same, the following corollary makes it convenient to handle the Liouville torus of time-series data.

\begin{corollary}
\label{cor:Liouville isomorphic}
The Liouville torus $\Psi_f$ is isometric to $r_1^f \cdot \mathbb{S}^1 \times \cdots \times r_N^f \cdot \mathbb{S}^1$. Therefore, we can identify both of them.  
\end{corollary}}

Table \ref{table:comparison} summarizes the proof of Theorem \ref{thm:sliding comparison}. The degree of the truncated Fourier series corresponds to the number of uncoupled harmonic oscillators, each with a frequency of a multiple of the fundamental frequency. We note that we adjust $\tau$ to ensure the circular shape of the data in  the sliding window embedding space.  This can be done by controlling the mass of each harmonic oscillator. The radius of the circle 
associated with the $L$th Fourier mode corresponds to the conserved energy of the harmonic oscillator.
\begin{table}[h]
    \centering
\begin{tabular}{lll}
Sliding window embedding      & $\longleftrightarrow$ & Uncoupled one-dimensional harmonic oscillators \\ \hline
$N$th truncated Fourier series & $\longleftrightarrow$ & $\omega = (1, \cdots, N )$                            \\ 
Control $\tau$                     & $\longleftrightarrow$ & Control $\mathbf{m} = (m_1,\cdots, m_N )$                  \\ 
$r_L^f$                            & $\longleftrightarrow$ & $I_L$                                             \\ \hline
\end{tabular}
    \caption{Summary of the relationship between sliding window embedding and uncoupled one-dimensional harmonic oscillators.}
    \label{table:comparison}
\end{table}

Figure \ref{fig:Liouville_torus} illustrates the difference between the sliding window embedding of the data and the Liouville torus, which contains the sliding window embedding. According to Takens' embedding theorem, sliding window embedding focuses on the shape of the trajectory (or the orbit) within the phase space. On the other hand, the Liouville torus focuses on the invariance of the trajectory (or the orbit) within the phase space.  

Since we have the Vietoris-Rips barcode formula of $\mathbb{S}^1$ in Theorem \ref{barcode of circle}, it enables us to find the exact formula of barcode of the Liouville torus, unlike the sliding window embedding, and provides a more concise interpretation of the barcode. It is important to note that the two different time-series data on the same Liouville torus are identical topologically when their conserved quantities are equal even though the initial conditions are different.

\begin{figure}
    \centering
    \includegraphics{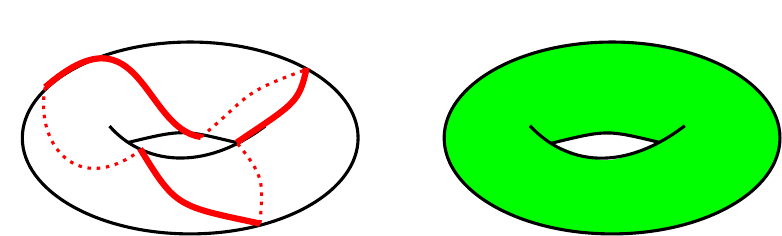}
    \caption{Left: Sliding window embedding of time-series data $\psi_f$, Right: Corresponding Liouville torus $\Psi_f$.}
    \label{fig:Liouville_torus}
\end{figure}

\subsection{Barcode in different metric spaces} 
\label{sec:choice of metric}

Note that in Theorem \ref{thm:persistent-kunneth} the product space is equipped with maximum metric. Thus, to apply Theorem \ref{thm:persistent-kunneth} to $\Psi_{f}$ for arbitrary $f$, we should consider maximum metric on $\Psi_{f}$. But we  note that the main characteristic of TDA is not metric but topology. Two different but topologically equivalent metric functions induce different hidden structures (e.g. Vietoris-Rips complex) and manifold inferences while those two metrics induce same topology. We explain this by the following example. 

\begin{example} 
\label{Different choice of metric}
Consider a  point cloud, as shown in the left in Figure \ref{differentmetric},  that is composed of four points in $\mathbb{R}^{2}$. In the figure, we consider two metrics, namely the Euclidean  and  maximum distances, denoted by  $d_2$ and $d_{\max}$, respectively. The right in Figure \ref{differentmetric} shows the generated complexes with filtration. As shown in the figure, with $d_2$ there is an intermediate complex of square while there is no such square with $d_{\max}$. This example shows that different choices of metric induce different Vietoris-Rips complex and barcode, {\color{black}while the topology of the space ($\mathbb{R}^2$) containing the point cloud is unchanged.}
\begin{figure}[h]
    \centering
    \includegraphics{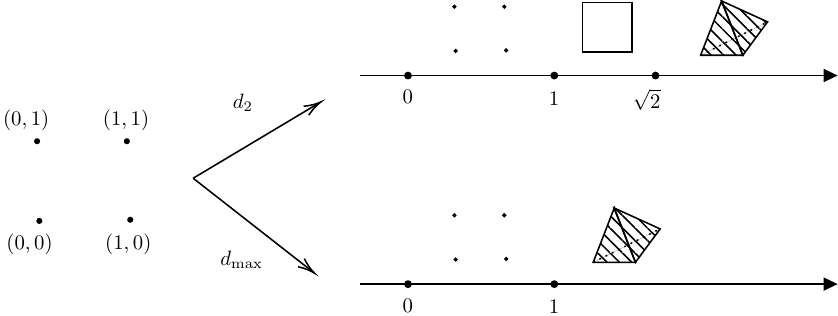}
    \caption{The given point cloud (left) and the corresponding Vietoris-Rips complexes with filtration for $d_2$ and $d_{\max}$ where $d_2$ and $d_{\max}$ denote the Euclidean and maximum metrics, respectively.}
    \label{differentmetric}
\end{figure}
\end{example}

Thus, for homological analysis it is not necessary to use the Euclidean metric $d_{2}$. 
Different metrics, however, induce different hidden structures and different topological inferences. 
If a persistent K\"unneth formula is available for the $p$-norm, it yields additional topological inferences for the given point cloud.

Note that our point cloud $\Psi_{f}$ forms an $N$-torus lying in the direct sum 
$P_1 \oplus \cdots \oplus P_N \subset \mathbb{R}^{M+1}$. 
Here $P_L$ denotes the $L$-plane, and since in Section~\ref{sec:Defintions} we assumed $\tau = \tfrac{2\pi}{M+1}$ with $M=2N$, 
the hypothesis of Proposition~\ref{roundness} is satisfied, ensuring that these planes are mutually disjoint except at the origin. 
The two metric spaces $(P_1 \oplus \cdots \oplus P_N, d_{2})$ and $(P_1 \oplus \cdots \oplus P_N, d_{\max})$ are topologically equivalent. Choosing $d_{\max}$ instead of $d_{2}$ does not change the topology of the space but only modifies the filtration, and can therefore be regarded as one of the possible topological inferences mentioned earlier.
In order to apply Theorem~\ref{thm:persistent-kunneth}, 
we analyze $\Psi_f$ within the direct sum space $(P_1 \oplus \cdots \oplus P_N, d_{\max})$.

\subsection{Exact formula and interpretation of barcode} 
\label{sec:Exact formula and interpretation}
In this section, we derive the exact formula of the barcode and provide its interpretation. To do this, first we derive the exact barcode formula of sinusoidal time-series data. First define a projection 
\begin{equation}
\pi_{i_1 i_2 \cdots i_{N}} : \mathbb{R}^{M+1} \rightarrow P_{i_1} \oplus \cdots \oplus P_{i_{N}}
\label{projection map}
\end{equation} 
by 
\begin{equation}
\pi_{i_1 i_2 \cdots i_{N}}(x) = \sum\limits_{j=1}^{N} c_{i_j} \tilde{x}_{i_j} + d_{i_j} \tilde{y}_{i_j}
\end{equation}
where $x = c_{1} \tilde{x}_{1} + d_{1} \tilde{y}_{1} + \cdots + c_{N} \tilde{x}_{N} + d_{N} \tilde{y}_{N}$ and $\tilde{x}_{i_j}$ and $\tilde{y}_{i_j}$ are defined in Theorem \ref{N-torus}. Our concerned point cloud is the Liouville torus $\Psi_{f}$ of the $N$th Fourier truncated time-series data \textcolor{black}{$f$}. But note that  Theorem \ref{thm:persistent-kunneth}
holds for
finite metric spaces. Therefore, we here generalize Theorem \ref{thm:persistent-kunneth}
for totally bounded metric spaces.

\begin{definition}[Correspondence, \cite{chazal2014persistence}]
We state that $C \subseteq X \times Y$ is a correspondence if for every $x \in X$, there exists $y_x \in Y$ such that $(x,y_x) \in C$ and for every $y \in Y$, there exists $x_y \in X$ such that $(x_y, y) \in C$.
\end{definition}

\begin{definition}[Gromov-Hausdorff distance, \cite{chazal2014persistence}] \label{Gromov-Hausdorff distance} Let $(X,d_X)$ and $(Y,d_Y)$ be metric spaces. Distortion of a correspondence $C \subseteq X \times Y$ is defined by $$dis(C) = \sup \left\{\left| d_{X}(x,\tilde{x}) - d_{Y}(y,\tilde{y}) \right| : (x,y),(\tilde{x},\tilde{y}) \in C  \right\}.$$
The Gromov-Hausdorff distance between $(X,d_X)$ and $(Y,d_Y)$ is defined by
$$d_{GH}(X,Y) = {1 \over 2}\inf	\left\{ dis(C) : C \subseteq X \times Y \mbox{ is a correspodence} \right\}.$$
\end{definition}

\begin{definition}[Totally boundedness]
A metric space $(X,d)$ is called totally bounded metric space if for every $r>0$, there are finite numbers of elements $x_1,\cdots,x_n$ such that $X = \bigcup\limits_{i=1}^{n} B_r(x_{i})$, where $B_r(x) := 	\bigl\{ y \in X $: $d(x,y)<r \bigr\}$.
\end{definition}

\begin{proposition}[Theorem 5.2, \cite{chazal2014persistence}]
\label{app:stability theorem}
Let $X,Y$ be totally bounded metric spaces. Then
$$d_B\left(\mathsf{bcd}_n^{\Rip}(X),\mathsf{bcd}_n^{\Rip}(Y) \right) \le 2 d_{GH}(X,Y).
$$
\end{proposition}

\begin{proposition}
\label{product totally}
If $(X_1, d_{X_1}),\cdots, (X_k, d_{X_k})$ are totally bounded metric spaces, then $(X_1 \times \cdots \times X_k, d_{\max})$ is also a totally bounded metric space. 
\end{proposition}

\begin{proof}
Fix $r > 0$. For every $j\in \left\{1,\cdots,n \right\}$, there are $n_j \in \mathbb{N}$ and $x^{j}_{i_j} \in X_j$ for $1 \le i_j \le n_j$ such that 
$$
X_j = \bigcup\limits_{i_j =1}^{n_j} B_r (x^j_{i_j}).
$$
We can show that 
$$X_1 \times \cdots \times X_k = \bigcup\limits_{i_1,\cdots,i_n} B_r	\left( (x^1_{i_1}, \cdots, x^k_{i_k}) \right).$$
$(\supseteq)$ is trivial. \\
$(\subseteq)$ Let $(y_1,\cdots,y_k) \in X_{1} \times \cdots \times X_k$. Then for each $j$, there are $x^{j}_{i_j}$ such that $y_j \in B_r(x^{j}_{i_j})$. \\
$\implies d_{\max} \left( (y_1 , \cdots, y_k), (x^1_{i_1}, \cdots, x^{k}_{i_k}) \right) < r \implies (y_1,\cdots,y_k) \in \bigcup\limits_{i_1,\cdots,i_k} B_r	\left( (x^1_{i_1}, \cdots, x^k_{i_k}) \right)$. Therefore $(X_1 \times \cdots \times X_k, d_{\max})$ is a totally bounded metric space. 
\end{proof}

\noindent

\begin{lemma}[General version of persistent K\"unneth formula \uppercase\expandafter{\romannumeral1}]
\label{totally bounded Persistent Kunneth formula}
Let $(X_1,d_{X_1}),\cdots, (X_k,d_{X_k})$ be totally bounded metric spaces. Then,
\begin{equation}
\label{Eq:general kunneth}
\mathsf{bcd}_{n}^{\Rip}(X_1 \times \cdots \times X_k , d_{\max}) = \left\{ J^{n_1}_1 \cap \cdots \cap J^{n_k}_k : J^{n_j}_j \in \mathsf{bcd}_{n_j}^{\Rip}(X_j,d_{X_j}) \ \text{and} \ \sum\limits_{j=1}^{k} n_j = n \right\}\end{equation}
for all $n\in\mathbb{Z}_{\ge 0}$ and $d_{\max}$ is the maximum metric.
\end{lemma}

\begin{proof}
Fix $r > 0$. For every $j\in \left\{1,\cdots,n \right\}$, there are $n_j \in \mathbb{Z}_{\ge 0}$ and $x^{j}_{i_j} \in X_j$ for $1 \le i_j \le n_j$ such that $X_j = \bigcup\limits_{i_j =1}^{n_j} B_r (x^j_{i_j})$. Let $X^r_j = \left\{ x^j_{i_j} \right\}_{1 \le i_j \le n_j}$. Note that $X^r_j$ are finite and $d_{GH}(X_1 \times \cdots \times X_k, X_1^r \times \cdots \times X_k^r) \le \max\limits_{j} d_{GH}(X_j, X_j^r) < r$.  By Proposition \ref{product totally}, we can apply Proposition \ref{app:stability theorem}, i.e.
$$d_B\left(dgm_n^{\Rip}(X_1 \times \cdots \times X_k),dgm_n^{\Rip}(X_1^r \times \cdots \times X_k^r) \right) \le 2 d_{GH}(X_1 \times \cdots \times X_k,X_1^r \times \cdots \times X_k^r) < 2r.$$
Therefore $dgm_n^{\Rip}(X_1^r \times \cdots \times X_k^r)$ converges to $dgm_n^{\Rip}(X_1 \times \cdots \times X_k)$ as $r\rightarrow 0$ with respect to the bottleneck distance. Hence Theorem \ref{thm:persistent-kunneth} can be generalized to totally bounded metric spaces.    
\end{proof}

If we apply Theorem \ref{barcode of circle} and Lemma \ref{totally bounded Persistent Kunneth formula}, we obtain the exact formula of $\Psi_{f}$. Moreover we can show that  each bar in $\mathsf{bcd}_{n}^{\Rip}(\Psi_{f})$ represents the bar of the projected point cloud onto $P_{i_1} \oplus \cdots \oplus P_{i_k}$, where $k=1,\cdots, n$ and $\sum\limits_{\substack{L=1 \\ n_{i_L} \in \mathbb{N}}}^k n_{i_L} = n$.

\begin{lemma}
\label{cosine barcode}
Suppose that the time-series data is $f = \cos Lt$ or $f= \sin Lt$. Then we have the followings

\begin{equation}
\label{cosinebarcodeformula}
\mathsf{bcd}_{n}^{\Rip}\left(\Psi_{f} \right)=
\begin{cases}
\left\{(0, \infty) \right\}, & \mbox{if }n=0 \\
\left\{ \left( 2 \sin\left(\pi{k \over 2k +1}\right), 2 \sin\left(\pi {k+1 \over 2k+3}\right) \right] \right\}, & \mbox{if } n=2k+1, k\in \mathbb{Z}_{\ge 0} \\
\ \emptyset, & \mbox{otherwise} 
\end{cases}. 
\end{equation}

\end{lemma}

\begin{proof}
Note that $\Psi_{f} = r_L^{f} \cdot \mathbb{S}^{1}$ and $ \mathsf{bcd}_{n}^{\Rip}\left( \textcolor{black}{\Psi_{f}} \right) = \mathsf{bcd}_{n}^{\Rip}\left( r_L^{f} \cdot \mathbb{S}^{1} \right) = \mathsf{bcd}_{n}^{\Rip}\left(\mathbb{S}^{1} \right)$ ($\because$ $r_{L}^{f} = 1$). By Theorem \ref{barcode of circle}, we obtain the formula (\ref{cosinebarcodeformula}).
\end{proof}
Since $\Psi_{f} = \pi_1 \Psi_{f} \times \cdots \times  \pi_N \textcolor{black}{\Psi_{f}}$ and $\mathsf{bcd}_{n}^{\Rip}\left( \pi_L \Psi_{f} \right) = r_L^f \cdot \mathsf{bcd}_{n}^{\Rip}\left(\mathbb{S}^{1} \right)$, according to Lemmas \ref{totally bounded Persistent Kunneth formula} and  \ref{cosine barcode}, we obtain the following theorem. 
\begin{theorem}
\label{cor:barcode}
$\mathsf{bcd}_{n}^{\Rip}(\Psi_{f}) = 	\left\{ J_1^{n_1} \bigcap \cdots \bigcap J_N^{n_N} : J_L^{n_{L}} \in \mathsf{bcd}_{n_L}^{\Rip}	\left(\pi_{L} \Psi_{f} \right) \ \text{and} \ \sum\limits_{L=1}^{N} n_L = n \right\}$, i.e.
$$
J_{L}^{n} = \begin{cases} (0,\infty), & \mbox{if }n=0 \\ \left(2r_L^f \sin\left(\pi {k \over 2k+1}\right), 2r_L^f \sin\left(\pi {k+1 \over 2k+3}\right)\right], & \mbox{if }n=2k+1 \\ \ \emptyset, & \mbox{otherwise}
\end{cases}.
$$
\end{theorem}
Clearly, $\mathsf{bcd}_{0}^{\Rip}(\Psi_{f}) = \left\{ (0, \infty) \right\}, \mathsf{bcd}_{1}^{\Rip}(\Psi_{f}) = \left\{ I_{1}, \cdots, I_{N} : I_{L} = \left(0, r_{L}^f \sqrt{3}  \right]  \right\}$. Notice that Lemma \ref{cosine barcode} indicates that the $2$-dimensional barcode of a circle is an empty set. Suppose that we set $n=2$. In  (\ref{Eq:general kunneth}), $\sum\limits_{L=1}^{N} n_L = 2$ implies $n_{i_1} = n_{i_2} = 1$ for $1 \le i_1 < i_2 \le N$ and $n_{L} = 0$ for $L \ne i_1,i_2$. Therefore 
\begin{eqnarray}
\mathsf{bcd}_{2}^{\Rip}(\Psi_{f}) &=& \Bigl\{ I_{i_1} \bigcap I_{i_2} :   I_{L} = \left(0, r_{L}^f \sqrt{3}  \right] \ \text{and} \ 1 \le i_1 < i_2 \le N \Bigr\} 
\nonumber \\
&=& \left\{ \left(0, \min	\left(r_{L_{i_1}}^f, r_{L_{i_2}}^f \right) \sqrt{3} \right] : \ \text{and} \ 1 \le i_1 < i_2 \le N \right\}. \nonumber 
\end{eqnarray}
But for $3$ or higher dimensions it is possible to have $n_{i_{j}} > 1$. Thus,  for $n \ge 3$, those $n$-dimensional barcodes have various type of bars.

Now we provide the interpretation of the derived exact barcode.

\begin{theorem}
\label{barcode meaning}
Each bar in $\mathsf{bcd}_{n}^{\Rip}(\Psi_{f})$ represents the bar of the projected point cloud onto $P_{i_1} \oplus \cdots \oplus P_{i_k}$, where $k=1,\cdots,n$ and $\sum\limits_{\substack{L=1 \\ n_{i_L} \in \mathbb{N}}}^k n_{i_L} = n$. That is, $\mathsf{bcd}_{n}^{\Rip}(\Psi_{f}) = \bigcup\limits_{1\le i_1 < \cdots <i_k \le N} \bigcup\limits_{1\le k \le n} \mathsf{bcd}_{n}^{\Rip}	\left( \pi_{i_1 \cdots i_k} \Psi_{f} \right)$.
\end{theorem}

\begin{proof}
$\sum\limits_{L=1}^{N} n_L = n$ from Theorem \ref{cor:barcode} and suppose that $n_{i_1}, \cdots, n_{i_k} >0$ and $n_j =0$ for $j \in \left\{1, \cdots, N \right\} \setminus \left\{ i_1, \cdots, i_k \right\}.$ 
Since $J^{n_j}_{j} = J^{0}_{j} = [0,\infty)$,
we have
$J_1^{n_1} \bigcap \cdots \bigcap J_N^{n_N} =J_{i_1}^{n_{i_1}} \bigcap \cdots \bigcap J_{i_k}^{n_{i_k}}.$ Therefore, we have the following 
\begin{flalign*}  
\mathsf{bcd}_{n}^{\Rip}(\textcolor{black}{\Psi_{f}}) &= \left\{ J_1^{n_1} \bigcap \cdots \bigcap J_N^{n_N} : J_L^{n_{L}} \in \mathsf{bcd}_{n_L}^{\Rip}	\left(\pi_{L}\textcolor{black}{\Psi_{f}}\right) \ \text{and} \ \sum\limits_{L=1}^{N} n_L = n \right\} \\
  &= \bigcup\limits_{1\le i_1 < \cdots <i_k \le N} \bigcup\limits_{1\le k \le n}\left\{ J_{i_1}^{n_{i_1}} \bigcap \cdots \bigcap J_{i_k}^{n_{i_k}} : J_{i_L}^{n_{i_L}} \in \mathsf{bcd}_{n_{i_L}}^{\Rip}	\left( \pi_{i_L} \Psi_{f} \right) \ \text{and} \sum\limits_{L=1}^{k} n_{i_L} = n \right\} \\&= \bigcup\limits_{1\le i_1 < \cdots <i_k \le N} \bigcup\limits_{1\le k \le n}\left\{ J_{i_1}^{n_{i_1}} \bigcap \cdots \bigcap J_{i_k}^{n_{i_k}} : J_{i_L}^{n_{i_L}} \in \mathsf{bcd}_{n_{i_L}}^{\Rip}	\left( \pi_{i_L} (\pi_{i_1 \cdots i_k} \Psi_{f}) \right) \ \text{and} \sum\limits_{L=1}^{k} n_{i_L} = n \right\}  \\&=\bigcup\limits_{1\le i_1 < \cdots <i_k \le N} \bigcup\limits_{1\le k \le n} \mathsf{bcd}_{n}^{\Rip}	\left( \pi_{i_1 \cdots i_k} \Psi_{f} \right).&&
\end{flalign*} 
{\color{black}The last equality utilizes the fact that $\pi_{i_1}(\pi_{i_1 \cdots i_k} \Psi_f) \times \cdots \times \pi_{i_k}(\pi_{i_1 \cdots i_k} \Psi_f) = \pi_{i_1 \cdots i_k} \Psi_f$ along with Theorem \ref{cor:barcode}.}
\end{proof}

\begin{example}
\label{onedimension persistent}
$\mathsf{bcd}_{1}^{\Rip}(\Psi_{f}) = \left\{ I_{1}, \cdots, I_{N} : I_{L} = \left(0, r_{L}^f \sqrt{3}  \right]  \right\} = \bigcup\limits_{1\le L \le N} \mathsf{bcd}_{1}^{\Rip}(\pi_L \textcolor{black}{\Psi_{f}}) $. Each bar $I_{L}$ in the barcode represents the barcode of the projected point cloud onto $L$-plane.
\end{example}

\begin{figure}[h]
    \centering
    \includegraphics{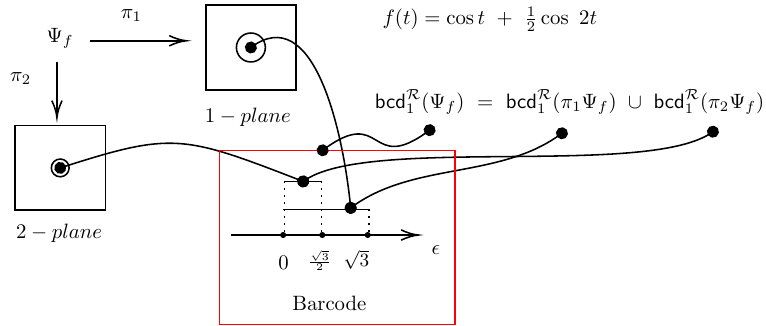}
    \caption{Schematic illustration of the 1-dimensional barcode}
    \label{Barcode interpretation}
\end{figure}

Figure \ref{Barcode interpretation} shows the corresponding barcode to the time-series data of  $f(t) = \cos t + {1 \over 2} \cos 2t$. In this case, the barcode of $\Psi_{f}$ is given by the two bars from the projections to 1-plane and 2-plane, corresponding to the Fourier modes $1$ and $2$, respectively.  That is, 
\begin{eqnarray}
\mathsf{bcd}_{1}^{\Rip}(\Psi_{f}) &=& \mathsf{bcd}_{1}^{\Rip}(\pi_{1}\Psi_{f}) \bigcup \mathsf{bcd}_{1}^{\Rip}(\pi_{2}\Psi_{f})\nonumber \\
&=& \left\{ \left(0, \sqrt{3}\right], 	\left(0,{\sqrt{3} \over 2}\right] \right\}.
\nonumber 
\end{eqnarray}
The figure shows how the actual barcode is decomposed into two bars, each from the projection. The circles in the black square boxes indicate that the bar in the projected space is not empty. Notice that the radii of the circles in the black square boxes are different. The radius of each circle is proportional to the size of the corresponding Fourier coefficient. That is, the radius of
{\color{black}$\pi_{1}\Psi_{f}$ is twice the radius of $\pi_{2}\Psi_{f}$}. This is also reflected in the barcode, $\mathsf{bcd}_{1}^{\Rip}(\Psi_{f})$ in the red square box. 

\begin{example}
\label{twodimension persistent}
We consider the two-dimensional barcode. That is, 
\begin{eqnarray} 
\mathsf{bcd}_{2}^{\Rip}(\Psi_{f}) &=& {\textstyle \bigcup\limits_{1\le i_1 \le N}} \mathsf{bcd}_{2}^{\Rip}(\pi_{i_1}\Psi_{f}) \ \bigcup \  {\textstyle \bigcup\limits_{1\le i_1 <i_2 \le N}} \mathsf{bcd}_{2}^{\Rip}(\pi_{i_1 i_2}\Psi_{f})
\nonumber \\
&=& 
\emptyset \quad \bigcup \  {\textstyle \bigcup\limits_{1\le i_1 <i_2 \le N}} \mathsf{bcd}_{2}^{\Rip}(\pi_{i_1 i_2}\Psi_{f})
\nonumber \\
&=& 
\left\{ I_{i_1} \bigcap I_{i_2} :   I_{L} = \left(0, r_{L}^f \sqrt{3}  \right] \ \text{and} \ 1 \le i_1 < i_2 \le N \right\}
\nonumber \\
&=& \Bigl\{ \left(0, \min	\left(r_{L_{i_1}}^f, r_{L_{i_2}}^f \right) \sqrt{3}  \right]: 1 \le i_1 < i_2 \le N \Bigr\} \nonumber .
\end{eqnarray}
For this example, each bar 
$I_{i_1} \cap I_{i_2}$ in $\mathsf{bcd}_{2}^{\Rip}(\Psi_{f})$ is corresponding to the projected point cloud onto $P_{i_1} \oplus P_{i_2}$. When we consider $\mathsf{bcd}_{2}^{\Rip}(\Psi_{f})$, we can neglect the Fourier coefficients of large amplitude. For example, let $f(t) = \cos t + \cos 2t$ and $g(t) = \cos t + 10 \cos 2t$. These two time-series data are apparently different. But we can regard these two time-series to be same with respect to 2-dimensional persistent homology.
\end{example}

The following proposition allows us to use the permutation symmetric property for time-series data. Note that the barcode is a combinatorial object and it is permutation symmetric (e.g., two barcodes, $\left\{ (0,1], (3,5] \right\}$ and $\left\{ (3,5], (0,1] \right\}$ are identical).
\begin{proposition}
\label{prop:permutation symmetric}
Let $\sigma : \left\{ 1, \cdots, N \right\} \rightarrow  \left\{ 1, \cdots, N \right\}$ be a bijection. Then if a time-series data $g$ satisfies $\hat{g}(n) = \hat{f}(\sigma(n))$, 
 we have $\mathsf{bcd}_{n}^{\Rip}(\Psi_{f}) = \mathsf{bcd}_{n}^{\Rip}(\Psi_{g})$.
\end{proposition}

\begin{proof}
First we show that $\mathsf{bcd}_{n}^{\Rip}(\pi_{\sigma (L)} \textcolor{black}{\Psi_{f}}) = \mathsf{bcd}_{n}^{\Rip}(\pi_{L} \textcolor{black}{\Psi_{g}})$. Note that $\hat{g}(n) = \hat{f}(\sigma(n))$ implies $r_L^g = r_{\sigma (L)}^f$. Equation (\ref{eq:N-torus}) yields 
\begin{flalign*}  
\pi_{\sigma (L)} \textcolor{black}{\Psi_{f}} &= r_{\sigma (L)}^f 	\left( \cos (\sigma (L) t) \tilde{x}_{\sigma (L)} + \sin (\sigma (L) t) \tilde{y}_{\sigma (L)}\right) \\
&= r_{L}^g 	\left( \cos (\sigma (L) t) \tilde{x}_{\sigma (L)} + \sin (\sigma (L) t) \tilde{y}_{\sigma (L)}\right)  \ 
\end{flalign*} 
and 
\begin{flalign*} 
\pi_{L} \Psi_{g}(\mathbb{T}) &= r_L^g 	\left( \cos (Lt) \tilde{x}_{L} + \sin (Lt) \tilde{y}_{L}\right) . 
\end{flalign*} 
Therefore $\pi_{\sigma (L)} \Psi_{f}$ and $\pi_{L} \Psi_{g}$ are isometric. By  Proposition \ref{app:stability theorem}, we have $\mathsf{bcd}_{n}^{\Rip}(\pi_{\sigma (L)} \Psi_{f})=  \mathsf{bcd}_{n}^{\Rip}(\pi_{L}\Psi_{g})$. Finally, 
from Theorem \ref{cor:barcode}, 
\begin{flalign*}  
\mathsf{bcd}_{n}^{\Rip}(\Psi_{f}) &= 	\left\{ J_1^{n_1} \bigcap \cdots \bigcap J_N^{n_N} : J_L^{n_{L}} \in \mathsf{bcd}_{n_L}^{\Rip}	\left(\pi_{L} \Psi_{f})\right) \ \text{and} \ \sum\limits_{L=1}^{N} n_L = n \right\} \\
&= \left\{ J_1^{n_1} \bigcap \cdots \bigcap J_N^{n_N} : J_{\sigma (L)}^{n_{L}} \in \mathsf{bcd}_{n_L}^{\Rip}	\left(\pi_{\sigma (L)} \Psi_{f})\right) \ \text{and} \ \sum\limits_{L=1}^{N} n_L = n \right\} \\
&= \left\{ \tilde{J}_1^{n_1} \bigcap \cdots \bigcap \tilde{J}_N^{n_N} : \tilde{J}_{L}^{n_{L}} \in \mathsf{bcd}_{n_L}^{\Rip}	\left(\pi_{L} \Psi_{g})\right) \ \text{and} \ \sum\limits_{L=1}^{N} n_L = n \right\} \\
&= \mathsf{bcd}_{n}^{\Rip}(\Psi_{g})
\end{flalign*} 
where 
we set $\tilde{J}_{L}^{n_L} = J_{\sigma (L)}^{n_L}$. 
\end{proof}
\end{section}

\begin{section}{Application of multi-parameter theory and its interpretation}
\label{Sec : Application of multi-parameter theory and its interpretation}
As mentioned earlier, one-parameter persistent homology theory may not be sufficient to capture the important characteristics of the given data. For example, as shown in Lemma~\ref{cosine barcode}, $\cos L_{1}t$ and $\cos L_{2}t$ yield the same barcode. But they are physically different and one may need to distinguish them when it needed. 
For this reason, based on the results from the previous section, we propose a multi-parameter persistent homology method based
on the filtration with the Fourier bases.

Due to theoretical shortage of complete invariant in multi-parameter persistence theory, we consider an incomplete invariant, that is, the rank invariant (cf. Definition \ref{rank invariant}). 
In this section, we consider one-dimensional reduction of multi-parameter persistent homology and derive the exact barcode formula. We provide the detailed analysis of the proposed method. 
\subsection{Construction of  multi-parameter persistent homology}
Note that persistent homology is the method that matches simplicial complex to each point in the filtration space and record the changes of homology of simplicial complex with filtration. We will construct multi-parameter persistent homology in the similar way with each of the Fourier bases being the filtration coordinate.

\begin{construction}
\label{Def:multi-parameter}
For each parameter $\boldsymbol{\epsilon} = (\epsilon_1,\cdots,\epsilon_N) \in \mathbb{R}^N$, 
consider the product of Vietoris-Rips complexes
\[
\Rip_{\boldsymbol{\epsilon}}(\Psi_f)
:= \Rip_{\epsilon_1}(\pi_1 \Psi_f) \times \cdots \times \Rip_{\epsilon_N}(\pi_N \Psi_f).
\]
Then the pair
\(
\Bigl(
\{ H_{n}(\Rip_{\boldsymbol{\epsilon}}(\Psi_f)) \}_{\boldsymbol{\epsilon} \in \mathbb{R}^N},\;
\{ \iota_{*}^{\boldsymbol{\epsilon}, \boldsymbol{\epsilon}'} : H_{n}(\Rip_{\boldsymbol{\epsilon}}(\Psi_f)) 
\to H_{n}(\Rip_{\boldsymbol{\epsilon}'}(\Psi_f)) \}_{\boldsymbol{\epsilon} \le \boldsymbol{\epsilon}'}
\Bigr)
\)
is the $n$-dimensional multi-parameter persistent homology of $\Psi_f$,
where $\iota^{\boldsymbol{\epsilon},\boldsymbol{\epsilon}'}$ is the inclusion of complexes 
and $\iota_{*}^{\boldsymbol{\epsilon},\boldsymbol{\epsilon}'}$ is the induced map on homology.
\end{construction}

As we construct multi-parameter persistent homology, there are infinitely many one-dimensional reductions where persistent homology is calculated in the filtration space. One easy choice would be a line, which is defined with the direction vector with the origin (endpoint) vector. The following definition provides the definitions of the one-parameter reduction of multi-parameter persistent homology on a ray.

\begin{definition}
\label{one parameter reduction of multi-parameter persistent homology}
Let $\mathbf{a} = (a_{1},\cdots , a_{N})$ be the direction vector of a ray with each component $a_i > 0$, 
and let $\mathbf{b} = (b_1, \cdots , b_N) \in \mathbb{R}^{N}$ be the endpoint of the ray. 
For the ray $\ell(t) = \mathbf{b} + \sqrt{N}\, t \cdot \tfrac{\mathbf{a}}{\lVert \mathbf{a} \rVert}$ with $t \ge 0$ in the multi-parameter space,
define a filtration
\[
\Rip^{\ell}_{t}(\Psi_{f}) 
= \Rip_{\,b_1 + \sqrt{N} t \cdot \tfrac{a_1}{\lVert \mathbf{a} \rVert}}(\pi_1\Psi_{f})
\times \cdots \times
\Rip_{\,b_N + \sqrt{N} t \cdot \tfrac{a_N}{\lVert \mathbf{a} \rVert}}(\pi_N\Psi_{f}).
\]
This construction yields a one-parameter filtration with respect to $t$, for which the barcode is well defined. Denote the $n$-dimensional barcode of this filtration by $\mathsf{bcd}_{n}^{\Rip,\ell}(\Psi_{f})$.
\end{definition}

\subsection{Exact formula and interpretation of exact multi-parameter persistent homology}
\label{sec:EMPH on a ray}
In the Introduction, we explained the difficulty of dealing with multi-parameter persistent homology. Alternatively, we deal with fibered barcode, which is one-dimensional reduction of multi-parameter persistent homology. One-dimensional reduction involves the assignment of a ray characterized by both a direction vector and an endpoint. In this subsection, we derive the exact barcode formula on a ray and provide its interpretation. As shown below, the derived barcode implies that the choice of the direction vector can change the ratio of the considered modes, i.e. the weight of each mode resulting in the change of the barcode of the Liouville torus projected onto the $L$-plane, which consequently changes the overall barcode of the Liouville torus. 
Additionally, the endpoint vector can be used to assign threshold values of the Fourier modes.

\begin{proposition}[Stability theorem]
\label{stability theorem}
Let $(X_1, d_1^X), \cdots ,(X_k, d_k^X), (Y_1, d_1^Y), \cdots ,(Y_k, d_k^Y)$ be totally bounded metric spaces, $X=(X_1 \times \cdots \times X_k, d_{\max}^X)$ and $Y=(Y_1 \times \cdots \times Y_k, d_{\max}^Y)$ equipped with the maximum distance. For $\ell(t) = \mathbf{b} + \sqrt{k} t \cdot {\mathbf{a} \over \lVert \mathbf{a} \rVert}$, the following inequality holds. 
$$d_B\left(\mathsf{bcd}^{\Rip,\ell}_{n}(X),\mathsf{bcd}^{\Rip,\ell}_{n}(Y) \right) \le {2\lVert \mathbf{a} \rVert \over \sqrt{k}\min\limits_{L}a_L} \max\limits_{L}d_{GH}(X_L,Y_L)$$
\end{proposition}

\begin{proof}
We prove the theorem with a similar argument as in the proof of Lemma 4.3 in \cite{chazal2014persistence}.
Let $C \subseteq X \times Y$ be a correspondence. By Proposition \ref{product totally}, $X$ and $Y$ are also totally bounded metric spaces.  Therefore, $d_{GH}(X,Y)$ is finite. Let $\epsilon > 2\max\limits_{L}d_{GH}(X_L,Y_L)$. $\sigma \in \Rip^{\ell}_{t}(X) \implies$ for every $x = (x_1,\cdots,x_k)$ and $\tilde{x} = (\tilde{x}_1,\cdots,\tilde{x}_k) \in \sigma$, $d^{X}_{L}(x_L,\tilde{x}_L) \le b_L + \sqrt{k}t\cdot {a_L \over \lVert \mathbf{a} \rVert}$ for every $L=1,\cdots,k$. Let $Y' \subseteq Y$ be any finite subset such that for every $y \in Y'$, there is $x \in \sigma$ such that $(x,y) \in C$. For any $y=(y_1,\cdots,y_k)$ and $\tilde{y} = (\tilde{y}_1,\cdots,\tilde{y}_k) \in Y'$, we obtain $d^Y_{L}\left( y_L, \tilde{y}_L \right) \le d^X_{L}\left( x_L, \tilde{x}_L \right) + \epsilon \le b_L+\sqrt{k}t \cdot {a_L \over \lVert \mathbf{a} \rVert} + \epsilon \le b_L+\sqrt{k}\left(t + {\lVert \mathbf{a} \rVert \epsilon \over \sqrt{k}\min\limits_{L}a_L} \right) \cdot {a_L \over \lVert \mathbf{a} \rVert}$ for every $L$. This implies that $Y' \in \Rip^{\ell}_{t+{\lVert \mathbf{a} \rVert\epsilon \over \sqrt{k}\min\limits_{L}a_L}}(Y)$. By Proposition 4.2 in \cite{chazal2014persistence}, $d_B\left(\mathsf{bcd}^{\Rip,\ell}_{n}(X),\mathsf{bcd}^{\Rip,\ell}_{n}(Y) \right) \le {\lVert \mathbf{a} \rVert \epsilon \over \sqrt{k}\min\limits_{L}a_L}\rightarrow {2\lVert \mathbf{a} \rVert \over \sqrt{k}\min\limits_{L}a_L} \max\limits_{L}d_{GH}(X_L,Y_L)$ as $\epsilon \rightarrow 2\max\limits_{L}d_{GH}(X_L,Y_L)$.
\end{proof}

\begin{lemma}[General version of persistent K\"unneth formula \uppercase\expandafter{\romannumeral2}]
\label{multi Persistent Kunneth formula}
Let $(X_1,d_{X_1}),\cdots, (X_k,d_{X_k})$ be totally bounded metric spaces. Then,
\begin{equation}
\mathsf{bcd}_{n}^{\Rip, \ell}(X_1 \times \cdots \times X_k , d_{\max}) = \left\{ J^{n_1,\ell}_1 \cap \cdots \cap J^{n_k,\ell}_k : J^{n_j,\ell}_j \in \mathsf{bcd}_{n_j}^{\Rip,\ell}(X_j,d_{X_j}) \ \text{and} \ \sum\limits_{j=1}^{k} n_j = n \right\}\end{equation}
for all $n\in\mathbb{Z}_{\ge 0}$ and $d_{\max}$ is the maximum metric.
\end{lemma}

\begin{proof}
For each $j$, let $X_j^r$ be a finite subset of $X_j$ such that $d_{GH}(X_j,X_j^r)<r$, 
whose existence is guaranteed by Lemma~\ref{totally bounded Persistent Kunneth formula}. 
Let $X = X_1 \times \cdots \times X_k$ and $X^r = X_1^r \times \cdots \times X_k^r$. 
For each $t \ge 0$, the filtration along the ray $\ell$ satisfies 
$\Rip_{t}^{\ell}(X^r) 
= \Rip_{b_1 + \sqrt{k}t \cdot \tfrac{a_1}{\lVert \mathbf{a} \rVert}}(X_1^r) 
\times \cdots \times 
\Rip_{b_k + \sqrt{k}t \cdot \tfrac{a_k}{\lVert \mathbf{a} \rVert}}(X_k^r)$. 
Thus the filtration $\Rip^{\ell}(X^r)$ is the product filtration 
$\Rip^{\ell}(X_1^r)\times \cdots \times \Rip^{\ell}(X_k^r)$. 
By Theorem~\ref{thm:persistent-kunneth}, the barcode of $\Rip^{\ell}(X^r)$ satisfies the K\"unneth formula.  
Finally, by Proposition~\ref{stability theorem}, we have 
$d_B\!\left(\mathsf{bcd}^{\Rip,\ell}_{n}(X),\; \mathsf{bcd}^{\Rip,\ell}_{n}(X^r)\right) 
\le \frac{2\lVert \mathbf{a} \rVert}{\sqrt{k}\,\min\limits_{L} a_L}
\max\limits_{L} d_{GH}(X_L,X_L^r)
\le \frac{2r\lVert \mathbf{a} \rVert}{\sqrt{k}\,\min\limits_{L} a_L}$. 
Letting $r \to 0$ yields the desired formula for totally bounded metric spaces.
\end{proof}

\begin{theorem}[Exact Multi-parameter Persistent Homology (EMPH)]
\label{direction barcode}
Consider a ray $\ell$ in a filtration space whose direction vector is $\mathbf{a} = (a_{1},\cdots , a_{N})$ with each component $a_i > 0$, and whose endpoint is $\mathbf{b} = (b_1, \cdots , b_N)$. Then we have \begin{equation}
\label{eq:multi barcode}
\mathsf{bcd}_{n}^{\Rip,\ell}( \Psi_{f}) = \left\{ J_{1}^{n_1, \ell}\bigcap \cdots \bigcap J_{N}^{n_N, \ell} : J_{L}^{n_L, \ell} \in  \mathsf{bcd}_{n_L}^{\Rip,\ell} 	\left( \pi_L \Psi_{f}\right)  \ \text{and} \ \sum\limits_{L=1}^{N} n_L = n \right\}
\end{equation}
i.e.
$$J_{L}^{n, \ell} = \begin{cases} \left({-b_L \over \sqrt{N}a_L / \lVert \mathbf{a} \rVert},\infty\right), & \mbox{if }n=0 \\ \left({2r_L^f \sin\left(\pi {k \over 2k+1}\right) - b_L \over \sqrt{N}a_L / \lVert \mathbf{a} \rVert}  , {2r_L^f \sin\left(\pi {k+1 \over 2k+3}\right) - b_L \over \sqrt{N}a_L / \lVert \mathbf{a} \rVert}\right], & \mbox{if }n=2k+1 \\ \ \emptyset, & \mbox{otherwise}
\end{cases}. $$
\end{theorem}

\begin{proof}
Note that $\Rip_{t}^{\ell}\left( \pi_L \Psi_{f} \right) = \Rip_{b_L + \sqrt{N}t \cdot {a_L \over \lVert \mathbf{a} \rVert}}\left( \pi_L \Psi_{f} \right) = \Rip_{b_L + \sqrt{N}t \cdot {a_L \over \lVert \mathbf{a} \rVert}}\left( r_L^f \cdot \mathbb{S}^1\right)$. Theorem \ref{barcode of circle} implies
$$\mathsf{bcd}_{n}^{\Rip,\ell}(\pi_{L} \Psi_{f})=
\begin{cases}
\left\{ \left({-b_L  \over \sqrt{N}a_L / \lVert \mathbf{a} \rVert},\infty\right) \right\}, & \mbox{if }n=0 \\
\left\{ J_{L}^{2k+1, \ell} \right\}, & \mbox{if }n=2k+1 \\
\emptyset, &\mbox{ otherwise}
\end{cases},$$
where 
$$J_{L}^{2k+1, \ell} := \left( \min\limits_{t} b_L + \sqrt{N}t \cdot {a_L \over \lVert \mathbf{a} \rVert} \ge 2r_L^f \sin\left(\pi{k \over 2k +1}\right) , \min\limits_{t}b_L + \sqrt{N}t \cdot {a_L \over \lVert \mathbf{a} \rVert} \ge 2r_L^f \sin\left(\pi {k+1 \over 2k+3}\right) \right]. 
$$
And by Lemma \ref{multi Persistent Kunneth formula}, we obtain 
\begin{eqnarray}
\mathsf{bcd}_{n}^{\Rip,\ell}( \Psi_{f}) &=& \mathsf{bcd}_{n}^{\Rip,\ell}	\left( \pi_1 \Psi_{f} \times \cdots \times \pi_N \Psi_{f} \right) 
\nonumber \\
&=& \Bigl\{ J_{1}^{n_1, \ell}\bigcap \cdots \bigcap J_{N}^{n_N, \ell} : J_{L}^{n_L, \ell} \in  \mathsf{bcd}_{n_L}^{\Rip,\ell} 	\left( \pi_L \Psi_{f}\right)  \ \text{and} \ \sum\limits_{L=1}^{N} n_L = n \Bigr\}.\nonumber 
\end{eqnarray}
\end{proof}

\begin{corollary}
\label{general one-parameter}
The exact formula implies that the one-dimensional reduction of multi-parameter persistent homology of the given time-series data in the diagonal ray is equivalent to the usual one-parameter persistent homology of the time-series data, i.e. if $\mathbf{a}= (1,\cdots,1)$ and $\mathbf{b} = (0,\cdots,0)$. Then $\mathsf{bcd}_{n}^{\Rip,\ell}( \textcolor{black}{\Psi_{f}}) = \mathsf{bcd}_{n}^{\Rip}( \Psi_{f})$.
\end{corollary}

Figure \ref{direction vector} shows the ray $\ell$ with the direction vector $\mathbf{a}$ and the endpoint vector $\mathbf{b}$. The diagonal ray (or the standard ray) is the ray with $\mathbf{a} = (1, \cdots, 1)$ and $\mathbf{b} = (0, \cdots, 0)$. 
\begin{figure}[hbt!]
    \centering
\includegraphics{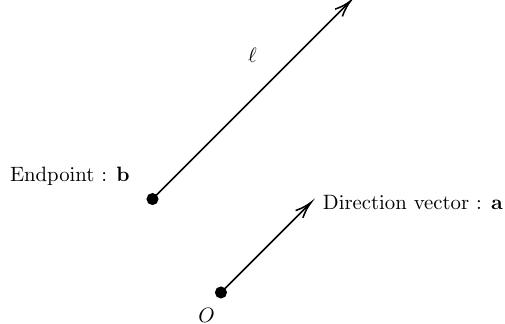}
    \caption{A ray in a multi-parameter space}
    \label{direction vector}
\end{figure}

\begin{corollary}
\label{multi barcode meaning}
If $\mathbf{b} = \mathbf{0}$, each bar in $\mathsf{bcd}_{n}^{\Rip,\ell}(\Psi_{f})$ represents the bar of the projected point cloud onto $P_{i_1} \oplus \cdots \oplus P_{i_k}$ for $k=1,\cdots,n$. That is, $\mathsf{bcd}_{n}^{\Rip,\ell}(\Psi_{f}) = \bigcup\limits_{1\le i_1 < \cdots <i_n \le N} \bigcup\limits_{1 \le k \le n}\mathsf{bcd}_{n}^{\Rip,\ell}	\left( \pi_{i_1 \cdots i_k} \Psi_{f} \right)$.
\end{corollary}

\begin{proof}
It can be proved similarly to Theorem \ref{barcode meaning}.
\end{proof}

In \eqref{eq:multi barcode}, the barcode formula is determined by the simple relation involving Fourier coefficients. The following example shows the difference between the method that simply uses  Fourier coefficients and the method that uses $\mathsf{bcd}_{n}^{\Rip,\ell}(\Psi_{f})$.
It also shows the benefit of considering $\mathsf{bcd}_{n}^{\Rip,\ell}(\Psi_{f})$.

{\color{black}
\begin{example}
\label{ex:benefit}
Consider the following two time-series data

\begin{equation}
  \left\{  \begin{array}{l}  f_1  = \cos t + 1\\
      f_2 = \cos 5t 
      \end{array}
      \right. \nonumber .  
\end{equation}

Let $N = 5$. The Fourier coefficients of $f_1$ and $f_2$ are  $(1,1,0,0,0,0)$ and $(0,0,0,0,0,1)$ respectively. Dynamic time warping \cite{sakoe1978dynamic} measures the difference between two time-series data. The simple Fourier method and the dynamic time warping method judge $f_1$ and $f_2$ to be different, but they do not provide a perspective on how we can regard them as the same. 

The sliding window embedding method can regard $f_1$ and $f_2$ as the same if we choose $M=2$ and $\tau={2\pi \over 3}$, and distinguish them by choosing $M=1$ and $\tau={\pi \over 10}$. More explicitly, for $M=2$ and $\tau={2\pi \over 3}$, in Proposition \ref{embedding dimension}, $u_1 = \left( 1, -{1\over 2}, -{1\over 2} \right)$, $v_1 = \left( 0, {\sqrt{3} \over 2}, -{\sqrt{3} \over 2} \right)$, and $u_5 = \left( 1, -{1 \over 2}, -{1 \over 2} \right)$, $v_5 = \left( 0, -{\sqrt{3} \over 2}, {\sqrt{3} \over 2} \right)$. Therefore, $SW_{M,\tau} f_1$ and $SW_{M,\tau} f_2$ become circles with radius $\sqrt{3}$. Similarly, for $M=1$ and $\tau={\pi \over 10}$, $u_1 = \left( 1, \cos  {\pi \over 10} \right)$, $v_1 = \left( 0, \sin  {\pi \over 10}  \right)$, and $u_5 = \left( 1, 0 \right)$, $v_5 = \left( 0, 1 \right)$. Since $SW_{M,\tau} f_1$ is an ellipse and $SW_{M,\tau} f_2$ is a circle, their Vietoris-Rips barcodes are not the same. {\color{black}However, due to the lack of a barcode formula for sliding window embedding, it is difficult to adjust $M$ and $\tau$ to obtain the desired results in general. And the sliding window embedding method becomes computationally expensive as the number of Fourier modes increases.}

On the other hand, with the analysis of the Liouville torus, we have the exact barcode formula, which can overcome the shortcomings of the sliding window embedding. If we choose the standard ray $\ell_1$, that is, if the direction vector is $(1, 1, 1, 1, 1)$ and the endpoint is $(0,0,0,0,0)$, we obtain 

$$
\mathsf{bcd}_{1}^{\Rip, \ell_1}(\Psi_{f_1}) = 
\mathsf{bcd}_{1}^{\Rip, \ell_1}(\Psi_{f_2}) = \left\{ \left(0,\sqrt{3}\right] \right\}.
$$
In this sense, we can regard those two are the same. However, if we choose a non-standard ray $\ell_2$, e.g. the ray with the direction vector of  $(1,10^{-6},10^{-6},10^{-6},10^{-6})$ and the endpoint of  $(0,0,0,0,0)$, then we obtain the result that shows the difference between $f_1$ and $f_2$ as  
$$
\mathsf{bcd}_{1}^{\Rip,\ell_2}(\Psi_{f_1}) = \left\{\left(0,\sqrt{{3 \over 5}} \cdot \sqrt{1+4\cdot 10^{-12}}\right]\right\}$$
and 
$$
\mathsf{bcd}_{1}^{\Rip,\ell_2}(\Psi_{f_2}) = \left\{\left(0,\sqrt{{3 \over 5}} \cdot 10^6 \cdot \sqrt{1+4\cdot 10^{-12}}\right]\right\}. 
$$ 

Here note that we used a small value, such as $10^{-6}$, in the direction vector $(1,10^{-6},10^{-6},10^{-6},10^{-6})$ instead of zero. This is done in order to 
to avoid the case that the denominator of the exact barcode formula in Theorem \ref{direction barcode} vanishes. 
\end{example}
}

\begin{example}
\label{ex:choice}
Consider the time-series data used in the previous example, i.e.,
\(
f(t) = \cos t + \tfrac{1}{2} \cos 2t.
\)
In this case, we have two non-trivial parameters corresponding to the modes $1$ and $2$, which define a two-dimensional filtration space. First, suppose that we choose the \textit{diagonal ray} $\ell_{1}$ with direction vector $\mathbf{a} = (1,1)$
and the endpoint $\mathbf{b}=(0,0)$.
Using the exact formula, we obtain
\[
\mathsf{bcd}_{1}^{\Rip,\ell_1}(\Psi_{f}) 
= \left\{ \left(0, \sqrt{3}\right], \left(0, \tfrac{\sqrt{3}}{2}\right] \right\}.
\]

Next, for another ray $\ell_{2}$ with direction vector $\mathbf{a} = (2,1)$ and the endpoint $\mathbf{b}=(0,0)$,
the corresponding one-dimensional barcode is
\[
\mathsf{bcd}_{1}^{\Rip,\ell_{2}}(\Psi_{f}) 
= \left\{ \left(0, \tfrac{\sqrt{30}}{4}\right]_{(2)} \right\}.
\]

Figure~\ref{filtrationspace} schematically illustrates the barcodes restricted to these two rays.
In the left figure (for $\ell_{1}$), the two red solid lines represent the two bars in 
$\mathsf{bcd}_{1}^{\Rip,\ell_{1}}(\Psi_{f})$.
In the right figure (for $\ell_{2}$), the barcode consists of two identical bars.

\begin{figure}[hbt!]
    \centering
    \includegraphics{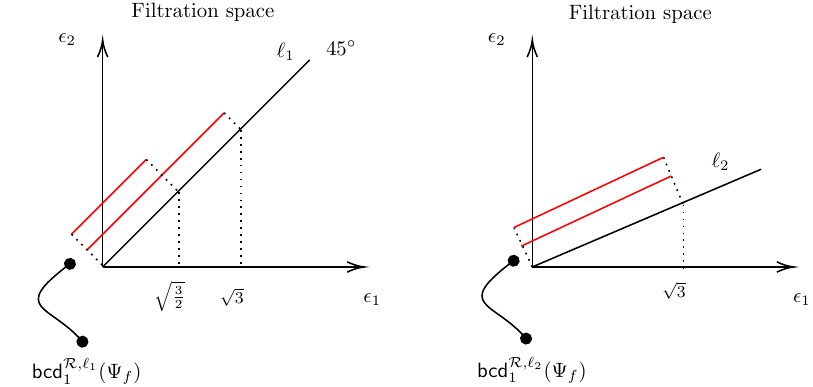}
    \caption{
Schematic illustration of the barcodes restricted to two different rays. 
The left figure corresponds to the ray $\ell_{1}$ with direction vector $\mathbf{a}=(1,1)$, 
where the two red solid lines in the filtration space represent the two bars in 
$\mathsf{bcd}_{1}^{\Rip,\ell_{1}}(\Psi_{f})$. 
The right figure corresponds to the ray $\ell_{2}$ with direction vector $\mathbf{a}=(2,1)$. 
Both cases are computed for $f(t)=\cos t+\tfrac{1}{2}\cos 2t$ with $\mathbf{b}=(0,0)$.
}

\label{filtrationspace}
\end{figure}

This example shows that choosing different rays may result in different barcode and provide different topological inference. In the following section, we will illustrate this observation with some real data.

\end{example}

\begin{example}
\label{motivation sec 4.3}
Let $f = \cos t + 20\cos 2t$ and $g = 2\cos t + 20\cos 2t$. Consider the ray $\ell_1$ with $\mathbf{a} = (1,1)$ and $\mathbf{b_1} = (\sqrt{3},0)$. Then $\mathsf{bcd}_{1}^{\Rip,\ell_1}(\Psi_{f}) = \left\{\left(0, 20 \sqrt{3}\right] \right\}$ and $\mathsf{bcd}_{1}^{\Rip,\ell_1}(\Psi_{g}) = \left\{\left(0,\sqrt{3}\right], \left(0, 20 \sqrt{3}\right] \right\}$. For the ray $\ell_2$ with $\mathbf{a} = (1,1)$ and $\mathbf{b_2} = (2\sqrt{3},0)$, we have $\mathsf{bcd}_{1}^{\Rip,\ell_2}(\Psi_{f}) = \left\{\left(0,20\sqrt{3}\right]\right\}$ and $\mathsf{bcd}_{1}^{\Rip,\ell_2}(\Psi_{g}) = \left\{\left(0,20\sqrt{3}\right]\right\}$. In this example, we can observe the followings.
\begin{enumerate}[(i)]
    \item If $\mathbf{b} \ne \mathbf{0}$, $\mathsf{bcd}_{1}^{\Rip,\ell}(\Psi_{f})$ may not necessarily be equal to $\mathsf{bcd}_{1}^{\Rip,\ell}(\pi_1 \Psi_{f}) \bigcup \mathsf{bcd}_{1}^{\Rip,\ell}(\pi_2 \Psi_{f})$ as in Example \ref{onedimension persistent}. \item We can neglect unnoticeable Fourier mode using the two rays $\ell_1$ and $\ell_2$ (the bar by the Fourier $1$-mode vanishes on $\ell_1$, but the bar by the $2$-mode does not vanish on $\ell_1$ and $\ell_2$). We can regard $f$ and $g$ as the same if we consider $\mathsf{bcd}_{1}^{\Rip,\ell_2}(\Psi_{f})$ and $\mathsf{bcd}_{1}^{\Rip,\ell_2}(\Psi_{g})$. A similar observation is made in Example \ref{twodimension persistent}: the prominent Fourier modes could be neglected in the inference by using the 2-dimensional barcode. This infers that varying the endpoint of the ray allows for the establishment of a threshold for each Fourier mode.
\end{enumerate}    

\end{example}

One of the advantages of the proposed method is that we can easily compute persistent homology on a curve in the filtration space with the exact barcode formula. Note that it is hard or impossible, in general, to compute persistent homology along a curved in the filtration space with arbitrary parameters. However, by using the Fourier bases as parameters for the filtration space and with the complete knowledge of the exact barcode in a line segment, we can easily estimate persistent homology in a curved in the filtration space. This method provides a high flexibility of choosing various rays and is useful in real applications. In our future work, we will further investigate multi-parameter persistent homology in curved rays. The following remark is on persistent homology in a curved ray.

\begin{remark}[Curved filtration in the multi-parameter space]
With the proposed exact method, it is possible to compute a curved filtration in the multi-parameter filtration space. In Theorem \ref{direction barcode}, we mentioned that the direction vector is related to the weights of frequencies. A curved ray means time (filtration parameter) varying weights of frequencies. With this, consider the following situation where we  regard both $\cos t$ and $\cos 2t$ to be the same while we consider  $2\cos t$ and $2\cos 2t$ to be different. More precisely, consider a curve 
$$c(t) = \begin{cases}
(t,t), & \mbox{if }0\le t \le \sqrt{3} \\
\left(t, {1 \over \sqrt{3}}t + \sqrt{3} -1\right), & \mbox{if }t \ge \sqrt{3}
\end{cases}.
$$
\textcolor{black}{Let $f(t) = \cos t$ and $g(t) = \cos 2t$.} Then, we have 
\begin{eqnarray}
\mathsf{bcd}_{1}^{\Rip,c}(\Psi_f) &=& \left\{ \left( 0, \sqrt{3} \right] \right\}, \nonumber \\ \mathsf{bcd}_{1}^{\Rip,c}(\Psi_g)) &=& \left\{ \left( 0, \sqrt{3} \right] \right\}, \nonumber \\
\mathsf{bcd}_{1}^{\Rip,c}(\Psi_{2f})) &=& \left\{ \left( 0, \sqrt{3} + \sqrt{2} \right] \right\}
\nonumber \\
\mathsf{bcd}_{1}^{\Rip,c}(\Psi_{2g})) &=& \left\{ \left( 0, \sqrt{3} + \sqrt{6} \right] \right\}. \nonumber 
\end{eqnarray}

\begin{figure}[hbt!]
    \centering
    \includegraphics[width=1 \linewidth]{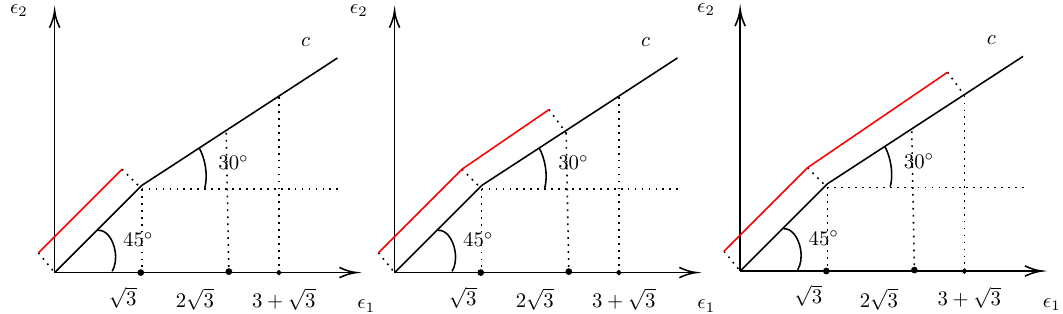}
     \caption{Curved rays in the filtration space. Left : $\cos t$ and $\cos 2t$, middle : $2\cos t$, Right : $2\cos 2t$}
    \label{curve in parameter}
\end{figure}
As shown above, the first two barcodes are exactly the same while the last two barcodes are different. Figure \ref{curve in parameter} shows the rays and the corresponding barcodes. In the left figure, the barcodes for $\cos t$ and $\cos 2t$ are shown while the middle and right figures show those for $2\cos t$ and $2\cos 2t$. The figure shows how those are distinguished on a curved ray in the filtration space. 
\end{remark}

{\color{black}\begin{remark}[Computational complexity of $\mathsf{bcd}_{n}^{\Rip,\ell}(\Psi_{f})$]
\label{remark:computational complexity}
Let $T$ be the length of time-series data, $n$ be the dimension of persistent homology, and $N$ be the degree of the truncated trigonometric polynomial. The complexity associated with the Fourier transform is $O(T\log T)$, and computation of all $J_{L}^{n_L, \ell}$ is $O(n \cdot N)$. The computational complexity of the  intersection of elements $J_{1}^{n_1, \ell}\bigcap \cdots \bigcap J_{N}^{n_N, \ell}$ is $O(N)$\footnote{The complexity of the intersection of two sets is $O(1)$. In the actual calculation, we find the maximum/minimum of birth/death time of $J_{L}^{n_L, \ell}$ for $1 \le L \le N$. The former becomes the birth time of $J_{1}^{n_1, \ell}\bigcap \cdots \bigcap J_{N}^{n_N, \ell}$, and the latter becomes the death time of $J_{1}^{n_1, \ell}\bigcap \cdots \bigcap J_{N}^{n_N, \ell}$. However, the complexity remains unchanged, as the complexity of the minimum/maximum operator is $O(N)$. \\ Source: \url{https://ics.uci.edu/~pattis/ICS-33/lectures/complexitypython.txt}.},  and we repeat this operation until we cover all cases of $\sum\limits_{L=1}^{N} n_L = n$. Therefore, the total computational complexity of $\mathsf{bcd}_{n}^{\Rip,\ell}(\Psi_{f})$ is $O(T\log T) + O(n \cdot N) + O\left(N \times \binom{N+n-1}{n}\right) = O(T\log T) + O\left(N \times \binom{N+n-1}{n}\right)$, where $\binom{N+n-1}{n}$ represents a combination. For example, for $n=1$, $O(T \log T) + O(N^2) \le O(T^2)$, and for $n=2$, $O(T \log T) + O(N^3) \le O(T^3)$.
\end{remark}}

\subsection{Exact multi-parameter persistent homology on a collection of rays}
\label{sec:EMPH on a collection of rays}
In the multi-parameter setting, while we have discussed the barcode $\mathsf{bcd}_{n}^{\Rip,\ell}(\Psi_f)$ obtained along a single ray, a natural follow-up question is how to integrate the information obtained from multiple rays. Example~\ref{motivation sec 4.3} demonstrated that short bars can be eliminated by adjusting the endpoints. In this section, as an example of utilizing barcodes obtained from multiple rays, we show that combining these barcodes allows for a parallel translation of the birth and death times of the bars.

\begin{definition}
Let $\mathcal{L} = 	\left\{\ell_1, \cdots, \ell_s \right\}$ be a collection of rays. Define
$$\mathsf{bcd}_{n}^{\Rip,\mathcal{L}}(\Psi_{f}) = \bigcup\limits_{i=1}^{s} \mathsf{bcd}_{n}^{\Rip,\ell_i}(\Psi_{f}).$$
\end{definition}

We now show that, for a ray \(\ell\) with direction vector \(\mathbf{a}=(a_1,\cdots,a_N)\) and endpoint \(\mathbf{b}=(b_1,\cdots,b_N)\), one can assign distinct birth time thresholds to each Fourier mode. To this end, we select an associated collection of rays and consider the union of their barcodes.

\begin{theorem}
\label{thm:collection of ray}
Let \(\ell\) have direction vector \(\mathbf{a}=(a_1,\cdots,a_N)\) and endpoint \(\mathbf{b}=(b_1,\cdots,b_N)\). 
Consider the collection of rays \(\mathcal{L}_{\ell}=\{\ell_1,\cdots,\ell_N\}\) associated with \(\ell\), 
where each \(\ell_L\) has direction \(\mathbf{a}^L=(a_1^L,\cdots,a_L,\cdots,a_N^L)\) and endpoint 
\(\mathbf{b}^L=(b_1^L,\cdots,b_L,\cdots,b_N^L)\), satisfying \(\|\mathbf{a}^L\|=\|\mathbf{a}\|\) and 
\(\max\limits_{i\ne L} \tfrac{r_L^f\sqrt{3} - b^L_i}{\sqrt{N}a^L_i / \lVert \mathbf{a}^L \rVert} \le 
\tfrac{-b_L}{\sqrt{N}a_L/ \lVert \mathbf{a}^L\rVert}\) for \(L=1,\cdots,N\). Then the EMPH on the collection of rays $\mathcal{L}_{\ell}$ associated with \(\ell\) is given by
\begin{equation}
\label{Eq:multiple lines}
\begin{aligned}
\mathsf{bcd}_{0}^{\Rip,\mathcal{L}_{\ell}}(\Psi_{f}) 
   &= \left\{\left(\tfrac{-b_L}{\sqrt{N} a_L / \lVert \mathbf{a} \rVert},\infty\right) : L=1,\cdots, N \right\}, \\[0.4em]
\mathsf{bcd}_{1}^{\Rip,\mathcal{L}_{\ell}}(\Psi_{f}) 
   &= \Bigl\{ I_{1}^{\ell}, \cdots, I_{N}^{\ell} : 
       I_{L}^{\ell} =\Bigl( \tfrac{-b_L}{\sqrt{N} a_L / \lVert \mathbf{a} \rVert}, 
                            \tfrac{r_L^f \sqrt{3} - b_L}{\sqrt{N} a_L / \lVert \mathbf{a} \rVert} \Bigr] \Bigr\}.
\end{aligned}
\end{equation}
\end{theorem}

\begin{proof}
From Eq.~(\ref{eq:multi barcode}), the EMPH along a single ray \(\ell_L\) is
\(
\mathsf{bcd}_{0}^{\Rip,\ell_L}(\Psi_{f}) 
   = \left\{ \bigcap\limits_{i=1}^N \left( \tfrac{-b_i^L}{\sqrt{N} a_i^L / \lVert \mathbf{a} \rVert},\infty \right) \right\} 
   = \Bigl\{ \left(\tfrac{-b_L}{\sqrt{N} a_L / \lVert \mathbf{a} \rVert},\infty \right) \Bigr\}
\)
and
\(
\mathsf{bcd}_{1}^{\Rip,\ell_L}(\Psi_{f}) 
   = \left\{ \left( \tfrac{-b_L}{\sqrt{N} a_L / \lVert \mathbf{a} \rVert}, 
                     \tfrac{r_L^f \sqrt{3} - b_L}{\sqrt{N} a_L / \lVert \mathbf{a} \rVert} \right] \right\}.
\)
By definition, the EMPH on the collection \(\mathcal{L}_\ell\) is obtained by taking the union of these barcodes over all \(L\), which completes the proof.
\end{proof}

\begin{example}
Suppose $f(t) = \cos t + 20\cos 2t$.  
Consider the ray $\ell$ with direction vector $\mathbf{a} = (1,1)$ and endpoint $\mathbf{b} = (\sqrt{3},0)$.  
Then one possible choice of the associated rays is
\[
\ell_1: \ \mathbf{a}^{1} = (1,1), \ \mathbf{b}^{1} = \left(\sqrt{3},2\sqrt{3}\right), 
\qquad
\ell_2: \ \mathbf{a}^{2} = (1,1), \ \mathbf{b}^{2} = \left(20\sqrt{3},0\right),
\]
which satisfy the hypothesis of Theorem~\ref{thm:collection of ray}.  
Figure~\ref{Fig_multiple_lines} illustrates the rays $\ell, \ell_1,$ and $\ell_2$.  
For the original ray $\ell$, we obtain
\[
\mathsf{bcd}_{1}^{\Rip,\ell}(\Psi_{f}) = \left\{\left(0, 20\sqrt{3}\right] \right\} \quad \text{(Example~\ref{motivation sec 4.3})}.
\]
In contrast, for the collection of rays associated with $\ell$, we obtain
\[
\mathsf{bcd}_{1}^{\Rip,\mathcal{L}_{\ell}}(\Psi_{f}) = \left\{\left(-\sqrt{3},0\right], \left(0, 20\sqrt{3}\right]\right\}.
\]
Thus, considering the collection of rays associated with $\ell$ corresponds not to discarding short bars as in Example~\ref{motivation sec 4.3}, but rather to a parallel translation of these bars.

\begin{figure}[hbt!]
    \centering
    \includegraphics[scale=0.7]{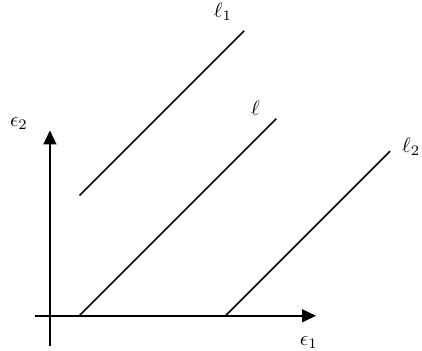}
    \caption{A ray $\ell$ with $\mathbf{a}=(1,1)$, $\mathbf{b}=(\sqrt{3},0)$ and its associated rays 
$\ell_{1}, \ell_{2}$.}

    \label{Fig_multiple_lines}
\end{figure}
\end{example}

If the endpoint vector is zero, Theorem~\ref{thm:collection of ray} reduces to the single ray case, as summarized in the following remark.
\begin{remark}
\label{for experiment}
If the endpoint of $\ell$ is zero vector, then $\mathsf{bcd}_{1}^{\Rip,\mathcal{L}_{\ell}}(\Psi_{f})= \mathsf{bcd}_{1}^{\Rip,\ell}(\Psi_{f})$. 
\end{remark}

\end{section}

\begin{section}{EMPH in machine learning workflows}
\label{Experiment}
The  EMPH method proposed in Section \ref{Sec : Application of multi-parameter theory and its interpretation} is highly efficient in terms of computational complexity and it provides a framework for variable topological inferences. 
Due to its significantly lower computational complexity compared to existing TDA methods and its capability for variable topological inferences, the EMPH can be implemented much more efficiently in machine learning workflows, thereby easily enabling topological inferences in such workflows. 

In order to implement TDA in a machine learning workflow, we need to vectorize the barcode. Persistence Landscape \cite{bubenik2015statistical} and persistence image \cite{adams2017persistence} have emerged as typical methods for transforming barcodes into vectors, serving as essential tools for combining TDA with machine learning. 
To apply the EMPH method in a machine learning workflow, we employ the following vectorization methods, i.e.  Betti sequence and persistence image.
\begin{enumerate}
    \item \textbf{(Betti sequence, \cite{umeda2017time})} Given an interval $I \subset \mathbb{R}$, consider $\epsilon_1 \le \epsilon_2 \le \cdots \le \epsilon_k$ representing equally spaced points within $I$. The $n$-dimensional Betti sequence $BS_n(B)$ of $n$-dimensional barcode $B$ is a $k$-dimensional vector defined as follows: $BS_n(B) = \bigl( \left| \left\{ [b,d) \in B : \epsilon_i \in [b,d) \right\} \right| \bigr)_{1 \le i \le k} \in \mathbb{R}^k$. This vectorization method discretizes the variation of the Betti numbers of simplicial complexes (e.g., $\Rip_{\epsilon}(X)$) with the  filtration parameter (e.g., $\epsilon$). The computation of $BS_n(B)$ is straightforward.

    \item \textbf{(Persistence image, \cite{adams2017persistence})} Let $\omega : \mathbb{R} \rightarrow \mathbb{R}$ be a weight function  and $\sigma \in \mathbb{R}$ a smoothing parameter. Given a square $S \subset \mathbb{R}^2$, consider equally subdivided squares $\left(S_{lm}\right)_{1 \le l,m \le k}$ of $S$. The $n$-dimensional persistence image $V_n(B)$ of  $n$-dimensional barcode $B$ is a $k^2$-dimensional vector defined as follows: $V_n(B) = \Bigl(\int_{(x,y) \in S_{lm}} \sum\limits_{[b,d) \in B} \omega(d-b) \cdot {1 \over 2\pi \sigma^2} e^{-\left[ (x-b)^2 + (y-(d-b))^2 \right] / 2\sigma^2} \ dxdy\Bigr)_{1\le l,m \le k} \in \mathbb{R}^{k^2}$. Here, the exponential term contributes to the stability of vectorization, indicating that a small perturbation of the barcode with respect to the bottleneck distance induces a small perturbation of its persistence image. In the following examples, we choose a weight function as $\omega(d-b) = d-b$, meaning that we assign a weight proportional to the ``persistence'' of homology.
\end{enumerate}

In practical cases, time-series data are discrete and finite, and we consider the domain of a time-series to be $\{0,\cdots,T-1\}$. The Fourier coefficients required for the  barcode computation are obtained using the fast Fourier transform. Algorithm \ref{alg:main} illustrates the simple EMPH usage workflow for time-series data analysis.

\begin{algorithm}[H]
\begin{algorithmic}[1]
\caption{EMPH of time-series data in a machine learning workflow}
\label{alg:main}

\STATE \textbf{Input:} $f_1, \cdots ,f_m : \left\{ 0, \cdots, T-1 \right\} \rightarrow \mathbb{R}$ (time-series data)
\vspace{0.1cm}
\STATE \textbf{Variables:}
$N \in \mathbb{Z}_{\ge 0}$ (degree of truncated Fourier series), $n \in \mathbb{Z}_{\ge 0}$ (dimension of barcode), $\ell = \left(\mathbf{a}, \mathbf{b}\right)$ (ray with direction vector $\mathbf{a}$ and endpoint $\mathbf{b}$) and $r \in \mathbb{Z}_{>0}$ (resolution for vectorization, i.e., vector dimension);

\vspace{0.1cm}

\FOR{$k = 1,\cdots ,m$}
\STATE Calculate the Fourier transforms $\hat{f}_k$ using the fast Fourier transform.

\STATE Calculate $\mathsf{bcd}_{n}^{\Rip,\ell}(\Psi_{f_k})$ (or $\mathsf{bcd}_{n}^{\Rip,\mathcal{L}_{\ell}}(\Psi_{f_k}))$ using Theorem \ref{direction barcode} (or Theorem~\ref{thm:collection of ray}).

\STATE Vectorize $\mathsf{bcd}_{n}^{\Rip,\ell}(\Psi_{f_k})$ (or $\mathsf{bcd}_{n}^{\Rip,\mathcal{L}_{\ell}}(\Psi_{f_k}))$ with resolution $r$.

\ENDFOR

\STATE
Integrate the vectorization of $\mathsf{bcd}_{n}^{\Rip,\ell}(\Psi_{f_k})$ into various machine learning techniques. 

\STATE \textbf{Ouput:}
Topological inferences for problems such as classification, clustering, etc. 

\end{algorithmic}
\end{algorithm}

\medskip

In Example~\ref{example 2}, we apply EMPH to a single ray, namely 
$\mathsf{bcd}_{n}^{\Rip,\ell}(\Psi_{f_k})$, following the setting of 
Section~\ref{sec:EMPH on a ray}. As noted in Remark~\ref{for experiment}, 
when the endpoint of $\ell$ is the zero vector we have 
$\mathsf{bcd}_{n}^{\Rip,\mathcal{L}_{\ell}}(\Psi_{f_k})
= \mathsf{bcd}_{n}^{\Rip,\ell}(\Psi_{f_k})$, 
so Example~\ref{example 2} can be regarded as a special case from the perspective of a collection of rays.  
In contrast, Examples~\ref{example 3} and \ref{example 4} also consider 
nonzero endpoints and perform experiments on both 
$\mathsf{bcd}_{n}^{\Rip,\ell}(\Psi_{f_k})$ and 
$\mathsf{bcd}_{n}^{\Rip,\mathcal{L}_{\ell}}(\Psi_{f_k})$ 
as defined in Sections~\ref{sec:EMPH on a ray} and 
\ref{sec:EMPH on a collection of rays}.  
For all subsequent examples, we fix the barcode dimension at $n=1$. To facilitate reproducibility, the implementation of the proposed EMPH methods is publicly available at \url{https://github.com/KeunsuKim/EMPH}.

\begin{example} 
\label{example 2}
In this example, we address the clustering problem of four different shapes (circle, square, star, and triangle) available in Kaggle\footnote{https://www.kaggle.com/datasets/smeschke/four-shapes}. To apply EMPH to this problem, we first transform the image data into time-series data. Specifically, we compute the center of each image using the \texttt{Scipy} library and extract its contour with the \texttt{scikit-image} library. From the extracted contour, we generate a time-series by rotating around the center at $1^\circ$ intervals and measuring the distance between the center and the contour, followed by normalization. Since the distances are uniformly sampled by angle, all time-series have a fixed length of $360$. Figure~\ref{Four shape} shows sample images of the four shapes and their corresponding time-series representations. In this study, we apply EMPH to the transformed time-series and show that different choices of rays can lead to distinct topological interpretations. In particular, this method can be used to reveal the existence of potential subclusters depending on the shape, and we experimentally demonstrate that differences between such subclusters can be quantitatively identified at both the image and time-series levels. For the clustering experiment, we use the one-dimensional barcode and its corresponding persistence image with a resolution of $2500$ and a bandwidth of $0.05$, and apply the $k$-means clustering method with $k$ set to $4$ or $5$.

\begin{figure}[hbt!]
    \centering
    \includegraphics[width=1 \linewidth]{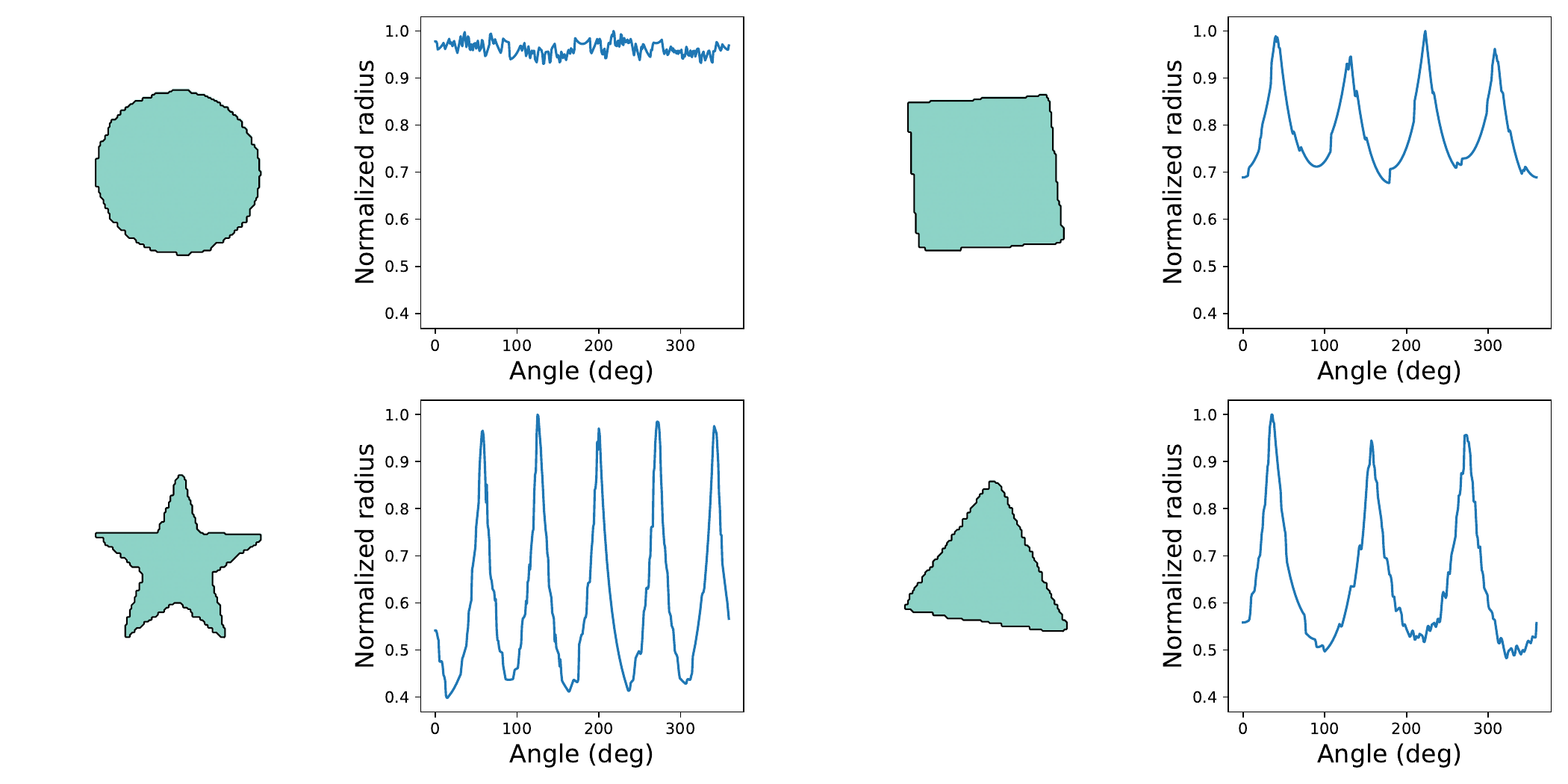}
    \caption{Images of four different shapes, i.e.,  circle, square, star,  and triangle and the corresponding time-series data with the length on the $y$-axis versus the rotation angle on the  $x$-axis. \textcolor{black}{Notice that the scales of each time-series data are different.}}
    \label{Four shape}
\end{figure}

In the case of $k=5$, we observe the emergence of a subcluster within shapes of the same type. To investigate the validity of this phenomenon, we quantitatively compared the subclusters both at the image level and at the time-series level. At the image level, we measured differences in Circularity and Eccentricity, and statistical analysis confirmed that the differences between subclusters were significant with respect to these descriptors (as indicated by $p$-values). At the time-series level, we examined the frequency domain representation, which further revealed distinctive patterns between the subclusters.

In order to quantify the morphological differences observed between the subclusters, 
we employed the \texttt{scikit\allowbreak-image} library, which provides a set of classical shape 
descriptors commonly used in image analysis~\cite{gonzalez2018image}. Among these, we focus on two descriptors that are particularly effective in capturing 
the differences between clusters: Circularity and Eccentricity.
 
\begin{enumerate}
    \item \textbf{Circularity.}  
    Circularity measures how close a shape is to a circle and is defined as  
    \[
        \text{Circularity} = \frac{4\pi A}{P^2},
    \]
    where $A$ is the area and $P$ is the perimeter. For a circle, this value is theoretically $1$, 
    and it decreases as the shape deviates from circularity.  
    Circularity is sensitive to boundary smoothness because, for the same area, jagged boundaries increase the perimeter $P$.
    Therefore, shapes with complex boundaries, such as stars, tend to have low circularity values.

    \item \textbf{Eccentricity.}  
    Eccentricity is defined as the eccentricity of the ellipse having the same second central 
    moments as the region. Using the eigenvalues $\lambda_{\max}$ (major axis) and 
    $\lambda_{\min}$ (minor axis) of the covariance matrix, it is given by  
    \[
        \text{Eccentricity} = \sqrt{1 - \frac{\lambda_{\min}}{\lambda_{\max}}}.
    \]
    This value is $0$ for a circle and increases as the shape becomes more elongated.
\end{enumerate}

\textbf{Subcluster analysis under the standard ray:}
Figure \ref{Four shape2} shows the clustering results with the standard ray. That is, with $\mathbf{a} = (1, 1, 1, 1, 1)$ and $\mathbf{b} = (0, 0, 0, 0, 0)$, the left figure presents the case of $k = 4$, while the right figure presents the case of $k = 5$ using $k$-means clustering. When $k = 4$, four clusters are formed as one would naturally expect. However, when $k = 5$, an unexpected subcluster appears within the square group. Figure \ref{fig:square cluster} illustrates the squares clustered into cluster 1 and cluster 2.

Table \ref{tab:square} compares the two square subclusters obtained under the standard ray. 
While the visual differences between the two subclusters are hardly noticeable, the quantitative analysis reveals clear distinctions. 
Cluster~1 exhibits significantly higher circularity ($p$-value: $0.000$) and eccentricity ($p$-value: $0.000$) than Cluster~2. 
In other words, the squares in Cluster~1 tend to have smoother boundaries and are more elongated, whereas those in Cluster~2 are less circular and less elongated.

Figure~\ref{fig:time_series_level_square} analyzes the differences between the two subclusters at the time-series level. 
While no clear distinction was observed in the time domain, replacing each data point 
$f_i$ with $\left(\vert \widehat{f_i}(1) \vert, \cdots, \vert \widehat{f_i}(5) \vert \right)$ and applying two-dimensional Principal Component Analysis (PCA) revealed a clear separation between the subclusters.

\textbf{Subcluster analysis under a non-standard ray:} Figure \ref{Four shape3} shows the clustering results with a non-standard ray, where we set $\mathbf{a} = (1, 1, 1, 1, 0.2)$ and $\mathbf{b} = (0, 0, 0, 0, 0)$. The left figure presents the case of $k = 4$, while the right figure presents the case of $k = 5$ using $k$-means clustering. When $k = 4$, four clusters are formed as one would naturally expect. However, when $k = 5$, an unexpected subcluster appears within the star group. Figure \ref{fig:star cluster} illustrates the stars clustered into cluster 2 and cluster 3.

Table \ref{tab:star} compares the two star subclusters obtained under the non-standard ray.
Unlike the star subclusters in Figure~\ref{fig:star cluster}, the star subclusters show visible differences in their average boundary shapes.
The average boundary of Cluster~2 appears more irregular and less smooth compared to that of Cluster~3, and it is also more elongated.
The quantitative analysis further confirms this observation: Cluster~2 exhibits significantly lower circularity ($p$-value: $0.000$) and significantly higher eccentricity  ($p$-value: $0.005$) than Cluster~3.

Figure~\ref{fig:time_series_level_star} analyzes the differences between the two star 
subclusters at the time-series level. While no clear distinction was observed in the 
time domain, examining the fifth Fourier coefficient $\lvert \widehat{f_i}(5) \rvert$ of each 
data point $f_i$ revealed a clear separation between the subclusters. This separation 
of subclusters can be interpreted as arising from the choice of the direction vector 
$\mathbf{a} = (1,1,1,1,0.2)$, which, according to Theorem~\ref{direction barcode}, greater weight is assigned to the length of the bar corresponding to the fifth-frequency component.

\begin{figure}[H]
    \centering
    \includegraphics[width=1.0\linewidth]{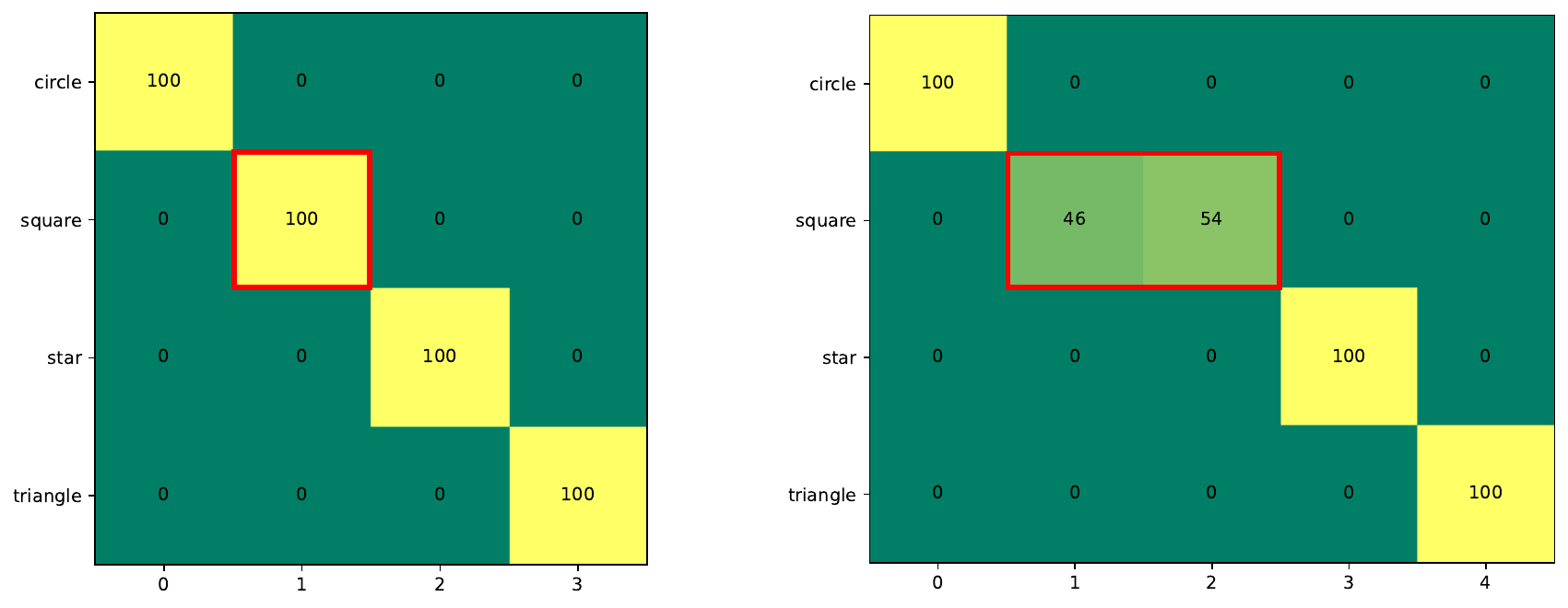}
\caption{{\bf The $k$-means clustering of four shape image groups with the standard ray.} 
Left: results with the standard ray $\mathbf{a} = (1,1,1,1,1)$ and $\mathbf{b} = (0,0,0,0,0)$ 
for $k=4$, where the four shape classes are cleanly separated as expected. 
Right: results for $k=5$, where the square group is further subdivided into two subclusters.}
    \label{Four shape2}
\end{figure}

\begin{figure}[H]
    \centering
    \includegraphics[width=1.0\linewidth]{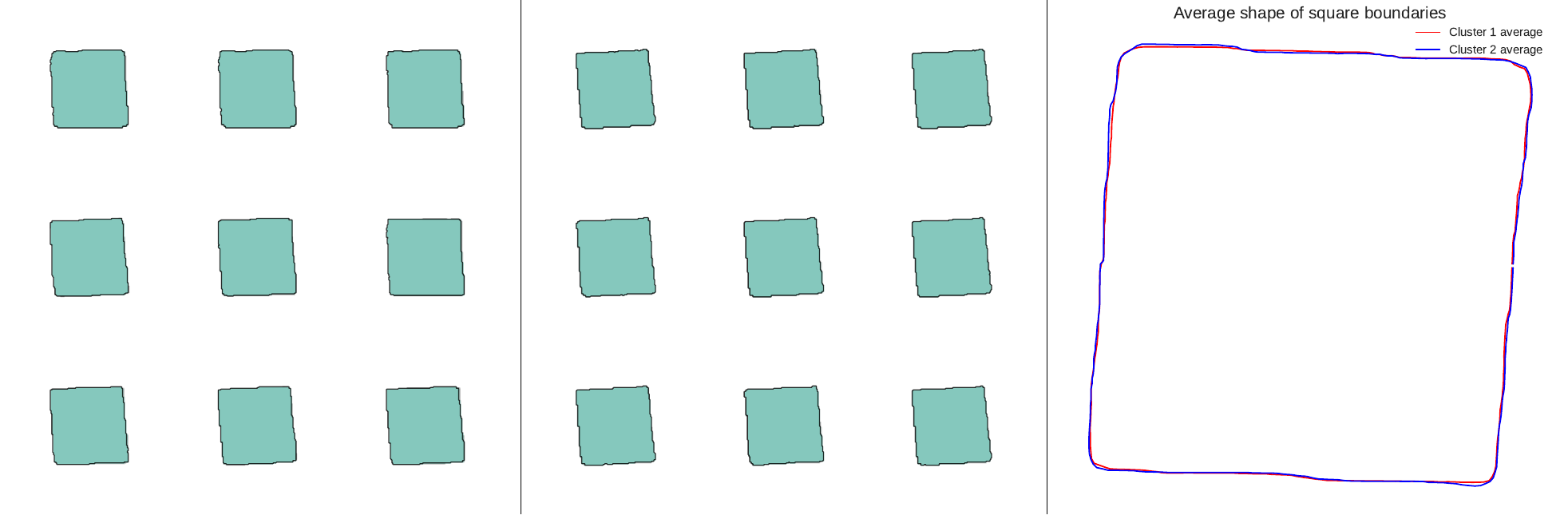}
\caption{{\bf Subclusters of the square shape.} 
These subclusters are obtained by further dividing the original square cluster shown in Figure~\ref{Four shape2}. Left: subcluster corresponding to Cluster 1, Center: subcluster corresponding to Cluster 2, Right: aligned average boundary plots.}

    \label{fig:square cluster}
\end{figure}

\begin{table}[H]
    \centering
    \begin{tabular}{lccc}
        \toprule
        Square & Cluster 1 & Cluster 2 & $p$-value \\
        \midrule
        Circularity   & $0.778 \pm 0.007$ & $0.772 \pm 0.011$ & 0.000 \\
        Eccentricity  & $0.325 \pm 0.017$ & $0.304 \pm 0.016$ & 0.000 \\
        \bottomrule
    \end{tabular}
    \caption{{\bf Comparison of square subclusters (standard ray).} 
Values are given as mean $\pm$ standard deviation. 
Cluster~1 exhibits significantly higher circularity ($p = 0.000$) and eccentricity ($p = 0.000$) compared to Cluster~2.}

    \label{tab:square}
\end{table}

\begin{figure}[H]
    \centering
    \includegraphics[width=0.85\linewidth]{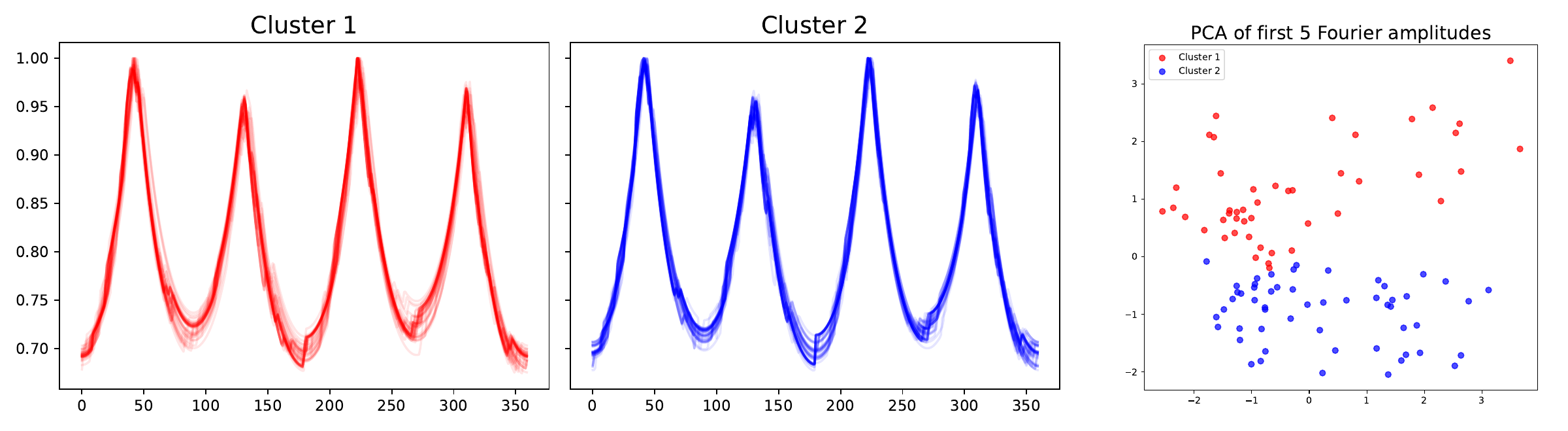}
    \caption{{\bf Subclusters of star shape at the time-series level.} 
Left: time-series of Cluster~1, Middle: time-series of Cluster~2, 
Right: two-dimensional PCA of the first five Fourier amplitudes 
(red: Cluster~1, blue: Cluster~2). 
While no clear distinction was observed in the time domain, 
the PCA representation reveals a clear separation between the two subclusters.}

    \label{fig:time_series_level_square}
\end{figure}


\begin{figure}[H]
    \centering
    \includegraphics[width=1.0\linewidth]{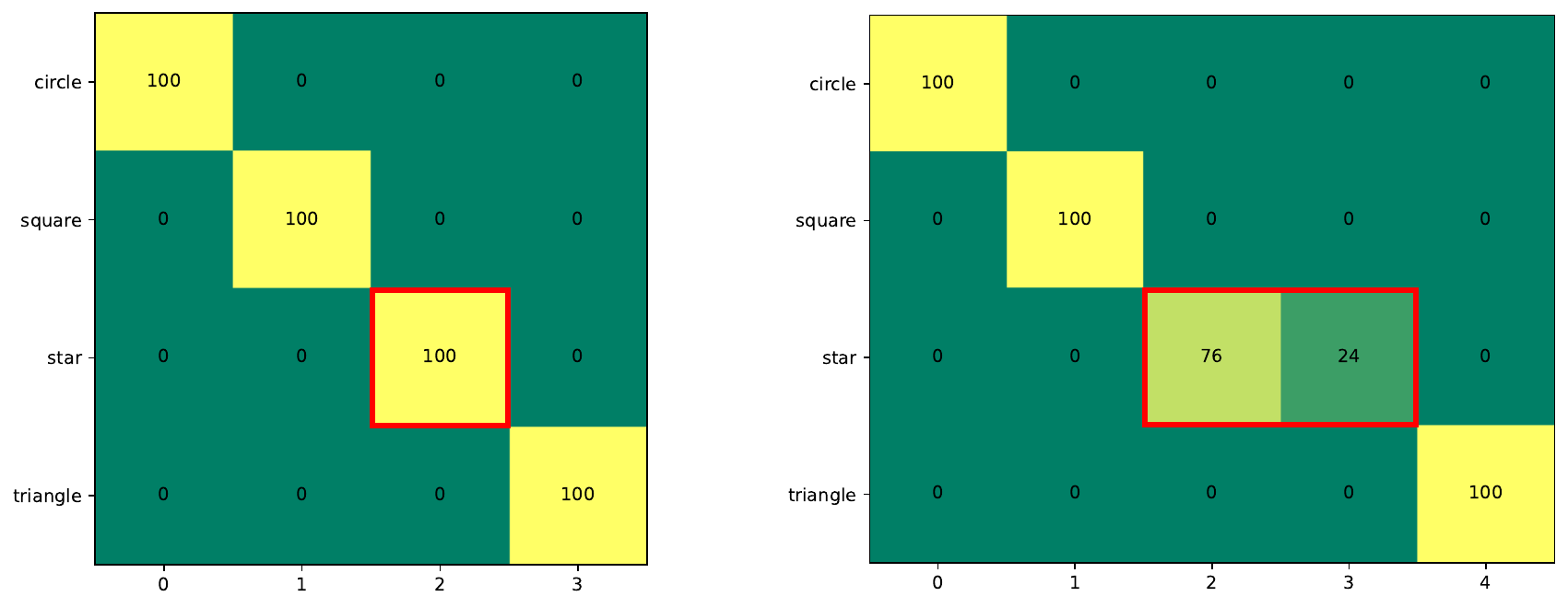}
\caption{{\bf The $k$-means clustering of four shape image groups with a non-standard ray.} 
Here the direction vector is set to $\mathbf{a} = (1,1,1,1,0.2)$ and $\mathbf{b} = (0,0,0,0,0)$. 
Left: results for $k=4$, where four shape classes are obtained as expected. 
Right: results for $k=5$, where the star group is further subdivided into two subclusters.}

    \label{Four shape3}
\end{figure}

\begin{figure}[H]
    \centering
    \includegraphics[width=1.0\linewidth]{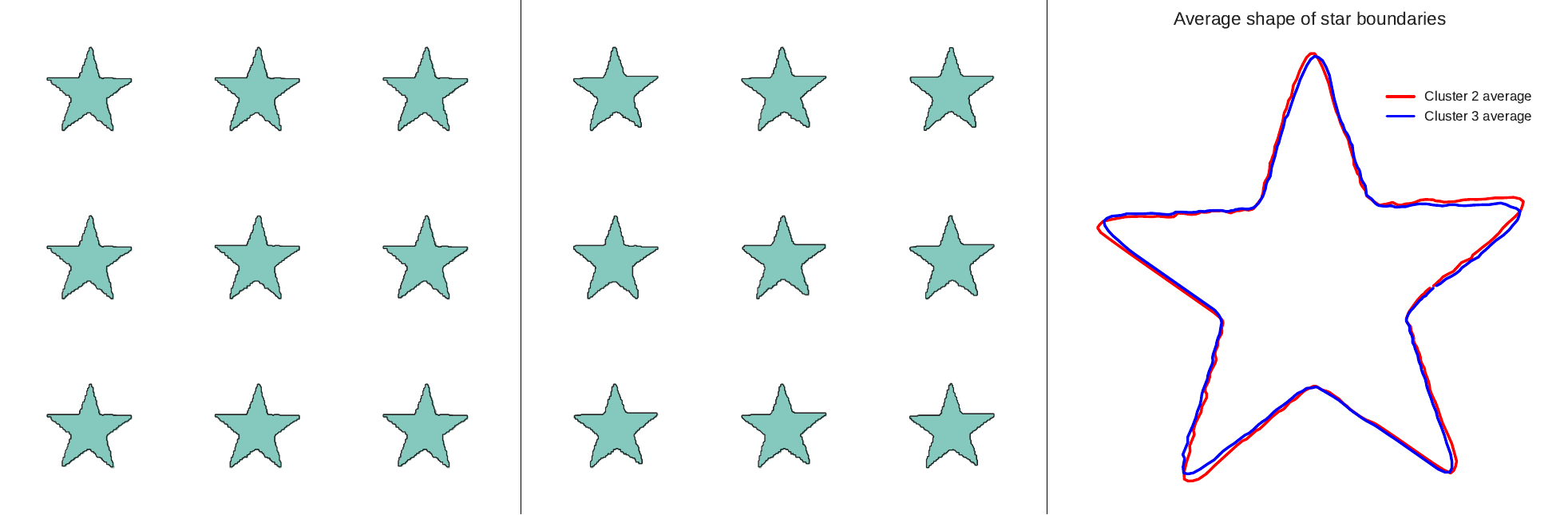}
    \caption{{\bf Subclusters of star shape.} 
These subclusters are obtained by further dividing the original star cluster shown in Figure~\ref{Four shape3}. Left: subcluster corresponding to Cluster 2, Center: subcluster corresponding to Cluster 3, Right: aligned average boundary plots.}
    \label{fig:star cluster}
\end{figure}

\begin{table}[H]
    \centering
    \begin{tabular}{lccc}
        \toprule
        Star & Cluster 2 & Cluster 3 & $p$-value \\
        \midrule
        Circularity   & $0.270 \pm 0.004$ & $0.275 \pm 0.004$ & 0.000 \\
        Eccentricity  & $0.306 \pm 0.021$ & $0.288 \pm 0.027$ & 0.005 \\
        \bottomrule
    \end{tabular}
    \caption{{\bf Comparison of star subclusters (non-standard ray).} 
Values are reported as mean $\pm$ standard deviation.
Cluster~2 exhibits significantly lower circularity ($p = 0.000$) and significantly higher eccentricity ($p = 0.005$) compared to Cluster~3.} 
    \label{tab:star}
\end{table}

\begin{figure}[H]
    \centering
    \includegraphics[width=0.9\linewidth]{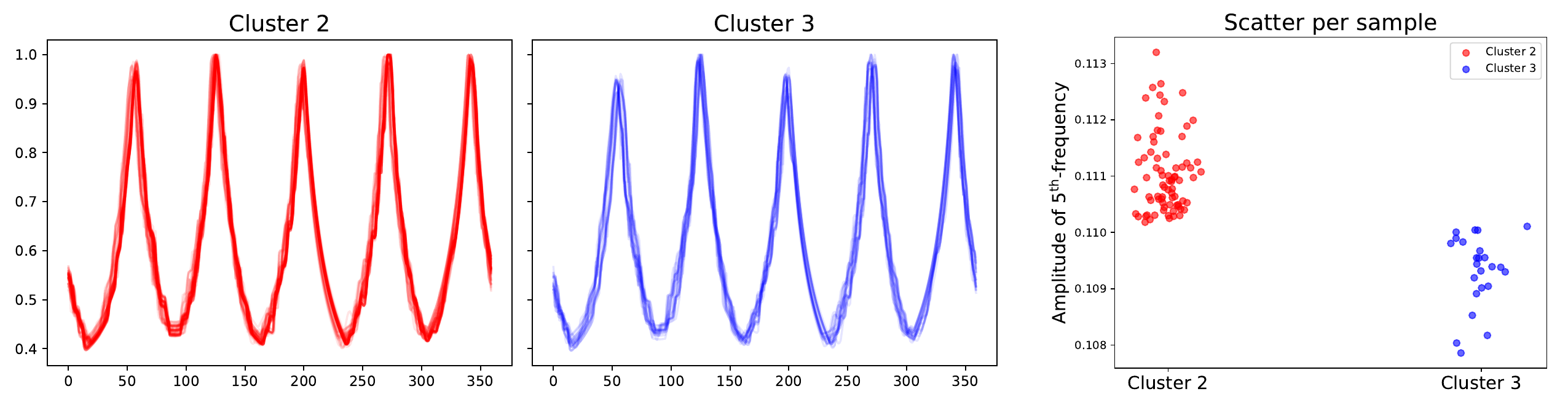}
    \caption{{\bf Subclusters of star shape at the time-series level.} 
Left: time-series of Cluster~2, Middle: time-series of Cluster~3, 
Right: scatter plot of the fifth Fourier coefficient amplitude (red: Cluster~2, blue: Cluster~3). 
While no clear distinction was observed in the time domain, 
the separation in the fifth-frequency component reveals a clear distinction between the two subclusters.}

    \label{fig:time_series_level_star}
\end{figure}

\end{example}

\subsubsection*{\textbf{Methods used in Examples~\ref{example 3} and \ref{example 4}:}}
Example~\ref{example 2} demonstrated the ability of EMPH to reveal hidden subcluster structures through clustering experiments. 
In Examples~\ref{example 3} and \ref{example 4}, we instead evaluate its classification performance by comparing EMPH with several existing approaches. 
The methods used in our experiments are summarized below (*SWE: Sliding Window Embedding with Vietoris–Rips complex using the specified norm).

\begin{itemize}
    \item \textbf{SWE-based methods:}

In sliding window embeddings, both the embedding dimension $M+1$ and the delay parameter $\tau$ are crucial, and several strategies have been proposed for their optimal selection. For example, \cite{perea2015sliding} and its application \cite{perea2015sw1pers} recommend setting the embedding dimension $M$ to be at least twice the number of Fourier modes. They also suggest choosing the delay parameter as $\tau = \tfrac{2\pi}{(M+1)Z}$, where $Z$ denotes the expected number of periods of the time-series data. Other approaches include selecting $M$ by minimizing the number of false nearest neighbors \cite{kennel1992determining}, or determining $\tau$ as the first local minimum of the mutual information between a time-series and its $\tau$-delayed time-series \cite{fraser1986independent}.

In this section, we determine $M$ and $\tau$ by performing a grid search over multiple combinations, since the optimal values of these parameters may vary across data points. Moreover, in the method of \cite{perea2015sliding}, the optimal value of $Z$ is not known a priori, which further motivates this approach. Note that all experiments here use discrete time-series with time steps $t \in \{0,\cdots,T-1\}$, and $\tau$ is treated as an integer measured in samples.

For convenience, suppose a time-series is given as a function $f:\{0,\cdots,T-1\}\to\mathbb{R}$, where $f(t)=x_t$. We identify this time-series with $\mathbf{x} = (x_0,\cdots,x_{T-1})$.

    \begin{itemize}
    \item Method A: Given a time-series $\mathbf{x}=\left(x_0,\cdots,x_{T-1}\right)$, construct the sliding window embedding
\begin{equation}
    X = \bigl\{ (x_t, x_{t+\tau}, \cdots, x_{t+M\tau}) \in \mathbb{R}^{M+1} \;\big|\; 0 \le t \le T - M\tau - 1 \bigr\}.
    \label{eq:sliding}
\end{equation}
Build the Vietoris-Rips complex $\Rip_{\epsilon}(X)$ under the maximum metric, compute its barcode, transform it into a Betti sequence, and use this sequence as input to a Support Vector Machine (SVM) classifier.

  \item Method B: Same as Method A, except that the barcode is transformed into a persistence image.

    \item Method C: Same as Method A, but using the Euclidean metric instead of the maximum metric.

    \item Method D: Same as Method B, but using the Euclidean metric instead of the maximum metric.
\end{itemize}

    \item \textbf{Non-TDA methods:}
    \begin{itemize}
        \item DTW: Dynamic Time Warping (DTW) is a classical technique for measuring similarity 
between two time-series that may be misaligned in phase or exhibit local variations in speed. 
Given time-series data $\mathbf{x}=\left(x_0,\cdots,x_{T-1}\right), \mathbf{y}=\left(y_0,\cdots,y_{T-1}\right)$, the DTW distance is defined as
        \[
            \mathrm{DTW}(\mathbf{x},\mathbf{y}) = \min_{\pi} \sum_{(i,j)\in\pi} |x_i-y_j|,
        \]
        where $\pi$ ranges over alignment paths from $(0,0)$ to $(T-1,T-1)$ that move 
only right, up, or diagonally and never decrease in either index. Time-series classification is then performed by the $1$-nearest neighbor method with respect to this metric.

        \item Fourier: Compute the discrete Fourier transform $\hat{\mathbf{x}}(k) = \frac{1}{T} \sum\limits_{t=0}^{T-1} x_t e^{-2\pi i kt/T}$ and use the feature vector
        \[
            (|\hat{\mathbf{x}}(1)|,|\hat{\mathbf{x}}(2)|,\cdots,|\hat{\mathbf{x}}(N)|)
        \]
        as input to an SVM classifier. Here, $N$ is the preset maximum frequency and we also vary $N$ in the experiments.
    \end{itemize}
    
    \item \textbf{Proposed methods:}
    \begin{itemize}

\item EMPH-Betti (SR): In Algorithm~\ref{alg:main}, use 
$\mathsf{bcd}_{n}^{\Rip,\ell}(\Psi_{f_k})$ (single ray), vectorize by Betti sequences, and classify with an SVM.

\item EMPH-PI (SR): In Algorithm~\ref{alg:main}, use 
$\mathsf{bcd}_{n}^{\Rip,\ell}(\Psi_{f_k})$ (single ray), 
vectorize by persistence images, and classify with an SVM.

\item EMPH-Betti (CR): In Algorithm~\ref{alg:main}, use 
$\mathsf{bcd}_{n}^{\Rip,\mathcal{L}_{\ell}}(\Psi_{f_k})$ (collection of rays associated with $\ell$), 
vectorize by Betti sequences, and classify with an SVM.

\item EMPH-PI (CR): In Algorithm~\ref{alg:main}, use 
$\mathsf{bcd}_{n}^{\Rip,\mathcal{L}_{\ell}}(\Psi_{f_k})$ (collection of rays associated with $\ell$), 
vectorize by persistence images, and classify with an SVM.
    
\end{itemize}

\item \textbf{Hyperparameter tuning}

\begin{enumerate}
    \item (Methods A–D): For the sliding window methods (Methods A–D), we conducted a grid search. In Example~\ref{example 3} and Example~\ref{example 4}, the search space comprised embedding dimensions $(M+1)$ of 2, 3, 4, 5, 10, and 20; delay parameters $\tau$ of 1, 2, 3, 4, 5, 10, and 20 (subject to $T - M\tau - 1 > 0$, see Eq.~\eqref{eq:sliding}); resolutions (vectorized dimension) of 100 and 2500; and persistence image bandwidths of 0.1 and 1.

    \item (Fourier): 
For the Fourier method, we varied the maximum frequency $N$.
In Example~\ref{example 3}, $N$ was chosen from 10, 20, and 40. In Example~\ref{example 4}, $N$ was chosen from 20, 40, 60, 80, and 100.

\item (EMPH):
In Algorithm~\ref{alg:main}, the ray $\ell$ appears as a variable and therefore must be specified in advance.  
We treated the choice of rays as part of the overall hyperparameter search.  
In Example~\ref{example 3}, we randomly selected 100 rays, while in Example~\ref{example 4} we selected 500 rays.  
Each ray $\ell=(\mathbf{a},\mathbf{b})$ with $\mathbf{a}=(a_1,\cdots,a_N)$ and $\mathbf{b}=(b_1,\cdots,b_N)$ was generated componentwise as  
$a_L = 1 + 0.1 \cdot \epsilon_L$ with $\epsilon_L \in \{0,\cdots,10\}$ and  
$b_L = 0.1 \cdot \eta_L$ with $\eta_L \in \{0,\cdots,50\}$, chosen at random for each $L=1,\cdots,N$.  
In addition, in Example~\ref{example 3} we searched over resolutions of 100 and 2500,  
maximum Fourier frequency $N$ set to 10, 20, or 40, and persistence image bandwidths of 0.1 and 1.  
In Example~\ref{example 4}, the search space included resolutions of 100 and 2500,  
maximum Fourier frequency $N$ set to 20, 40, 60, 80, or 100, and persistence image bandwidths of 0.1 and 1.  
The reported results correspond to the best performance obtained within these search spaces.

\end{enumerate}

\end{itemize}

All experiments were conducted on a server equipped with an Intel Core i9-10850K CPU (10 cores, 20 threads, 3.60 GHz) and 64 GB of RAM.

\begin{example} 
\label{example 3}
In this example, we apply the proposed EMPH method to the Proximal Phalanx Outline Age Group Dataset from the UCR archive \cite{dau2019ucr}. The original dataset consists of image data from different age groups. The Proximal Phalanx consists of two components, one indicated in blue and the other in red. For each component, we consider its own center, and the Euclidean distance from the center to the outline is measured as the axis rotates counterclockwise from $0$ to $2\pi$. The resulting measurements are represented as a time-series. The right figure of Figure~\ref{fig:Proximal} shows the time-series obtained by concatenating the two componentwise time-series and then standardizing the result. In the experiment, three age groups were used: 0–6, 7–12, and 13–19 years old. The training and test sets consist of 400 and 205 samples, respectively. The task is to classify the given data into the three categories corresponding to the age groups.

\begin{figure}[hbt!]
    \centering
    \includegraphics[width=1 \linewidth]{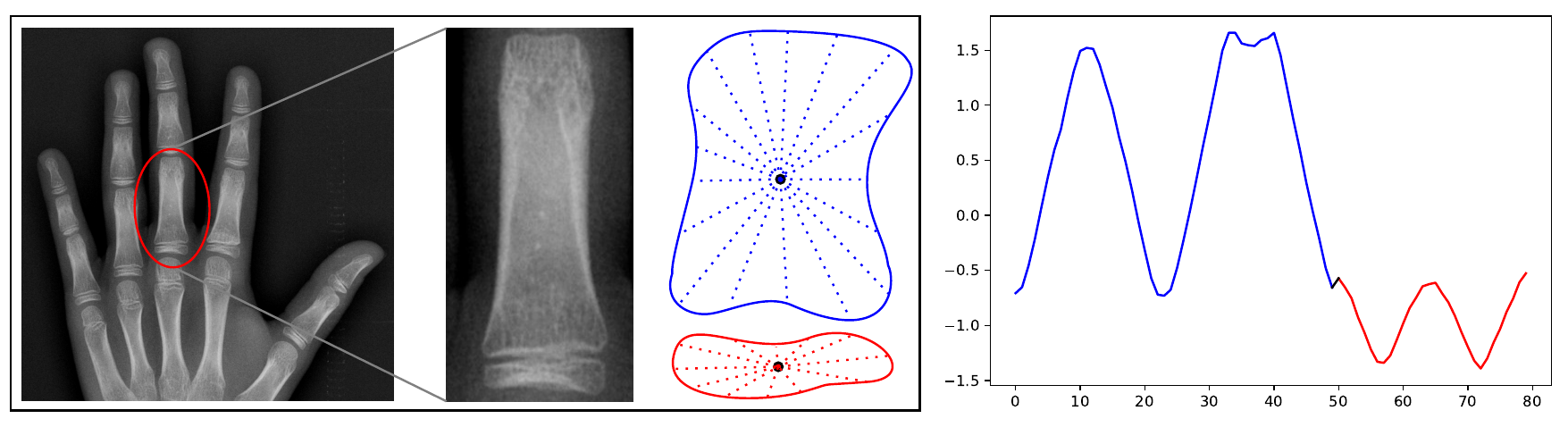}
    \caption[]{Left: Proximal phalanx outline image\footnotemark. Right: 
    Standardized time-series obtained by measuring the Euclidean distance between the center and the outline at each angle. \cite{bagnall2014predictive}.}
    \label{fig:Proximal}
\end{figure} 

\footnotetext{We used a hand X-ray image from the bone-age dataset created by Universidade Presbiteriana Mackenzie via Roboflow Universe (CC BY 4.0), \url{https://share.google/images/eVDXYlwyG4HgvjJ1q}.}

Tables~\ref{fig:Proximal_table12} and \ref{fig:Proximal_table12-3} summarize the classification results on the Proximal Phalanx Outline Age Group dataset using the best hyperparameters found by grid search. Table~\ref{fig:Proximal_table12} presents the performance of the SWE-based and DTW methods, while Table~\ref{fig:Proximal_table12-3} shows the results of the Fourier baseline and the EMPH methods. The direction vectors and endpoints given below Table~\ref{fig:Proximal_table12-3} correspond to the choices that yielded the highest accuracy. As the results show, the EMPH methods achieve the best overall accuracy.

\begin{table}[hbt!]
\centering
\begin{tabular}{|c||c|c|c|c|c|}
\hline
Method & Accuracy (\%) & Embedding dimension & Delay $\tau$ & Resolution & Bandwidth \\ \hline
Method A & 84.88 & 5 & 2 & 100 & -- \\ \hline
Method B & 86.34 & 5 & 2 & 2500 & 1.0 \\ \hline
Method C & 85.85 & 5 & 5 & 100 & -- \\ \hline
Method D & 85.85 & 2 & 5 & 100 & 1.0 \\ \hline
DTW      & 80.49 & -- & -- & -- & -- \\ \hline
\end{tabular}
\caption{{\bf SWE-based and DTW methods:} 
Classification accuracy on the Proximal Phalanx Outline Age Group dataset using the best hyperparameters obtained through grid search.}

\label{fig:Proximal_table12}
\end{table}

\begin{table}[hbt!]
\centering
\begin{tabular}{|c||c|c|c|c|c|}
\hline
Method & Accuracy (\%) & $N$ & Ray & Resolution & Bandwidth  \\ \hline
Fourier     & 84.39    & 20  & --  & -- & --\\ \hline
EMPH-Betti (SR)  & 86.34    & 10 & ($\mathbf{a}_1, \mathbf{b}_1$) & 100  & -- \\ \hline
EMPH-PI (SR)  & \textbf{88.78}    & 10 & ($\mathbf{a}_2, \mathbf{b}_2$) & 2500  & 0.1 \\ \hline
EMPH-Betti (CR)  & 86.34    & 10 & ($\mathbf{a}_3, \mathbf{b}_3$) & 100  & -- \\ \hline
EMPH-PI (CR)     & 87.80    & 10 & ($\mathbf{a}_4, \mathbf{b}_4$) & 100  & 1.0 \\ \hline
\end{tabular}

\caption{{\bf Fourier and EMPH methods:} 
Classification accuracy on the Proximal Phalanx Outline Age Group dataset using the best hyperparameters obtained through grid search. Here, $N$ is the preset maximum frequency, and Ray indicates the choice of direction vector and endpoint.}

\begin{eqnarray}
{\mathbf{a}_1}, {\mathbf{b}_1} &=& (1.4, 1.6, 1.4, 1.6, 1.7, 1.0, 1.8, 2.0, 1.5, 1.1), \quad (4.5, 3.1, 1.9, 3.6, 4.7, 2.8, 0.1, 0.7, 3.9, 0.2) \nonumber \\
{\mathbf{a}_2}, {\mathbf{b}_2} &=& (1.9, 1.3, 1.3, 1.2, 1.8, 1.2, 1.8, 1.7, 1.4, 1.5), \quad (4.5, 4.0, 0.8, 3.7, 2.0, 5.0, 1.1, 2.1, 4.7, 2.9) \nonumber \\
{\mathbf{a}_3}, {\mathbf{b}_3} &=& (1.2, 1.5, 1.1, 1.6, 1.1, 1.5, 1.7, 1.1, 1.9, 1.9), \quad (2.4, 3.6, 0.7, 0.9, 3.8, 1.9, 4.1, 4.7, 1.1, 3.3) \nonumber \\
{\mathbf{a}_4}, {\mathbf{b}_4} &=& (1.2, 1.7, 1.7, 1.7, 1.7, 1.1, 1.2, 1.3, 1.2, 1.1), \quad (1.8, 3.4, 3.1, 1.1, 4.3, 1.4, 4.6, 2.8, 4.8, 1.6) \nonumber
\end{eqnarray}        

\label{fig:Proximal_table12-3}
\end{table}


Figure~\ref{fig:Proximal PI bar} displays the CPU time (in milliseconds) for each method under the best hyperparameter settings. For each method, the reported value is the average CPU time over 10 runs required to compute up to the barcode. As shown in the figure, our proposed methods achieve significantly lower CPU times than the other methods. Note that the EMPH methods are evaluated under fixed rays. Consequently, as in our experiment, even when randomly exploring 100 rays $\ell$, a broad ray search is feasible without increasing the time burden relative to existing methods.
\begin{figure}[hbt!]
   \centering
  
        \includegraphics[width=0.8 \linewidth]{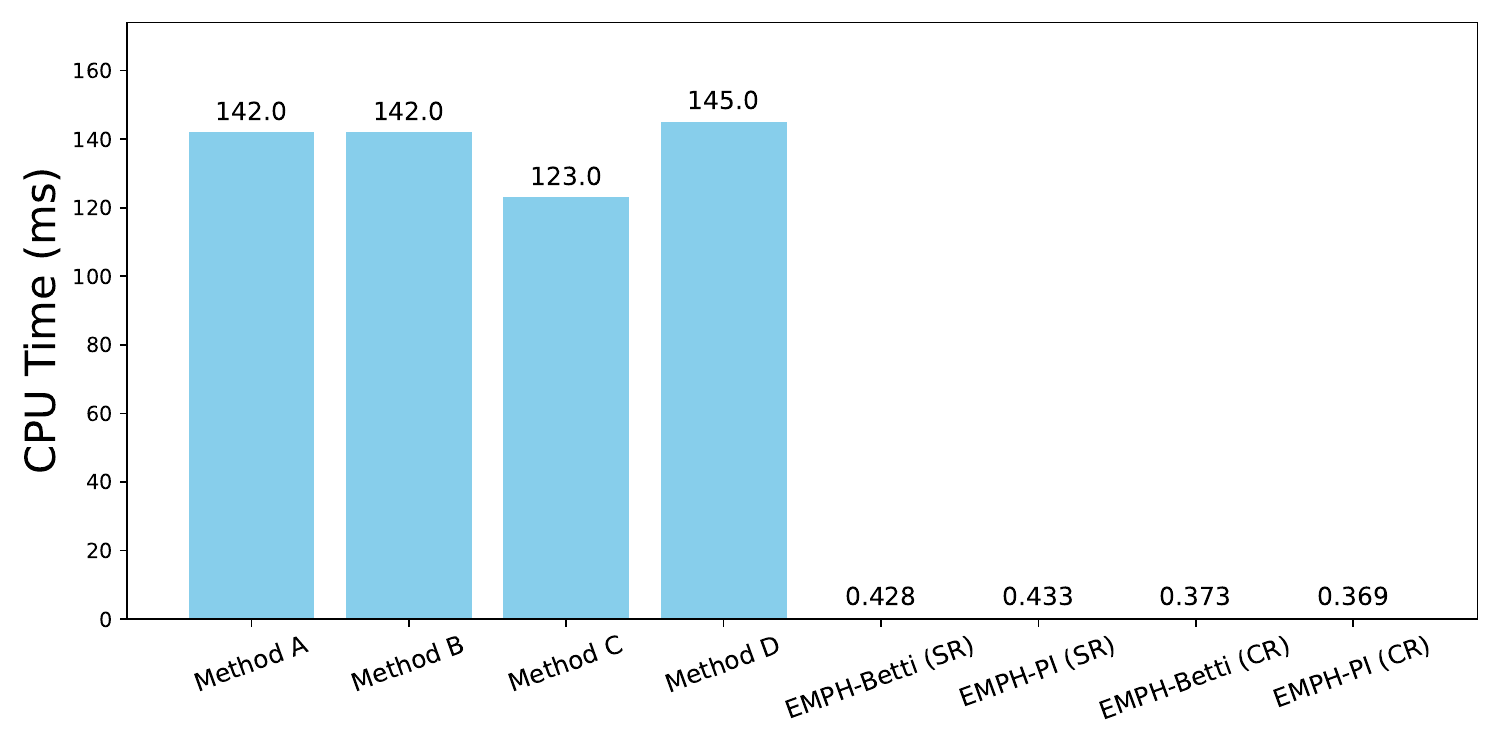}
    \caption{{\bf Average CPU time (ms)}: The average CPU time per computation for Methods A–D and the EMPH methods under the best hyperparameter settings. The average values were obtained by repeating the computation up to the barcode 10 times.}
    \label{fig:Proximal PI bar}
\end{figure}

\end{example}


\begin{example}{rs-fMRI data}
\label{example 4}
In this example, we use a resting state functional Magnetic Resonance Imaging (fMRI) dataset available in \url{https://github.com/laplcebeltrami/rsfMRI}, \cite{huang2020statistical}\footnote{We acknowledge that the fMRI dataset used in this research was provided by Dr. Moo K. Chung at the University of Wisconsin-Madison. We thank Dr. Chung for his valuable comments and suggestions on our proposed method.}. The dataset consists of time-series data collected from $100$ subjects across $6$ brain regions: the left and right orbital parts of the inferior frontal gyrus, the left and right hippocampus, and the left and right middle occipital regions. Each time-series has a length of $1200$, so the dataset has the format $(1200, 6, 100)$. Because of this length, computing persistent homology via the Vietoris–Rips complex with sliding window embedding becomes computationally very expensive. With a total of $600$ time-series, the computational cost is prohibitively high and cannot be handled within a reasonable time frame.

The purpose of this example is not to determine whether the six brain regions truly form distinct groups, but rather to provide a controlled experimental setting for comparing the performance of the two approaches. Since the fMRI dataset does not provide a predefined train/test split, we randomly partitioned the data into 80\% training and 20\% testing sets, and repeated the experiment five times with shuffled splits. Assuming that the fMRI data can be divided into six groups, Table~\ref{chung_original} presents the performance of the SWE-based and DTW methods. Table~\ref{chung_emph} shows the results of the Fourier baseline and the EMPH methods. The direction vectors and endpoints provided in Appendix \ref{appendix D} correspond to the choices that yielded the highest accuracy. As the results show, Method~D and EMPH-PI~(CR) achieved the best overall accuracy. Figures \ref{fig:Chung_SWE} and \ref{fig:Chung_EMPH} illustrate the persistence images obtained using the best parameters with Method D and EMPH-PI (CR), respectively. In each figure, the top row shows sample persistence images from six different brain regions of a single subject, and the bottom row presents the averaged persistence images computed over 100 subjects for each corresponding region.

\begin{table}[hbt!]
\centering
\begin{tabular}{|c||c|c|c|c|c|}
\hline
Method & Accuracy (\%) & Embedding dimension & Delay $\tau$ & Resolution & Bandwidth \\ \hline
Method A & 29.50 & 10 & 4 & 2500 & -- \\ \hline
Method B & 31.67 & 5 & 10 & 2500 & 0.1 \\ \hline
Method C & 30.67 & 20 & 2 & 2500 & -- \\ \hline
Method D & \textbf{32.50} & 10 & 3 & 2500 & 1.0 \\ \hline
DTW      & 10.67 & -- & -- & -- & -- \\ \hline
\end{tabular}
        \label{table-6}
            \caption{{\bf SWE-based and DTW methods:} 
Classification accuracy on the fMRI dataset using the best hyperparameters obtained through grid search (average over 5 experiments with randomly shuffled train/test splits).}

            \label{chung_original}
        \end{table}

         \begin{table}[hbt!]
\centering
\begin{tabular}{|c||c|c|c|c|c|}
\hline
Method & Accuracy (\%) & $N$ & Ray & Resolution & Bandwidth  \\ \hline
Fourier     & 23.83    & 60  & --  & -- & --\\ \hline
EMPH-Betti (SR)  & 29.00    & 100 & ($\mathbf{a}_1, \mathbf{b}_1$) & 2500  & -- \\ \hline
EMPH-PI (SR)  & 28.50    & 100 & ($\mathbf{a}_2, \mathbf{b}_2$) & 2500  & 0.1 \\ \hline
EMPH-Betti (CR)  & 29.50    & 100 & ($\mathbf{a}_3, \mathbf{b}_3$) & 100  & --  \\ \hline
EMPH-PI (CR)     & \textbf{32.50}   & 100  & ($\mathbf{a}_4, \mathbf{b}_4$) &  2500  & 1.0  \\ \hline
\end{tabular}
        \label{table-7}
            \caption{{\bf Fourier and EMPH methods:} 
Classification accuracy on the fMRI dataset using the best hyperparameters obtained through grid search (average over 5 experiments with randomly shuffled train/test splits).  Here, $N$ is the preset maximum frequency, and Ray indicates the choice of direction vector and endpoint.}

            \label{chung_emph}
        \end{table}

\begin{figure}[hbt!]
    \centering
    \includegraphics[width=0.81\linewidth]{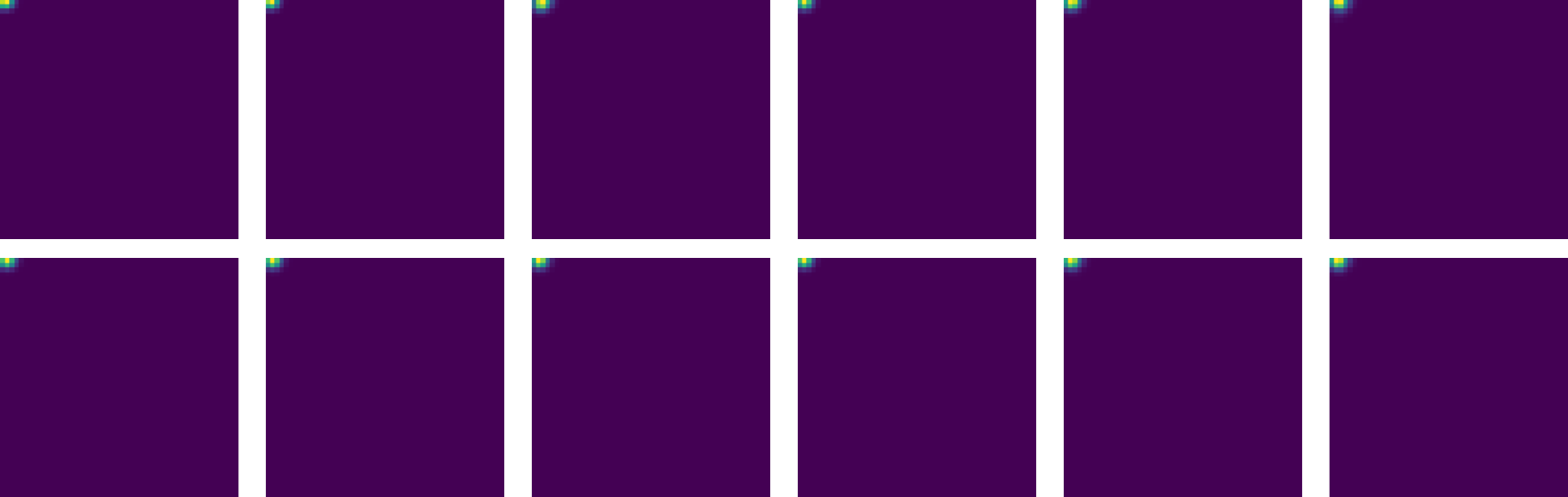}
    \caption{\textbf{Persistence images by Method D with best parameters:} The top row shows sample persistence images from six different brain regions of a single subject, and the bottom row displays the averaged persistence images computed over 100 subjects for each corresponding region.}
    \label{fig:Chung_SWE}
\end{figure}

\begin{figure}[hbt!]
    \centering
    \includegraphics[width=0.81\linewidth]{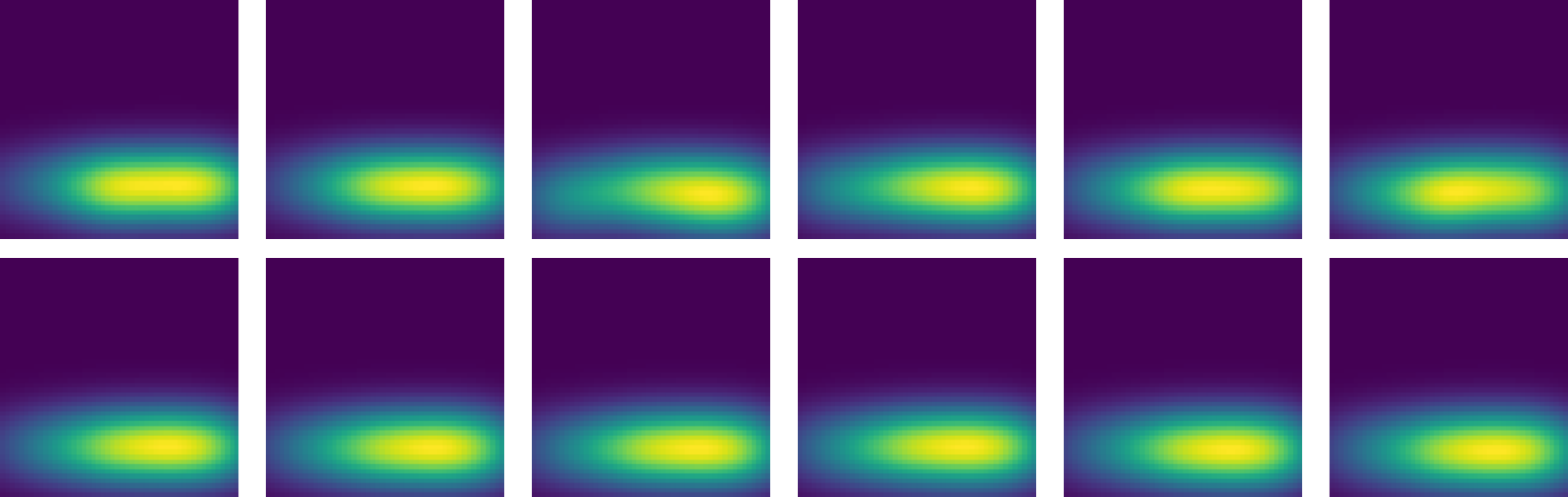}
    \caption{\textbf{Persistence images by EMPH-PI (CR) with best parameters:} The top row shows sample persistence images from six different brain regions of a single subject, and the bottom row displays the averaged persistence images computed over 100 subjects for each corresponding region.}
    \label{fig:Chung_EMPH}
\end{figure}

Figure~\ref{fig:Chung_PI_bar} displays the CPU time (in seconds) for each method under the best hyperparameter settings. For each method, the reported value is the average CPU time over 10 runs required to compute up to the barcode. As shown in the figure, our proposed methods achieve significantly lower CPU times than the other methods. Note that the EMPH methods are evaluated under fixed rays. Consequently, as in our experiment, even when randomly exploring 500 rays $\ell$, a broad ray search is feasible without increasing the time burden relative to existing methods. In particular, compared with Figure~\ref{fig:Proximal PI bar}, the gap in computational time between EMPH and the SWE-based methods becomes even larger, suggesting that the EMPH approach scales more efficiently as the length of the time-series increases.
 
\begin{figure}[hbt!]
    \centering
    \includegraphics[scale=0.6]{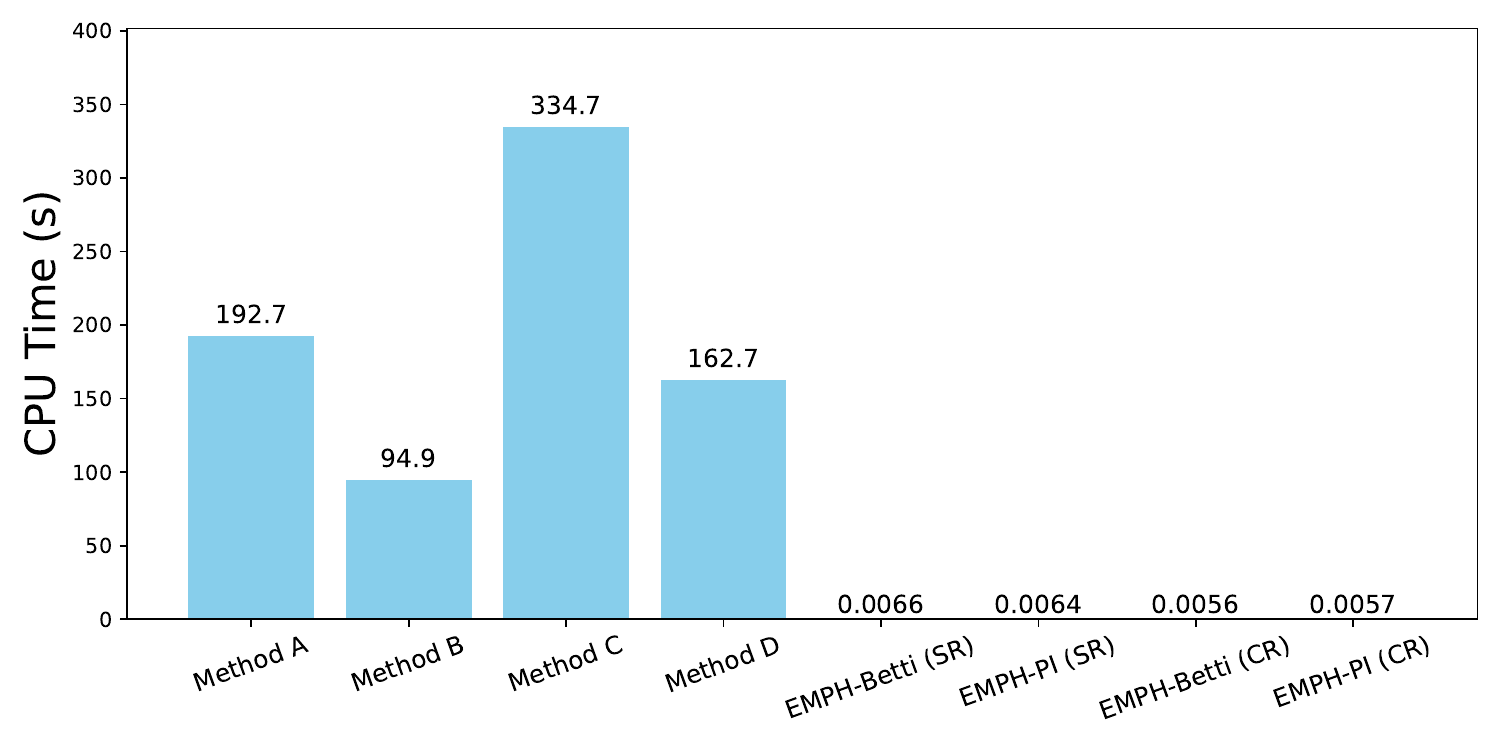}
\caption{{\bf Average CPU time (s)}: The average CPU time per computation for Methods A–D and the EMPH methods under the best hyperparameter settings. The average values were obtained by repeating the computation up to the barcode 10 times.}

    \label{fig:Chung_PI_bar}
\end{figure}

\end{example}
\end{section}

\section{Conclusion} 
\label{conclusion} 

In this work, we proposed the Exact Multi-parameter Persistent Homology (EMPH) method for time-series data analysis.
Assuming, as in Takens’ embedding theorem, that a time-series represents observations of an underlying dynamical system, we modeled this system as a Hamiltonian system of uncoupled one-dimensional harmonic oscillators.
Under this setting, one can consider the Liouville torus as the fundamental object of the Hamiltonian dynamics, and the computation of the persistent homology of the Vietoris–Rips complex constructed on this torus can be interpreted from the viewpoint of Fourier analysis.

EMPH provides a closed-form expression for the fibered barcode, which is an invariant of multi-parameter persistent homology.
Here, the fibered barcode is an invariant obtained by restricting the multi-parameter persistent homology along a specific ray, thereby reducing it to a one-parameter persistent homology.
This expression reveals that the choice of a ray corresponds to the weighting of the Fourier modes, establishing a direct connection between the fibered barcode and the spectral composition.
Varying the ray changes the mode weighting and generates a family of filtrations that reveal different topological features of the data.
This framework enables both efficient computation and variable topological inferences, providing a new bridge between Fourier analysis and multi-parameter persistent homology.

The proposed method is highly advantageous in that the computational cost of the proposed method is very low
as $O(T\log T) + O\left(N \times \binom{N+n-1}{n}\right)$ for $\mathsf{bcd}_{n}^{\Rip,\ell}(\Psi_{f})$. However, the usual one-parameter persistent homology method for time-series data through sliding window embedding and Vietoris-Rips complex on the high dimensional embedding space is highly costly  and in many cases is not usable especially when the size of time-series data is large.  As shown in the numerical results presented in this paper, the proposed EMPH method is comparable or superior in terms of accuracy compared to the one-parameter persistent homology based on sliding window embedding and Vietoris–Rips complex, and it is particularly outstanding in terms of computational cost. Due to its very low computational complexity, it is doable to generate as many ray vectors on the filtration space as desired.

In the paper, we showed that the result with the standard ray vector (the diagonal ray vector) is equivalent to the result by the one-parameter persistent homology of the Liouville torus. By having different ray vectors, one may have different topological inferences. For example, Example \ref{example 2} revealed hidden subclusters in the clustering problem and quantitatively analyzed their differences. In Example \ref{example 3}, the choice of an appropriate ray led to the best performance, while in Example \ref{example 4}, the collection of rays associated with a specific ray achieved the highest classification accuracy. These results indicate that not only the selection of an optimal ray but also the consideration of multiple rays deserves further investigation. Moreover, the proposed method provides a highly efficient way to compute persistent homology on a curved filtration. While calculating persistent homology on a curved filtration is generally difficult or even infeasible, our exact method makes this possible by approximating the curved filtration with line segments.

From the classification of the homotopy type of the Vietoris-Rips complex of a circle, we deduced the exact barcode formula. Our method motivates us to study the homotopy type of the Vietoris-Rips complex of other compact Riemannian manifolds. Meanwhile, we considered a sliding window embedding that translates a sinusoidal function into a circle, and we believe that exploring the basic functions and transformations that map to an $n$-sphere is also meaningful.

The subjects listed below represent potential research agendas aligning with the proposed method, which we will investigate in our future research.

\begin{enumerate}
    \item[$\bullet$] 
    Application of the proposed EMPH method to time-series data from real-world problems and further validation of its effectiveness. 

    \item[$\bullet$] 
    Extension of the proposed method on a curved filtration in the multi-parameter filtration space. 
    \item[$\bullet$] 
    Reduction of the computational complexity of $O\left(N \times \binom{N+n-1}{n}\right)$ by selecting proper frequencies and reducing the embedding dimension of $N$. 
    \sloppy 
    \item[$\bullet$] 
    Extension of the EMPH method to other Hamiltonian systems that are associated with other bases such as spherical harmonics rather than Fourier bases. 
    \item[$\bullet$]
    Filtration learning (e.g. \cite{hofer2020graph, bruel2020topology, leygonie2021framework, carriere2021optimizing})  for the optimal selection of rays or curves on the filtration space with the EMPH. 
\end{enumerate}

\vskip .1in 
{\bf Acknowledgement}
The authors thank Dr. Moo K. Chung for providing the fMRI dataset and his valuable comments and suggestions on the current work. This research was supported by National Research Foundation under the grant number 2021R1A2C3009648 and  also supported partially by  POSTECH Basic Science Research Institute under the grant number 2021R1A6A1A10042944 and the research grant from the NRF to the Center for the Gravitational-Wave Universe under the Grant Number 2021M3F7A1082053. This research was also partially supported by JST Moonshot R\&D under the Grant Number JPMJMS2021.

\titleformat{\section}
  {\normalfont\Large\bfseries}
  {Appendix \thesection.}
  {1em}
  {}
\begin{appendices}
\renewcommand{\thesection}{\Alph{section}}
\section{Notations} 

\begin{enumerate}
\item T : The length of time-series data, p.\pageref{not:length}

 \item $\Rip(X)$ : Vietoris-Rips complex of the metric space $X$, p.\pageref{Vietoris-Rips complex}.

    \item $\bcd_{n}^{\Rip}(X)$ : $n$-dimenisonal barcode of $\Rip(SW_{M,\tau}(f))$, p.\pageref{not:barcode}.

    \item bar : Element of barcode, p.\pageref{not:barcode}.
    
    \item
    $d_{B}$ : Bottleneck distance, p.\pageref{bottleneck distance}.
    
    \item $\mathbb{S}^{1}$ : Unit circle equipped with Euclidean metric, p.\pageref{barcode of circle}.

 \item $\mathbb{T}$ : $\mathbb{R} \big/ 2\pi\mathbb{Z}$, Domain of continuous time-series data, p.\pageref{slidingwindow}.

    \item $SW_{M,\tau}$ : Sliding window embedding, p.\pageref{slidingwindow}.

    \item $N$ : Truncation order of Fourier series. p.\pageref{meaning of N}

    \item $P_L$ : $L$-plane, p.\pageref{roundness}.
    
    \item $r_{L}^{f}$ : $2	\left| \hat{f}(L) \right|$ where $\hat{f}(L)$ is the $L$th Fourier coefficient of $f$, p.\pageref{N-torus}.
    
    \item $\psi_{f,N}, \psi_{f}$ : $\sqrt{{2 \over M+1}} C(SW_{M,\tau}S_N f (t))$, preprocessed point cloud , p.\pageref{N-torus}.

    \item $\Psi_{f,N}, \Psi_f$ : Liouville torus of time-series data $f$, p.\pageref{def:Liouville torus}.
    
    \item $\pi_{i_1 i_2 \cdots i_{N}} : \mathbb{R}^{M+1} \rightarrow P_{i_1} \oplus \cdots \oplus P_{i_{N}}$ : projection map p.$\pageref{projection map}$

    \item $d_{GH}$ : Gromov-Hausdorff distance p.\pageref{Gromov-Hausdorff distance}.

    \item $\Rip^{\ell}_{t} \left(\textcolor{black}{\Psi_{f}}\right)$ : One parameter reduction of multi-parameter persistent homology, p.\pageref{one parameter reduction of multi-parameter persistent homology}.
    
    \item $\mathsf{bcd}_{n}^{\mathcal{R},\ell}(\Psi_{f})$ : Exact Multi-parameter Persistent Homology (EMPH) along $\ell$, p.\pageref{one parameter reduction of multi-parameter persistent homology}.

    \item $\mathsf{bcd}_{n}^{\mathcal{R},\mathcal{L}_{\ell}}(\textcolor{black}{\Psi_{f}})$ : EMPH on the collection of rays $\mathcal{L}_{\ell}$ associated with $\ell$, p.\pageref{thm:collection of ray}.

     \item $\phi_H^t$ : Hamiltonian flow, p.\pageref{def:Hamiltonian flow}.

\end{enumerate}

\section{Elementary sympletic manifold theory}
\label{appendix:symplectic}
In this appendix, we will discuss basic Hamiltonian systems to better understand the meaning of the Liouville torus. For more details and reference, please see \cite{da2008lectures}.
\begin{definition}
A symplectic manifold is a pair $(M, \omega)$ where $M$ is a smooth manifold and $\omega$ is a non-degenerate closed two-form. A symplectic manifold corresponds to a phase space in classical mechanics.    
\end{definition}

\begin{proposition}[Darboux]
For a symplectic manifold $(M,\omega)$, there is a local chart $(q_1, \cdots, q_n, p_1, \cdots, p_n)$ such that 
$$\omega = \sum\limits_{i=1}^{n} dq_i \wedge dp_i.$$     
This local chart is called the canonical coordinates on $(M,\omega)$. This proposition tells us that every symplectic manifold is locally isomorphic.
\end{proposition}

\begin{definition}[Hamiltonian vector field]
For a smooth function (usually referred to as the Hamiltonian) $H : M \rightarrow \mathbb{R}$, the Hamiltonian vector field $X_H$ is defined by the equation $dH = X_H \righthalfcup \omega.$ The well-definedness of $X_H$ follows from $\omega$ being non-degenerate. Recall that the interior multiplication is defined by $v \righthalfcup \omega = \omega(v, \cdot)$.
\end{definition}

\begin{definition}[Hamiltonian flow]
\label{def:Hamiltonian flow}
We call the flow of the Hamiltonian vector field $X_H$ the Hamiltonian flow, denoted by $\phi_H^t$. That is,  $\phi_H^t$ satisfies
$$
\begin{cases}
\phi_H^0 = id_M \\
{d\phi_H^t \over dt} = X_H \circ \phi_H^t
\end{cases}.$$
\end{definition}

\begin{remark} Note that if $H$ has the compact support, then $\phi_H^t$ is defined for every $t \in \mathbb{R}$. cf) Theorem 9.16. in \cite{lee2012smooth}.
    
\end{remark}

\begin{definition}[Poisson bracket]
For $f,g \in C^{\infty}(M,\mathbb{R})$,
$$\left\{ f,g \right\} := \omega(X_f, X_g)$$
is called Poisson bracket of $f$ and $g$.
\end{definition}

\begin{theorem}
\label{thm:conservity}
$\left\{ f,H \right\} = 0$ if and only if $f$ is constant along integral curves of $X_H$. More precisely, ${d \over dt}(f \circ \phi_H^t) = \left\{f,H \right\} \circ \phi_H^t = 0$.    
\end{theorem}
    
\begin{example}[Hamiltonian equations]
For canonical coordinates on $(M,\omega)$, the Hamilton equations are given by
$$\begin{cases}
\dot{q_{i}}={\partial H \over \partial p_{i}} \\ \dot{p_{i}}=-{\partial H \over \partial q_{i}}    
\end{cases}$$
\begin{proof}
Let $X_{H} = \sum\limits_{i} {a_{i} {\partial \over \partial q_{i}} + b_{i} {\partial \over \partial p_{i}}}$. 

\begin{equation*}
X_{H} \righthalfcup \omega = \sum\limits_{i} X_{H} \righthalfcup (dq_{i} \wedge dp_{i}) = \sum _{i} (X_{H} \righthalfcup dq_{i}) \wedge dp_{i} - dq_{i} \wedge (X_{H} \righthalfcup dp_{i}) = \sum _{i} {a_{i} dp_{i} - b_{i} dq_{i}} \end{equation*}
\begin{equation*}
dH = \sum _{i}  {{\partial H \over \partial q_{i}} dq_{i} + {\partial H \over \partial p_{i}}dp_{i}}.
\end{equation*} Therefore $a_i = {\partial H \over \partial p_{i}}$ and $b_i = -{\partial H \over \partial q_{i}}$ imply $X_H = \sum\limits_{i}{{\partial H \over \partial p_{i}} {\partial \over \partial q_{i}} -{\partial H \over \partial q_{i}} {\partial \over \partial p_{i}}}$.
\end{proof}
\end{example}

\section{Elementary calculation of one-dimensional barcode of Vietoris-Rips complex of $\mathbb{S}^{1}$}
In \cite{adamaszek2017vietoris}  barcode formula of the Vietoris-Rips complex of a unit circle equipped with Euclidean metric $\mathbb{S}^{1}$ was suggeated. As mentioned in Section \ref{sec:Defintions}, cyclic graph $\vec{G}$ and its invariant winding fraction $wf(\vec{G})$ are used. As we saw in Theorem \ref{barcode of circle}, the Vietoris-Rips complex gives us redundant homology, that is, even if one-dimensional manifold $\mathbb{S}^1$, we can capture higher dimensional homology via persistent homology. Even though we give only a one-dimensional barcode formula, our proof is  elementary and good enough from the manifold inference perspective.
\begin{lemma}
\label{hexagon}
The birth time of the barcode of the Vietoris-Rips complex composed of vertices of regular hexagon is the length of the side of regular hexagon and the death time is the length of the shortest diagonal line.
\end{lemma}

\begin{proof}
In Figure \ref{hexagon_picture}, set $k=1$ to help to prove this theorem. Clearly, the birth time is the length of the side of regular hexagon and corresponding cycle is given by $[0,1]+[1,2]+[2,3]+[3,4]+[4,5]+[5,0]$. And this cycle is alive up to the length of the shortest diagonal line of hexagon since there is no $2$-simplex. For the length of the shortest diagonal line, $[0,1]+[1,2]+[2,3]+[3,4]+[4,5]+[5,0]$ is the boundary of $[0,1,2]+[2,3,4]+[4,5,0]+[0,2,4]$.
\end{proof}

\begin{theorem}
\label{6k barcode}
Let $P_T$ be a regular $T$-polygon with $T=6k$ $(k\in\mathbb{Z}_{>0})$, whose side length is $1$. Then the $0$- and $1$-dimensional barcodes of the Vietoris--Rips complex of $P_T$ are given by
\begin{eqnarray} \mathsf{bcd}_{0}^{\Rip}\left( P_T \right) &=& \left\{(0, \infty), \left(0, 2\sin {\pi \over T}\right]_{(T-1)} \right\}\nonumber \\ \mathsf{bcd}_{1}^{\Rip}\left( P_T \right) &=& \left\{\left(2\sin {\pi \over T}, \sqrt{3} \right] \right\}. \nonumber \end{eqnarray}
\end{theorem}

\begin{proof}
Clearly, $[0,1]+[1,2]+\cdots+[6k-1,0]$ is a $1$-cycle and $0$-boundary. So the birth time of this cycle and the death time of $0$-dim cycle are equal to the length of the side of $P_T$ ($= 2\sin {\pi \over T}$).
\begin{figure}[h]
    \centering
    \includegraphics[width=0.25\textwidth]{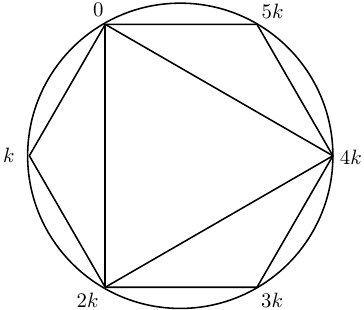}
    \caption{Hexagon used in the proof}
    \label{hexagon_picture}
\end{figure}

To calculate the death time of one-dimensional barcode, we use  Lemma \ref{hexagon}. Consider the time parameter bigger than the length of $[0,k]$. We can easily check that $$\partial ([0,1,k]+[1,2,k]+\cdots+[k-2,k-1,k]) = [0,1]+[1,2]+\cdots+[k-1,k]-[0,k].$$
This implies that 
$$\partial S  =  [0,1]+[1,2]+\cdots+[6k-1,6k] - ([0,k] + [k,2k] + \cdots + [5k,0])$$
where $S = [0,1,k]+[1,2,k]+\cdots+[k-2,k-1,k] + [k,k+1,2k] + [k+1,k+2,2k]+\cdots+[2k-2,2k-1,k] + \cdots $$
$$+ [6k-2,6k-1,6k]$.
Note that the cycle $[0,1]+[1,2]+\cdots+[6k-1,6k]$ is a boundary if and only if $[0,k] + [k,2k] + \cdots + [5k,0]$ is a boundary. 
In  Lemma \ref{hexagon}, We already proved that the cycle $[0,k] + [k,2k] + \cdots + [5k,0]$ is not a boundary (i.e. the death) until the time becomes the length of $[0,2k]$. Therefore the death time is $\sqrt{3}$.
\end{proof}

\begin{corollary}
\label{circle complete}
$\mathsf{bcd}_{0}^{\Rip}(\mathbb{S}^{1}) = \left\{\left(0, \infty \right)\right\}$ and $\mathsf{bcd}_{1}^{\Rip}\left(\mathbb{S}^{1}\right) = \left\{\left(0, \sqrt{3} \right] \right\}$.
\end{corollary}

\begin{proof}
By Proposition \ref{app:stability theorem}, $d_B(\mathsf{bcd}_{0}^{\Rip}(P_T), \mathsf{bcd}_{0}^{\Rip}(\mathbb{S}^{1})) \le 2 \cdot d_{GH}	\left(P_T, \mathbb{S}^{1}\right)$ 
and $2 \cdot d_{GH}	\left(P_T, \mathbb{S}^{1}\right)\rightarrow 0$ as $T \rightarrow \infty$. The one-dimensional case can be proved in a similar way.
\end{proof}

\section{Direction Vectors and Endpoints Used in Example \ref{example 4}}
\label{appendix D}

\begin{eqnarray}
{\mathbf{a}_1}  &=&  (1.7, 1.8, 1.6, 1.9, 1.9, 1.3, 2.0, 1.6, 2.0, 1.0, 1.4, 1.4, 1.9, 1.5, 1.9, 2.0, 1.4, 1.4, 1.0, 1.8,  \nonumber \\ && 1.5, 1.5, 1.1, 1.5, 1.2, 1.8, 1.0, 1.6, 1.3, 1.7, 1.1, 1.2, 1.7, 1.5, 1.3, 1.8, 1.5, 1.2, 1.9, 1.8,  \nonumber \\ && 2.0, 1.5, 1.1, 1.1, 1.1, 1.6, 1.2, 2.0, 1.4, 1.8, 1.5, 1.1, 1.6, 1.8, 1.2, 1.4, 1.6, 1.7, 1.6, 1.3, 
\nonumber \\ && 1.6, 1.6, 1.8, 1.2, 1.1, 1.8, 1.7, 1.8, 1.0, 1.3, 1.8, 1.9, 1.8, 1.5, 1.3, 1.3, 1.1, 1.4, 1.2, 1.6, 
\nonumber \\ && 1.1, 1.6, 1.5, 1.2, 1.2, 1.4, 1.4, 1.5, 1.6, 1.4, 1.4, 1.1, 1.8, 1.3, 1.3, 1.5, 1.7, 1.3, 1.4, 1.5) \nonumber \\
{\mathbf{b}_1} & = & (4.7, 0.0, 1.1, 1.1, 2.0, 0.5, 3.4, 3.3, 0.6, 0.6, 0.6, 1.5, 2.8, 1.3, 4.7, 2.7, 4.7, 0.3, 0.7, 4.2, \nonumber \\ && 2.3, 1.4, 4.4, 3.9, 4.5, 0.2, 0.7, 2.7, 2.0, 1.1, 4.3, 0.3, 1.2, 1.7, 4.9, 3.8, 0.4, 1.8, 4.8, 2.3, \nonumber \\ && 4.4, 4.3, 4.7, 4.6, 2.3, 1.9, 1.5, 0.8, 0.8, 0.1, 1.3, 2.9, 2.8, 0.9, 0.8, 4.8, 4.1, 3.8, 2.8, 3.6, 
\nonumber \\ && 0.5, 2.1, 3.4, 2.0, 0.6, 2.8, 2.4, 3.5, 1.4, 0.7, 0.8, 3.3, 2.1, 3.7, 2.6, 2.8, 1.4, 0.4, 0.2, 3.3,
\nonumber \\ && 0.4, 4.6, 2.5, 3.5, 2.5, 2.5, 1.9, 1.6, 2.2, 0.3, 0.8, 1.7, 1.2, 0.4, 1.7, 4.4, 3.1, 2.3, 0.7, 4.0) \nonumber \\
{\mathbf{a}_2}  &=&  (1.7, 1.8, 1.6, 1.9, 1.9, 1.3, 2.0, 1.6, 2.0, 1.0, 1.4, 1.4, 1.9, 1.5, 1.9, 2.0, 1.4, 1.4, 1.0, 1.8,  \nonumber \\ && 1.5, 1.5, 1.1, 1.5, 1.2, 1.8, 1.0, 1.6, 1.3, 1.7, 1.1, 1.2, 1.7, 1.5, 1.3, 1.8, 1.5, 1.2, 1.9, 1.8, \nonumber \\ && 2.0, 1.5, 1.1, 1.1, 1.1, 1.6, 1.2, 2.0, 1.4, 1.8, 1.5, 1.1, 1.6, 1.8, 1.2, 1.4, 1.6, 1.7, 1.6, 1.3, 
\nonumber \\ && 1.6, 1.6, 1.8, 1.2, 1.1, 1.8, 1.7, 1.8, 1.0, 1.3, 1.8, 1.9, 1.8, 1.5, 1.3, 1.3, 1.1, 1.4, 1.2, 1.6, 
\nonumber \\ && 1.1, 1.6, 1.5, 1.2, 1.2, 1.4, 1.4, 1.5, 1.6, 1.4, 1.4, 1.1, 1.8, 1.3, 1.3, 1.5, 1.7, 1.3, 1.4, 1.5) \nonumber
\end{eqnarray}

\begin{eqnarray}
{\mathbf{b}_2} & = & (4.7, 0.0, 1.1, 1.1, 2.0, 0.5, 3.4, 3.3, 0.6, 0.6, 0.6, 1.5, 2.8, 1.3, 4.7, 2.7, 4.7, 0.3, 0.7, 4.2, 
\nonumber \\ && 2.3, 1.4, 4.4, 3.9, 4.5, 0.2, 0.7, 2.7, 2.0, 1.1, 4.3, 0.3, 1.2, 1.7, 4.9, 3.8, 0.4, 1.8, 4.8, 2.3, \nonumber \\ && 4.4, 4.3, 4.7, 4.6, 2.3, 1.9, 1.5, 0.8, 0.8, 0.1, 1.3, 2.9, 2.8, 0.9, 0.8, 4.8, 4.1, 3.8, 2.8, 3.6, 
\nonumber \\ && 0.5, 2.1, 3.4, 2.0, 0.6, 2.8, 2.4, 3.5, 1.4, 0.7, 0.8, 3.3, 2.1, 3.7, 2.6, 2.8, 1.4, 0.4, 0.2, 3.3,
\nonumber \\ && 0.4, 4.6, 2.5, 3.5, 2.5, 2.5, 1.9, 1.6, 2.2, 0.3, 0.8, 1.7, 1.2, 0.4, 1.7, 4.4, 3.1, 2.3, 0.7, 4.0) \nonumber \\
{\mathbf{a}_3}  &=&  (2.0, 1.2, 1.5, 1.9, 1.6, 1.6, 1.0, 1.7, 1.1, 1.9, 1.6, 1.1, 1.1, 1.7, 1.2, 2.0, 1.6, 1.7, 1.0, 1.8,  \nonumber \\ && 1.7, 1.5, 1.9, 1.8, 1.0, 1.1, 1.7, 1.4, 1.7, 1.9, 1.5, 1.9, 1.3, 1.7, 1.8, 1.7, 1.8, 1.7, 1.8, 1.1, 
\nonumber \\ && 1.7, 1.2, 1.4, 2.0, 1.9, 1.8, 1.6, 1.9, 1.3, 1.5, 1.2, 1.2, 1.8, 1.7, 1.9, 1.5, 1.1, 1.8, 1.3, 1.8, 
\nonumber \\ && 1.8, 1.2, 1.9, 1.4, 1.0, 1.2, 1.4, 1.0, 1.7, 2.0, 1.4, 1.0, 1.5, 1.7, 1.9, 1.9, 1.8, 1.5, 1.1, 2.0,
\nonumber \\ && 1.2, 1.4, 1.5, 1.1, 1.2, 1.7, 1.1, 1.1, 1.3, 1.8, 1.6, 1.7, 1.9, 1.9, 1.2, 1.8, 1.8, 1.5, 1.1, 1.4) \nonumber \\
{\mathbf{b}_3} & = & (3.1, 0.2, 2.3, 0.3, 4.9, 0.2, 1.2, 0.2, 4.2, 5.0, 2.4, 2.4, 2.0, 2.3, 4.6, 4.2, 3.3, 1.3, 1.0, 2.9, \nonumber \\ && 3.4, 3.9, 4.3, 0.5, 2.2, 0.8, 2.0, 2.0, 2.5, 3.7, 4.5, 0.4, 4.4, 4.8, 0.6, 4.3, 4.3, 4.1, 3.7, 3.6,  \nonumber \\ &&2.8, 3.4, 3.5, 0.2, 3.8, 2.2, 3.5, 3.9, 4.0, 4.1, 0.0, 2.0, 3.4, 4.0, 0.3, 3.1, 4.2, 4.3, 1.7, 4.6, 
\nonumber \\ && 2.1, 2.8, 4.3, 3.0, 1.5, 3.9, 1.0, 4.3, 3.4, 3.2, 0.2, 0.1, 4.8, 4.6, 3.9, 3.2, 4.4, 4.5, 0.3, 1.8, 
\nonumber \\ && 1.0, 0.4, 3.7, 3.0, 2.8, 2.2, 0.5, 0.2, 2.3, 4.5, 4.7, 2.0, 0.6, 4.7, 3.1, 5.0, 1.1, 3.4, 2.5, 3.7) \nonumber \\
{\mathbf{a}_4}  &=&  (1.4, 1.9, 1.5, 1.8, 1.4, 1.6, 1.4, 1.3, 1.6, 1.7, 1.9, 1.2, 1.7, 1.7, 1.9, 1.2, 1.9, 1.1, 1.4, 1.9, 
\nonumber \\ && 1.5, 1.3, 1.5, 1.5, 1.6, 1.5, 1.9, 1.4, 1.3, 1.9, 1.9, 1.5, 1.8, 1.2, 1.5, 1.8, 1.8, 1.1, 1.6, 1.5, 
\nonumber \\ && 1.5, 1.7, 1.3, 1.1, 1.9, 1.6, 2.0, 1.0, 1.4, 1.4, 1.9, 1.3, 1.5, 1.8, 1.3, 1.5, 1.4, 1.5, 1.9, 1.3, 
\nonumber \\ && 1.6, 2.0, 1.7, 1.9, 1.2, 1.7, 1.5, 1.4, 1.1, 1.2, 1.1, 1.8, 1.9, 1.6, 1.2, 1.0, 1.6, 1.0, 1.6, 1.3, 
\nonumber \\ && 1.3, 1.4, 1.9, 1.2, 1.0, 1.9, 1.3, 1.9, 1.1, 1.4, 1.7, 1.6, 1.2, 1.3, 1.6, 2.0, 1.9, 1.4, 1.1, 1.5)
\nonumber \\
{\mathbf{b}_4} & = & (1.2, 2.5, 2.7, 2.6, 1.4, 0.8, 1.5, 3.7, 4.0, 0.7, 3.7, 3.3, 0.7, 3.8, 0.4, 0.9, 1.5, 1.3, 0.2, 1.5, 
\nonumber \\ && 2.4, 1.2, 2.8, 3.4, 2.8, 1.2, 2.4, 0.6, 3.2, 0.5, 2.7, 3.1, 2.7, 4.2, 4.0, 0.3, 2.2, 2.5, 2.5, 3.5, 
\nonumber \\ && 1.9, 4.6, 4.2, 3.9, 3.2, 1.7, 4.5, 1.0, 3.7, 2.5, 1.8, 4.9, 1.8, 3.7, 5.0, 4.7, 4.3, 0.4, 2.2, 2.9, 
\nonumber \\ && 1.1, 1.7, 0.6, 4.2, 4.0, 3.5, 1.3, 4.9, 3.4, 0.6, 1.7, 2.1, 3.8, 5.0, 5.0, 4.1, 1.5, 4.4, 2.9, 3.2, 
\nonumber \\ && 3.3, 4.8, 1.1, 3.0, 4.3, 2.9, 0.7, 4.8, 4.8, 4.2, 4.4, 2.6, 2.6, 2.2, 2.3, 2.0, 3.0, 0.5, 1.7, 0.7)
\nonumber 
\end{eqnarray}

\end{appendices}

\vskip .1in
\normalem
\bibliographystyle{abbrv} 
\nocite*{}
\bibliography{reference}
\end{document}